\documentclass[a4paper,12pt]{article}
\usepackage{amssymb,amsmath,amsthm}
\usepackage{graphicx}

\usepackage{amsfonts}
\usepackage{latexsym}    
\usepackage{color} 
\usepackage[english]{babel} 
\usepackage{hyperref}
\makeatletter
\renewcommand*{\toclevel@part}{0}
\makeatother
\usepackage{hyperref}
\usepackage{verbatim} 
\usepackage{stmaryrd} 
\usepackage{bm} 
\usepackage{xcolor}
\usepackage{graphicx} 
\usepackage{subcaption}
\usepackage[font=footnotesize,labelfont=bf]{caption}
\usepackage{mathrsfs} 
\usepackage{mathtools}

\definecolor{darkgr}{rgb}{0.0, 0.62, 0.42}

\numberwithin{equation}{section}

\newcommand{\la}{\langle}
\newcommand{\ra}{\rangle}

\newtheorem{lemma}{Lemma}[section]
\newtheorem{theorem}[lemma]{Theorem}
\newtheorem{proposition}[lemma]{Proposition}
\newtheorem{corollary}[lemma]{Corollary}
\newtheorem{remark}[lemma]{Remark}

\newtheorem{definition}[lemma]{Definition}
\newtheorem{defn}[lemma]{Definition}

\def\mass{\mathcal M \kern -2pt a \kern-2pt d}

\def\tr{{\tt r}}

\def\im{{\rm i}}

\def\tf{{\tt f}}
\def\tB{{\tt B}}
\def\uno{{\mathds 1}}
\def\norm#1{\left\|#1\right\|}
\def\grad{\nabla}
\def\mathds#1{{\bf #1}}
\def\cross{\times}

\def\cV{{\mathcal{V}}}
\def\bb{{\bf b}}

\def\cH{{\mathcal{H}}}

\def\cO{{\mathcal{O}}}

\def\cU{{\mathcal{U}}}
\def\cR{{\mathcal{R}}}

\def\cP{{\mathcal{P}}}
\def\cT{{\mathcal{T}}}

\def\cB{{\mathcal{B}}}

\def\bA{{\bf A}}
\def\Ph#1{\cH_e^{#1}}

\newcommand{\R}{\mathbb R}
\newcommand{\C}{\mathbb C}
\newcommand{\Z}{\mathbb Z}
\newcommand{\N}{\mathbb N}
\newcommand{\T}{\mathbb T}

\newcommand{\ii}{i}

\def\ac{{\mathcal A}\kern-.7pt\ell\kern-.9pt\mathcal{S}}

\renewcommand{\H}{\mathcal{H}}

\newcommand{\B}{\mathcal{B}}

\renewcommand{\r}{\mathbb{R}}
\newcommand{\n}{\mathbb{N}}

\newcommand{\br}{\bar{r}}
\def\len#1{\langle\left|#1\right|\rangle}

\renewcommand{\S}[2]{ S(A_{#1},\dots, A_{#2})}
\def\Sb#1#2{S(A_{#1},\dots,A_{#2},b)}
\newcommand{\mus}[2]{\mu(A_{#1},\dots,A_{#2})}
\newcommand{\musb}[2]{\mu(A_{#1},\dots,A_{#2},b)}

\newtheorem{manualtheoreminner}{Hypothesis}
\newenvironment{manualhypo}[1]{%
	\renewcommand\themanualtheoreminner{#1}%
	\manualtheoreminner
}{\endmanualtheoreminner}

\begin{document}
	\title{Almost global existence for some Hamiltonian
		PDEs on manifolds with globally integrable geodesic flow}
	\date{}
	
	\author{
		D. Bambusi\footnote{Dipartimento di Matematica, Universit\`a degli Studi di Milano, Via Saldini 50, I-20133
			Milano. \newline
			\textit{Email: } \texttt{dario.bambusi@unimi.it}}, 
		R. Feola\footnote{{Dipartimento di Matematica e Fisica, Universit\`a degli Studi RomaTre, Largo San Leonardo Murialdo 1, 00144 Roma}\newline
			\textit{Email:} \texttt{roberto.feola@uniroma3.it}}, 
		B. Langella\footnote{International School for Advanced Studies (SISSA), via Bonomea 265, I-34136 Trieste.\newline
			\textit{Email: } \texttt{beatrice.langella@sissa.it}},
		F. Monzani \footnote{Dipartimento di Matematica, Universit\`a degli Studi di Milano, Via Saldini 50, I-20133
			Milano. \newline
			\textit{Email: } \texttt{francesco.monzani@unimi.it}}
	}
	
	\maketitle
	
	\begin{abstract}
		In this paper we prove an abstract result of almost
		global existence for small and smooth solutions of
		some semilinear PDEs on Riemannian manifolds with
		globally integrable geodesic flow.  Some examples of
		such manifolds are  Lie groups (including flat tori), homogeneous
		spaces and rotational invariant surfaces. As
		applications of the abstract result we prove
		almost global existence for a nonlinear Schr\"odinger
		equation with a convolution potential and for a
		nonlinear beam equation. We also prove $H^s$ stability
		of the ground state in NLS equation.  The proof is
		based on a normal form procedure and the combination of the arguments used in
		\cite{QN} to
		bound the growth of Sobolev norms in linear systems
		and a generalization of the arguments in \cite{BFM22}.
	\end{abstract}
	\newpage
	\tableofcontents
	
	\section{Introduction}
	Given a Hamiltonian PDE on a compact manifold, a
	classical problem is the determination of the time of existence of
	its solutions and the possible existence of solutions
	enforcing transfer of energy from low to high frequency
	modes. Here we deal with the problem of proving that, given an
	initial datum with $s$--Sobolev norm of size $\epsilon\ll1$,
	 the corresponding solution exists for times
	of order $\epsilon^{-r}$, $\forall r$, provided that $s$ is large enough. Furthermore its $s$--Sobolev norm remains of order $\epsilon$ over this time
	scale. Such results are usually known as results of ``almost
	global existence''.
	
	
	It is known since almost 20 years
        \cite{bambusi2003birkhoff,BG06} how to prove almost global
        existence in semilinear Hamiltonian PDEs in \emph{space
          dimension $1$} (see also \cite{GreLan, BamMon, FaouGre, BMPro,
          FeoMass}), and  recently also the theory of
        quasilinear Hamiltonian PDEs in one space dimension has
        achieved a satisfactory form
        \cite{DelortS1,BD,berti2022hamiltonian} (see also
        \cite{FeoIanPisa}).  On the other hand, in space dimension
        larger than $1$, it is only known how to deal with equations on
        Zoll Manifolds \cite{BDGS,delort.sphere.ql} (see also
        \cite{BGCras, GrePT, plane} for other problems with the same
        structure) and recently the case of general tori has been
        treated \cite{BFM22}. Concerning more general  manifolds, one can
find many partial results in which existence and smoothness of small
solutions are proven for times of order $\epsilon^{-r}$ with
\emph{some $r>1$}, typically smaller or equal than $2$ (see
\cite{FANG2010151, Imekraz_beam, DI17, ionescu, beam-feola-bernier, qlifespan.bert.feol.franz,feola-mont, feola-murg-hydro,BerFeoPus,feola-greb-iandoli}).

In the present paper we investigate the possibility of extending the
result of \cite{BFM22} to general manifolds: in particular we aim at
stepping away from tori and Zoll manifolds. As a result, we prove almost global existence for some equations
on a class of manifolds in which the Laplacian is a ``globally
integrable quantum system'' \cite{QN} (see
	Sect.\ref{almost.glob.sec} for a precise
        definition). Some concrete examples of these
manifolds are flat tori, compact Lie groups, homogeneous spaces,
rotation invariant surfaces and also products of these kinds of
manifolds. Indeed, we prove an abstract theorem of almost global
existence for semilinear PDEs on these manifolds and subsequently we
apply it to some concrete models. Precisely, we prove almost
global existence for NLS with a ``potential'' which is a spectral
multiplier and for the beam equation. We also apply our abstract
theorem to the problem of $H^s$ stability of the ground state in NLS
and prove that it is stable for times of order $\epsilon^{-r}$
$\forall r$. We emphasize that all these results were only known for
equations on tori \cite{BFM22}. As explained below, in the present
paper we combine some of the key ideas of \cite{BFM22} and of \cite{DS,QN} with some other ingredients.

	A related line of research, parallel to the one pursued here, is
	the so-called KAM theory for PDEs, which aims to construct
	quasiperiodic solutions of the considered equations. As for almost global existence, KAM theory
	for PDEs is nowadays quite well understood, for both semilinear and quasilinear equations, in space dimension 1. Without trying to be exhaustive, we mention, among others, the seminal works \cite{Kuk1,Kuk2, wayne, poschel1996,BBM1,BBM2} for a first reading in this wide field. 
	Concerning the higher dimensional domains, which are our main focus, only the cases of tori
	\cite{bourgain2004green, EliKuk, geng.kam, BBolleOnde,procesiKAM, EKG} and compact Lie
	groups or homogeneous space \cite{BP,BCP} have been completely treated. We
	think that techniques similar to the ones developed in this work, combined with the so-called
	multiscale analysis \cite{bourgain2004green}, could be used to prove KAM-type
	results for rotation invariant manifolds. This is left for
	future developments.
	
	Finally, we recall the results in \cite{bou.almost,bou05, pro.almost, almost.mass}. In these works some invariant
	Lagrangian tori filled by almost periodic solutions have been
	constructed for some semilinear PDEs. This has been done in some model
	problems in dimension 1, but at present nothing is known for equations
	in higher dimensions.

	\vspace{0.5em}
	\noindent
	{\bf Scheme of the proof.}  As already mentioned, the starting
        point of the present paper is \cite{BFM22}. So we first recall the strategy of
        \cite{BFM22}. As usual, the idea is to look for a canonical
        transformation that puts the system in a normal form which is
        a variant of the classical Birkhoff normal form. The idea of
        \cite{BFM22} is to split the Fourier modes into low and high
        frequency modes, to work in the standard way on the low modes
        and, concerning the high modes to exploit a cluster
        decomposition of $\Z^d$ (indexes of the Fourier modes)
        introduced by Bourgain. Such a decomposition has the property
        that if $k,l\in\Z^d$ belong to different clusters then
	$$
	| |k|^2-|l|^2|+|k-l|\geq C(|k|^\delta+|l|^\delta)\ ,
	$$ for some $C>0$ and $\delta\in(0,1)$. This property allows
        to construct a normal form which ensures that there is no
        exchange of energy among different clusters of modes of high
        frequency.
	
	Altogether with a nonresonance assumption on the frequencies
        which is typically fulfilled in any space dimension, this
        allows us to show also that the low modes do not exchange
        energy among them and with the high modes and to conclude that
        one can control the Sobolev norm of the solution up to the
        considered times.
	
	\smallskip
	The main difficulty in generalizing the above idea to general manifolds is
	that in general Fourier series are not available: one has to identify
	some basis of $L^2(M)$ which has properties similar to the Fourier basis.
	
	Such a basis was identified in \cite{QN} for some manifolds
	in which the Laplacian is a ``globally integrable quantum system'', 
	namely it is a function of some commuting operators which play the
	role of quantum actions. In turn the quantum actions $I_j$ are
        operators whose spectrum is of the form
        $\left\{n+\kappa\right\}_{n\in\Z}$ with some 
 $\kappa\in(0,1)$
        (see
	Sect.\ref{almost.glob.sec} for a more precise
        definition). Thus the spectrum of  the
	Laplacian turns out to be given by 
	$$
	\sigma(-\Delta_g)=\left\{h_L(a^1,...,a^d)\ |\ (a^1,...,a^d)\in\Lambda\subset\Z^d\right\}\ , 
	$$ with $h_L$ a function. Correspondingly, the eigenfunctions of
        the Laplacian are labelled by discrete indexes
        $(a^1,...,a^d)\in\Lambda\subset\Z^d$. Such eigenfunctions are
        the objects that substitute the exponentials of the torus.
	
	\smallskip
	
	Once the basis is identified, one has to study its properties.
        It is known that the main property needed to develop
        perturbation theory is a bilinear estimate involving the
        expansion of the product of two eigenfunctions on the basis of
        the eigenfunctions themselves. The main idea is that we do not
        require such an expansion to {\it decay with the difference of
          the corresponding eigenvalues, but with the distance of the
          indexes labelling them!} Once one realizes this fact, it is
        easy to obtain the wanted estimate by working as in
        \cite{DS,BamMon}.
	
	After proving the estimate of the product of eigenfunctions, we have to
	study the composition properties of the polynomials obtained by
	expanding the nonlinearity on the eigenfunctions. This is very similar
	to what is done in \cite{DS,BamMon}. However, these computations are technically
	quite heavy in our setting, since we are dealing with a vectorial
	labelling of eigenfunctions instead of a scalar one. All the
	needed estimates are performed in Appendix \ref{esti.appendix}.
	
	After this is done one still has to verify a nonresonance
        condition which is more or less standard:  this is done following
        the ideas of \cite{DS0}. Finally one concludes the proof as in
        \cite{BFM22}.

	\smallskip
	Before closing this section, we remark that at present our method is
	restricted to systems whose linear parts have frequencies
	$\left\{\omega_a\right\}_{a\in\Lambda}$, with $\omega_a\sim |a|^\beta$
	and $\beta>1$. Thus, the case of the wave equation is not covered. 
	Moreover, our techniques apply only to the case of semilinear equations, 
	namely, perturbations involving derivatives are not allowed by our theory. We
	think that this second issue can be overcome by developing some 
	tools of the kind of those used in \cite{BD,QN}, while the
        first requires new ideas. \\ \\
	
	\noindent
	{\bf Acknowledgements.} 
	D. Bambusi would like to warmly thank Jean-Marc Delort for explaining to him in detail the proof of Lemma \ref{estiAdjoint} and for some discussions on the nonresonance condition. We also thank Massimiliano Berti and Michela Procesi for some discussions on the presentation of the results. \\
	This work is supported by INDAM. D. Bambusi, R. Feola, B. Langella and F. Monzani have been supported by the research projects PRIN 2020XBFL 
	``Hamiltonian and dispersive PDEs'' of the Italian Ministry of Education and Research (MIUR). R. Feola and B. Langella have also been supported by PRIN 2022HSSYPN ``Turbulent Effects vs Stability in Equations from Oceanography''.

	\section{Statement of the Abstract Theorem}\label{almost.glob.sec}
	
	Let $(M,g)$ be a compact Riemannian manifold of dimension $n$, denote
	by $\Psi^m(M)$ the space of pseudodifferential operators \emph{à la}
	H\"ormander of order $m$ on $M$ (see
	\cite{Hormander}).
	
	\noindent
	Let $H^{s}(M)=H^s(M;\C)$, $s\geq0$ be the standard Sobolev spaces on $M$. We shall study equations of the form
	\begin{equation}
		\label{HL.0}
		\im \dot u=H_L u+\nabla_{\bar u}P(u,\bar u)\ ,
	\end{equation}
	where $H_L$ is a linear selfadjoint operator on which we are going to
	make several assumptions and $P$ is a nonlinear functional that we
	describe below. Here $\nabla_{\bar u}$ is the gradient of $P$ with
	respect to the variable $\bar u$ and the $L^2$ metric. 
	
	\smallskip
	\noindent
	We remark that the system \eqref{HL.0} is Hamiltonian with Hamiltonian
	function
	\[
	H=H_0+P\,,
	\]
	where
	\begin{equation}
		\label{H0.s}
		H_0(u):=\int_M \bar u H_L u \,dx\ .
	\end{equation}
	
	\subsection{Assumptions on the linear system 
	}\label{linear} We are going to assume that $H_L$ is a {\it
		globally integrable quantum system}, as introduced in
	\cite{QN}. Roughly speaking, it is a linear operator which is
	a function of some first-order pseudodifferential operators
	that we will call \emph{quantum actions}. The precise
	definitions are recalled below.
	
	\begin{definition}
		\label{glob.int.quant}[System of Quantum Actions]
		Let $\left\{I_j\right\}_{j=1,\dots,d}$ be $d$ selfadjoint pseudo-differential operators of order 1, fulfilling
		\begin{itemize}
			\item[i.]$I_j \in \Psi^1(M),$ for any $ j=1,\dots,d$; 
			\item[ii.] $[I_i,I_j] = 0,$ for any $ i,j = 1,\dots,d$;
			\item[iii.] there exists a constant $ c_1 > 0$ such that $c_1 \sqrt{1-\Delta_g} \le \sqrt{1+\sum_{j=1}^{d} I_j^2}$\label{sobolev};
			\item[iv.]  there exists 
			$\kappa\in(0,1)^d$ such that the joint spectrum $\Lambda$ of
			the operators $I_1,...,I_d$ fulfills
			\begin{gather}
				\Lambda \subset \mathbb{Z}^d +
				\kappa\ .
			\end{gather}
		\end{itemize}
		we refer to $(I_1,...,I_d)$ as the \emph{quantum actions}.
	\end{definition}
	
	\begin{remark}
		\label{2.1}
		{	By {iii.}, the operator $\uno+\sum_{j=0}^d I_j^2$ has compact inverse,
			therefore the spectrum $\sigma(I_j)$ of each one of the $I_j$'s is
			pure point and formed by a sequence of eigenvalues.
		}
	\end{remark}
	
	\begin{remark}
		\label{joint}
		{We recall that the joint spectrum $\Lambda$ of the operators $I_j$ is
			defined as the set of the ${a = (a^1, \dots, a^d)}\in\R^d$ s.t. there
			exists $\psi_a\in\cH$ with $\psi_a\not=0$ and
			\begin{equation}
				\label{lambda}
				I_j\psi_a=a^j\psi_a
				\ ,\quad \forall j=1,...,d\ .
			\end{equation}
	}  \end{remark}

	\begin{definition}[Globally integrable quantum system]
		\label{gloInt}
		A linear selfadjoint operator $H_L$ will be said to be the Hamiltonian
		of a \emph{globally integrable quantum system} if
		there exists a function 
		$h_L : \R^d \rightarrow \R$ such that
		\begin{gather*}
			H_L = h_L(I_1,\dots,I_d)\,,
		\end{gather*}
		where the operator function is spectrally defined.
	\end{definition}
	\begin{remark}
		Systems fulfilling Definition \ref{gloInt} with the further property
		that the multiplicity of common eigenvalues $(a^1,...,a^d)$ of the
		actions is 1 were called toric integrable quantum systems in \cite{toth2002p}.
	\end{remark}
	If $H_L=h_L(I_1, \dots, I_d)$ is a globally integrable quantum system, then its
	eigenvalues are
	\[
	\omega_a:=h_L(a^1,...,a^d)\equiv h_L(a)\ ,\quad a\equiv (a^1,...,a^d)\in\Lambda\, .
	\]
	We denote by $\Sigma=\left\{ \omega_a\right\}_{a\in\Lambda}$
	the spectrum of $H_L$; in our setting, it coincides with the
	set of the frequencies.
	
	\smallskip      
	We assume the following Hypotheses on $H_L$:
	\begin{manualhypo}{L.0}[Integrability of $H_L$]\label{integro}
		$H_L$ is a globally integrable quantum system. 
	\end{manualhypo}
	
	\begin{manualhypo}{L.1}[Asymptotics]
		\label{asymptotic} There exist constants $C_1,C_2$ and
		$\beta$, with $\beta>1$, s.t.
		\begin{equation}\label{asy.1}
			\left|\omega_a-C_1|a|^{\beta}\right|\leq C_2\, ,
		\end{equation}
		where, for a vector $a\in\R^d$, we denote
		$|a|:=\sqrt{\sum_{j=1}^d(a^j)^2}$  its Euclidean norm.
	\end{manualhypo}
	By \ref{asymptotic}, one can partition $\Sigma$ in separated
	pieces. Precisely we have the following lemma.
	\begin{lemma}
		\label{parti_sigma}
		There exists a sequence of intervals
		$\Sigma_n=[a_n,b_n]$, $n\in\N, \n\ge1$
		and a positive constant $C$,  with the following properties:
		
		\smallskip
		$\bullet$ $a_n<b_n<a_{n+1}< 3n$;
		
		\smallskip
		$\bullet$ $\Sigma\subset[0,b_0]\cup \bigcup_{n}\Sigma_n$;
		
		\smallskip
		$\bullet$ $\left|b_n-a_n\right|\equiv \left|\Sigma_n\right|\leq {2}$;
		
		\smallskip
		$\bullet$ $d(\Sigma_n,\Sigma_{n+1})\equiv a_{n+1}-b_n\geq 2 n^{-d/\beta}$.
	\end{lemma}
	We are now going to assume a nonresonance condition which allows to
	use normal form theory in order to eliminate from the Hamiltonian the terms
	enforcing exchanges of energy among modes labelled by indexes
	belonging to different intervals $\Sigma_n$.\\
	First, in order to keep into account that the system depends both on
	$u$ and $\bar u$, we extend the space of the indexes $a\in\Lambda$, and
	consider 
	\begin{equation}\label{uovo1}
		\Lambda_e:=\Lambda\times
		\left\{\pm1\right\}\ni(a,\sigma)\equiv A\,.
	\end{equation} 
	Then we define what we mean
	by set of resonant indexes.  
	\begin{definition}
		\label{w.non.res}
		A multi-index $\textbf{A}\equiv
		(A_1,...,A_{r})\in\Lambda_e^r$, $A_j\equiv(a_j,\sigma_j)$
		is said to be \emph{resonant} if $r$ is even and there exists
		a permutation $\tau$ of $(1,...,r)$ and a sequence
		$n_1,...,n_{r/2}$ s.t.
		\begin{equation*}
			\forall
			j=1,...,r/2 , \quad 		\omega_{a_{\tau(j)}},\ \omega_{a_{\tau(j+r/2)}}\in
			\Sigma_{n_{\tau(j)}}\ \quad \text{and }\ \sigma_{\tau(j)} = -\sigma_{\tau(j+r/2)}\,.
		\end{equation*} 
		In this case, we will write $\textbf A \in W$.
	\end{definition}
	
	\begin{manualhypo}{L.2}[Non-resonance]\label{hypoNonRes}
		For any $r\ge3$, there are constants $\gamma,\tau>0$ 
		such that for any multi-index $\textbf{A}=(A_1\dots,A_r)\in
		\varLambda^r_e\setminus W $ one has
		\begin{equation*}
			\left|\sum_{j=1}^{r}\sigma_j\omega_{a_j}\right|\ge \frac{\gamma}{\left(\max_{j=1,\dots,r}{|a_j|}\right)^\tau}\,.
		\end{equation*}
	\end{manualhypo}
	Finally, we assume a clustering property of Bourgain's type. 
	\begin{manualhypo}{L.3}[Bourgain clusters]
		\label{bourgain.abstract} There exists a partition 
		\begin{gather*}
			\Lambda = \bigcup_{\alpha\in\mathfrak{A}}\Omega_\alpha
		\end{gather*}
		with the following properties.
		\begin{itemize}
			\item[i.] Each $\Omega_\alpha$ is dyadic, in the sense that there exists a 
			constant $C$, independent of $\alpha$, such that
			\begin{gather*}
				\sup_{a\in\Omega_\alpha}|a|\le C \inf_{a\in\Omega_\alpha}|a|\,.
			\end{gather*}
			\item[ii.] There exist $\delta, C_\delta>0$ such that, if $a\in\Omega_\alpha$ and $b\in\Omega_\beta$ with $\alpha\not=\beta$, then
			\begin{gather*}
				|a-b| + |\omega_a-\omega_b| \ge C_\delta (|a|^\delta + |b|^\delta)\,.
			\end{gather*}
		\end{itemize}
	\end{manualhypo}
	\subsection{Assumption on the nonlinearity  and statement} \label{nonlinP}
	Concerning $P$ we assume the following Hypothesis:
	\begin{manualhypo}{P}		\label{Nonline}
		\begin{itemize}
			\item[(1)] $P$ has the structure
			\begin{equation}
				\label{nonlin.1}
				P(u,\bar u)=\left(\int_M F(N(u,\bar u),u(x),\bar u(x),x)dx\right)  \ ,
			\end{equation}
			where
			\begin{equation*}
				N(u,\bar u):=\int_{M}u(x)\bar u(x)dx  
			\end{equation*}
			and
			$F\in C^{\infty}(\cU\times\cU\times\cU\times M;\C)$ is a smooth
			function and $\cU\subset \C$ an open neighbourhood of the
			origin.
			\item[(2)] if $\bar u$ is the complex conjugate of
			$u$, then $F(N,u,\bar u,x)\in\R$.
			
			\item[(3)] $P$ has a zero of order at least three at $u=0$.
		\end{itemize}  
	\end{manualhypo}
	Our main result is the following:
	\begin{theorem}\label{ab.res}
		Consider the Hamiltonian system \eqref{HL.0}. Assume
		Hypotheses \ref{integro}, \ref{asymptotic}, \ref{hypoNonRes}, \ref{bourgain.abstract}, \ref{Nonline},
		then for any
		integer $r\geq3$, there exists $s_r\in\n$ such that, for any
		$s\ge s_r$, there are constants $\epsilon_0>0$, $c>0$ and $C$
		for which the following holds: if the initial datum
		$u_0\in H^s(M,\C)$ fulfills
		\begin{gather*}
			\epsilon \coloneqq \norm{u_0}_s < \epsilon_0\,,
		\end{gather*}
		then the Cauchy problem has a unique solution 
		$u\in\mathcal C^0\left((-T_\epsilon,T_\epsilon), H^s(M,\C)\right)$ with 
		$T_\epsilon > c \epsilon^{-r}$. Moreover, one has
		\begin{gather}
			\norm{u(t)}_s \le C \epsilon,\quad \forall
			t\in(-T_\epsilon,T_\epsilon)\ .
		\end{gather}
	\end{theorem}
	\subsection{Remark on the abstract theorem}\label{steep.sub}
	We conclude the presentation of our main abstract results with a couple of comments. \\
	The assumptions on $H_{L}$ and
	$P$ are deeply connected with the specific equation one wants to consider and the geometry of the manifold on which the equation is posed. More specifically:
	\begin{itemize}
		\item (Assumptions on $P$) Assumption \ref{Nonline} on
		the nonlinearity is not strictly necessary. Actually
		we prove the abstract result in a slightly more
		general setting, where $P$ is a function with
		``localized coefficients'' (according to
		Def. \ref{locafunc}). This is done in Theorem
		\ref{LC}.	Then Theorem \ref{lo.generale}
		guarantees that nonlinearities fulfilling
		Hypothesis \ref{Nonline} belong to the
		class of functions with localized coefficients, and
		Theorem \ref{ab.res} follows as a consequence of
		Theorem \ref{LC} and Theorem \ref{lo.generale}.
		
		\item (Assumptions on $H_L$) In Section \ref{manifolds} we will
		consider some manifolds $M$ on which the Laplace-Beltrami operator
		$-\Delta_g$ is a globally integrable quantum system. Then, for these
		choices of the manifold $M$, in Section \ref{models} we will
		consider situations where $H_L$ is a (parameter dependent) function
		of $-\Delta_g$. To fulfill the nonresonance relation \ref{hypoNonRes}, we will
		suitable tune the parameter in $H_L$ in Subsections
		\ref{Schrodinger potenziale}--\ref{section beam}. The clustering
		Hypothesis \ref{bourgain.abstract} may follow directly from the
		spectral property of $-\Delta$ (this is the case for flat tori) or may follow from a ``quantum
		Nekhoroshev theorem'' proved in \cite{QN}, that extends the results in \cite{langella2022spectral,langella2022growth}, according to which, if
		the function $h_L$ (c.f. Definition \ref{gloInt}) is \emph{steep}
		(see Definition \ref{steep.def} below), then Hypothesis
		\ref{bourgain.abstract} is automatic. 
		We state such a quantum Nekhoroshev theorem in the rest of this section.
	\end{itemize}

	We first need some definitions:
	\begin{definition}
		\label{homoinfi}
		A  function $h_L\in C^\infty(\R^d;\R) $ is said to be homogeneous of
		degree $\tt d $ at infinity if there exists an open ball $\mathcal{B}_r\in \r^d$, centered at the origin, such that
		\begin{gather*}
			h_L(\lambda a) = \lambda^{\tt d} h_L(a)\,,\qquad
			\forall \lambda>0\,,\quad a\in
			\r^d\setminus\mathcal{B}_r\ .
		\end{gather*}
	\end{definition}
	
	For the sake of completeness, we now recall the
	definition of steepness \cite{GCB}. It is important to note that such a property is generic and it is implied by convexity. Moreover, it can be quite
	easily verified using the results of \cite{Nie}. For the examples treated in 
	Subsect. \ref{manifolds} it was explicitely proved in \cite{QN}.
	\begin{definition}[Steepness] \label{steep.def}
		Let $\cU\subset \R^d$ be a bounded connected open set with nonempty
		interior.  A function $h_L\in C^ 1 (\cU )$ is said to be steep in
		$\cU$ with {steepness radius $\tr$,} steepness indices $\alpha_1,\dots
		,\alpha_{d-1}$ and (strictly positive) steepness coefficients $\tB_ 1
		,\dots , \tB_{d-1}$, if its gradient $\upsilon(a):=\frac{\partial
			h_L}{\partial a}(a)$ satisfies the following estimates:
		$\displaystyle{\inf_{a\in\cU}\norm{\upsilon(a)}>0}$ and for any
		$a\in\cU$ and for any $s$ dimensional linear subspace $E \subset \R^d$
		orthogonal to $\upsilon(a)$, with $1\leq s\leq d-1$, one has
		\begin{equation*}
			\max_{0\leq\eta\leq\xi}\min_{u\in E:\norm{u}=1}
			\norm{\Pi_M \upsilon(a+\eta u)}\geq \tB_s\xi^{\alpha_s}\,, 
			\quad \forall \xi \in (0, \tr]\, ,
		\end{equation*}
		where $\Pi_E$ is the orthogonal projector on $E$. The quantities $u$
		and $\eta$ are also subject to the limitation $a+\eta u\in\cU$. 
	\end{definition} 
	
	The following theorem holds:
	
	\begin{theorem}\label{QNtoNoi}
		Let $H_L = h_L(I_1,\dots,I_d)$ a globally integrable quantum system. 
		If $h_L$ is smooth and homogeneous at infinity (see Def. \ref{homoinfi}) and steep (see Def. \ref{steep.def}), then  Hypothesis \ref{bourgain.abstract} is satisfied.
	\end{theorem}
	We will prove Theorem \ref{QNtoNoi} in the Appendix
	\ref{QNtoNoiSection}, as a Corollary of \cite[Theorem 8.28]{QN}.
	
	\section{Applications}\label{applications}
	
	In this section, we collect various applications which will be
	a consequence of the abstract Theorem.\ref{ab.res}.
	As anticipated in the Introduction, we are interested in
	understanding in which manifolds one can prove results of
	almost global existence. We start by presenting in Subsection
	\ref{manifolds} three classes of manifolds for which  Hypothesis \ref{bourgain.abstract} holds for the Laplace-Beltrami
	operator; then, in Subsection \ref{models}, we consider some explicit
	PDEs on these manifolds. For these equations, 
	we can prove almost global existence. Finally, we recall the connection between the theory
	developed in the present paper and the theory developed in
	\cite{QN}, discussed in Subsection \ref{steep.sub}. In particular, we have shown that the validity of the Hypothesis
	\ref{bourgain.abstract} is automatic if the function ${h_L}$ is
	\emph{steep}, a well-known generic property, that turns out to be the key
	assumption for the validity of the classical Nekhoroshev theorem for
	finite dimensional Hamiltonian systems. 
	
	\subsection{Manifolds}\label{manifolds}
	We now exhibit three specific examples of compact manifolds $M$ where the Laplace-Beltrami operator $-\Delta_g$ fulfills Hypotheses \ref{integro} and \ref{bourgain.abstract}, namely $-\Delta_g$ is a quantum integrable system with Bourgain clustering of the eigenvalues.  In fact, all the examples treated in this section are globally integrable quantum systems for which the Laplacian is steep. For simplicity, we will use the notation $-\Delta_g = \Delta$.
	\begin{itemize}
		\item [1.] \textbf{Flat tori}. Given a basis $u_1,...,u_d$ of $\R^d$ we define a maximal dimensional
		lattice $\Gamma\subset \R^d$ by 
		$$
		\Gamma:=\left\{x=\sum_{j=1}^{d}m_ju_j\ ,\quad m_j\in\Z\right\}
		$$
		and the corresponding maximal dimensional torus
		$T^d_{\Gamma}:=\R^d/\Gamma$. By using in $\R^d$ the
		basis $u_j$, one is reduced to the standard torus
		$\T^d$ endowed by a flat metric. Then the actions are
		given by the operators $-\im \partial_j$. Bourgain's
		clustering property holds for the eigenvalues of the
		Laplacian, as proved in \cite{bourgain2004green,BM}.
		
		\item[2.] \textbf{Rotation invariant surfaces}. Consider a real analytic
		function $f:\R^3\to\R$, invariant by rotations around the $z$ axis,
		and assume it is a submersion at $f(x,y,z)=1$ . Denote by $M$ the
		level surface $f(x,y,z)=1$ and endow it by the natural metric $g$
		induced by the Euclidean metric of $\R^3$.  We introduce coordinates in $M$ as
		follows: let $N$ and $S$ be the
		north and the south poles (intersection of $M$ with the $z$ axis) and denote by
		$\theta\in[0,L]$ the curvilinear abscissa along the geodesic given by
		the intersection of $M$ with the $xz$ plane; we orient it as going
		from $N$ to $S$ and consider also the cylindrical coordinates
		$(r,\phi,z)$ of $\R^3$: on $M$ we will use the coordinates
		$$
		(\theta,\phi)\in (0,L)\times(0,2\pi)\ .
		$$
		Using such coordinates, one can write the
		equation of $M$ {by} expressing the cylindrical coordinates of a point in
		$\R^3$ as a function of $(\theta,\phi)$ getting
		$$
		M=\left\{ (r(\theta),\phi,z(\theta))\ ,\quad (\theta,\phi)\in
		(0,L)\times\T^1\right\} \,.
		$$
		Since $\theta$ is a geodesic parameter, the metric takes the form
		$$
		g=r^2(\theta)d\phi^2+d\theta^2\ .
		$$
		We assume that the function $r(\theta)$ has only one
		critical point $\theta_0\in(0,L)$. The fact that the
		Laplacian is a globally integrable quantum system was
		proved by Colin de Verdiere \cite{CdV}. The property of
		Bourgain's clustering was already proved by Delort
		\cite{Del} (see also \cite{QN}).  
		
		\item[3.] \textbf{Lie Groups and Homogeneous
			spaces}. This is a case covered by our theory. In the
		case where $M$ is a compact, simply connected Lie group, then
		property \ref{integro} of the
		Laplacian of $M$ was proved in detail in \cite{QN}. Property \ref{bourgain.abstract} was proved in \cite{BP, BCP} for arbitrary compact Lie groups (see also \cite{QN}). The extension
		to homogeneous spaces can be obtained by using the same tools used in
		\cite{QN} for compact Lie groups (see \cite{BFM22}). 
	\end{itemize}

	{\it From now on we assume that $M$ is one of the previous manifolds.} 
	\subsection{Models}\label{models}
	We now give three concrete examples of equations where $M$ is one of the previous manifolds, and $H_L$ is a function of the Laplace-Beltrami operator $-\Delta_g$ on $M$. 
	In particular, we choose $H_L$ to be a superlinear function of the Laplacian, hence also \ref{asymptotic} readily follows. Finally, in our models $H_L$ is parameter dependent; such parameters will be used in order to fulfill the more technical assumption \ref{hypoNonRes}.
	\subsubsection{Schr\"odinger equation with spectral multiplier}\label{Schrodinger potenziale}
	Our first result concerns a nonlinear Schr\"odinger equation with a
	spectral multiplier, according to \eqref{spec.mult} below. To define it, we start by defining the spectral projector. 
	\begin{defn}\label{proje}
		For $a\equiv(a^1,\dots,a^d)\in\Lambda$, 
		\begin{gather*}
			\Pi_a  \coloneqq \Pi_{a^1}\dots\Pi_{a^d}
		\end{gather*}
		where $\Pi_{a^j}$         is the orthogonal projector on the
		eigenspace of $I_j$ with eigenvalue $a^j$.
	\end{defn}
	Given $u\in L^2(M,\C)$, consider its spectral
	decomposition, 
	\begin{gather*}
		u = \sum_{a\in\Lambda} \Pi_a u \,,
	\end{gather*}
	and let $V=\left\{V_{a}\right\}_{a\in\Lambda}$ with 
	$V_{a}\in\R$ be a sequence; correspondingly, we define
	a spectral multiplier by
	\begin{equation}
		\label{spec.mult}
		V\sharp u:=\sum_{a}V_{a}\Pi_au \,.
	\end{equation}
	In the following, we will assume that $V$ belongs to the space
	\begin{equation}
		\label{spaz.V}
		\cV_n:=\left\{
		V=\left\{V_{a}\right\}_{a\in\Lambda}\ :\ 
		V_{a}\in\R\ , |V_{a}|\langle
		a\rangle^n\in\left[-\frac{1}{2},\frac{1}{2}\right]  \right\}\ ,
	\end{equation}
	that we endow by the product measure. 
	
	\medskip
	\noindent
	Consider the Cauchy problem 
	\begin{gather}\label{convoluzione}
		\begin{cases}
			\ii \partial_t \psi = - \Delta \psi + V \sharp \psi +
			f(x,|\psi|^2)\psi , \quad x \in M\,,
			\\
			\psi(0)=\psi_0
		\end{cases}	
	\end{gather}
	where 	$V\in \cV_n$ and the nonlinearity $f$ is of class
	$C^\infty(M\times\cU,\R)$,  $\cU\subset \R$ being a neighbourhood of
	the origin, and fulfills $f(x,0)=0$,
	$\partial_yf(x,y)|_{y=0}=0$.
	
	\begin{theorem}
		\label{nsl.esti}
		There exists a set $\cV^{(res)}\subset \cV_n$ of zero measure,
		s.t., if $V\in \cV\setminus \cV^{(res)}$ the following holds.
		For any $r\in\N$, there exists $ s_r>d/2$ such that for any
		$s>s_r$ there is 
		$\epsilon_s>0$ and $C>0$ such that if the initial datum for
		\eqref{convoluzione} belongs to $H^s$ and fulfills
		$\epsilon:=\left|\psi_0\right|_s<\epsilon_s$ then
		\begin{equation*}
			\left\|\psi(t)\right\|_s\leq C\epsilon\,,\quad \text{for\ all}\,\quad 
			\left|t\right|\leq C\epsilon^{-r}\,.
		\end{equation*}
	\end{theorem}
	\subsubsection{Sobolev stability of ground state for NLS equation}
	Our second result is the long time Sobolev stability of the ground state solution of  the nonlinear Schr\"odinger equation
	\begin{gather}\label{nls}
		i\dot \psi = -\Delta \psi +
		f\left(\left|\psi\right|^2\right) \psi\ ,
	\end{gather}
	with $f\in C^\infty(\cU;\R)$, $\cU\subset\R$ an open
	neighborhood of the origin, and $f$ having a zero of order at
	least one at the origin. 
	It is well known that for any $p_0 \in \cU \cap \R^+$ equation \eqref{nls} 
	has a solution given by a plane wave of the form 
	\begin{gather*}
		\psi_{*}(t) = \sqrt{p_0} e^{-i\nu t}\,,
	\end{gather*}
	provided $\nu = f(p_0)$. Denote by $\bar\lambda$ the lowest
	non vanishing eigenvalue of $-\Delta$, then we will prove the
	following result.
	\begin{theorem}\label{main.plane}
		Assume there exists  $\bar{p}_0>0$ such that $ \bar\lambda + 2 f(p_0) > 0$
		for any $p_0 \in (0,\bar{p}_0]$.	Then there exists a zero measure set $\cP$ such that
		if $p_0\in (0,\bar{p}_0]\setminus \mathcal{P}$ then for any $r\in \n$ there
		exists $s_r$ for which the following holds. For any $s\ge s_r$, there
		exists constants $\epsilon_0$ and $C$ such that if the initial datum $\psi_0$
		fulfills
		\begin{gather*}
			\norm{\psi_0}_0^2 = p_0,\qquad \inf_{\alpha\in\T}\norm{\psi_0 - \sqrt{p_0}e^{-i\alpha}}_s = \epsilon \le \epsilon_0\,,
		\end{gather*}
		then the corresponding solution fulfills
		\begin{equation*}
			\inf_{\alpha\in\T}\norm{\psi(t) - \sqrt{p_0}e^{-i\alpha}}_s
			\le C \epsilon \qquad \forall\, |t|\le C \epsilon^{-r}
		\end{equation*}
		with $\psi(0) = \psi_0$.
	\end{theorem}
	\begin{remark}
		Note that if $f$ is a positive function (the so-called \emph{defocusing} case) there is no restriction in the $L^2$ norm of the initial datum, since  $\lambda + 2f(p_0) > 0$ for any $p_0$.
	\end{remark}
	\subsubsection{Beam equation}\label{section beam}
	
	The third result concerns the beam equation
	\begin{equation}
		\label{beam}
		\psi_{tt}+\Delta^2\psi+m \psi=-\partial_\psi F(x,\psi)\ ,
	\end{equation}
	with $F\in C^\infty(M\times\cU)$, $\cU\subset\R$ being a neighbourhood of
	the origin, and $m>0$ a real positive parameter that we will call
	mass. We will assume $F$ to have a zero of order at least 2 at $\psi=0$. The precise statement of the main theorem is the following one. 
	\begin{theorem}
		\label{main.beam}
		There exists a set of zero measure ${\cal M}^{(res)}\subset \R^+$ such that if
		$m\in \R^+ \setminus {\cal M}^{(res)}$ then for all $ r\in\N$ there
		exist $s_r>d/2$ such that the following holds. 
		For any $s>s_r$ there exist $ \epsilon_{s}, C>0$ 
		such that
		if the initial datum for \eqref{beam} fulfills
		\begin{equation*}
			\epsilon:=\left\|(\psi_0,\dot \psi_0)\right\|_s:=
			\left\|\psi_0\right\|_{H^{s+2}}+
			\left\|\dot \psi_0\right\|_{H^{s}}<\epsilon_{s}\ , 
		\end{equation*}
		then the corresponding solution satisfies 
		\begin{equation*}
			\left\| \left(\psi(t),\dot \psi(t)\right)\right\|_{s}\leq
			C\epsilon\,,\quad  \text{for}\quad \left|t\right|\leq C\epsilon^{-r}\,.
		\end{equation*}
	\end{theorem}
	\begin{remark}
		\label{extension}
		The result holds also for the case of Hamiltonian
		nonlinearities which are functions also of the first
		and second
		derivatives of $\psi$.
	\end{remark}

	\section{Functional setting}
	
	\subsection{Phase space}
	
	Let $M$ be an arbitrary manifold and assume that there exists a system
	of quantum actions $(I_1,...,I_d)$ (see Def. \ref{glob.int.quant});
	correspondingly we define the spectral projectors $\Pi_a$ on points $a$ 
	of their joint spectrum as in Def. \ref{proje}.

	\begin{definition}\label{def.Hs}
		For any $s\geq0$, the space $\cH^s:= \cH^s(M)$ is the space of the functions
		$u\in L^2(M,\C)$ s.t. 
		\begin{gather}\label{equiNorm}
			\norm{u}_s^2 \coloneqq \sum_{a\in\Lambda}(1+|a|)^{2s}\norm{\Pi_a u}_0^2<\infty\,.
		\end{gather}
		For $s<0$  $\cH^s$ is the
		completion of $L^2$ in the norm \eqref{equiNorm}.
	\end{definition}
	
	\begin{remark}
		\label{sob.1}
		By (i.) and (iii.) of Def. \ref{glob.int.quant}, for any $s$ the norm
		\eqref{equiNorm} is equivalent to the norm 
		\begin{gather*}
			\norm{\left(\mathds{1} - \Delta\right)^{\frac{s}{2}} u }_{L^2(M,\mathbb{C})}\,,
		\end{gather*}
		so that the spaces $\cH^s$ of Def \ref{def.Hs} are equivalent to the standard
		Sobolev spaces $H^s(M, \C)$.
	\end{remark}
	It is also useful to introduce the spaces 
	\begin{gather*}
		\H^{-\infty} \coloneqq \bigcup_s \H^s, \qquad
		\H^\infty \coloneqq \bigcap_s \H^s \,.
	\end{gather*}
	In the following, we will work in the complex extension of the phase
	space, which amounts to considering $\bar u$ as a variable
	independent of $u$, where the bar denotes the complex conjugate.
	\begin{defn}
		For $s\in\r$, define $\cH^s_e \coloneqq \H_e^s(M) \coloneqq \H^s(M) \times
		\H^s(M)$. For $u\equiv(u_+,u_-)\in\cH^s_e$ we will use the norm
		\begin{gather*}
			\norm{u}_{\H_e^s(M)}^2\equiv
			\norm{(u_+,u_-)}_{\H_e^s(M)}^2
			\coloneqq 
			\left\|u_+\right\|_s^2+\left\|u_-\right\|_s^2\,.
		\end{gather*}
	\end{defn}
	\noindent
	Correspondingly, for $u\in\H_e^s$ and $A\equiv(a,\sigma)\in\Lambda_e$, we define 
	\begin{gather*}
		\Pi_A u \coloneqq \Pi_{(a,\sigma)} (u_+,u_-) \coloneqq \Pi_a u_\sigma\,,
	\end{gather*}
	where $\Pi_a$ is given in Definition \ref{proje}. For
	$u\in\cH_e^s$ one has again the spectral decomposition
	\begin{equation*}
		u=\sum_{A\in\Lambda_e}\Pi_Au\, .
	\end{equation*}
	For $A\in\Lambda_e$ we will denote
	\[
	|A|\equiv|(a,\sigma)|:=|a|\,.
	\]
	Given an element in $\H_e^s$, we define the	involution that we use in order to identify the subspace of ``real functions'', on which $u_+ = \overline{u}_-$.
	\begin{defn}\label{reals.u}
		Let $u\equiv(u_+,u_-)\in \H^{-\infty}_e$, we define
		\[
		I(u):=(\overline{u_-},\overline{u_+}) 
		\] 
		with the bar denoting the
		complex conjugate.
		If $I(u) = u$
		we will say that $u$ is \emph{real}. 
	\end{defn}
	Correspondingly, it is useful to define, for $A\in\Lambda_e$,
	\begin{gather*}
		\bar{A} \equiv \overline{(a,\sigma)} \coloneqq (a,-\sigma)\,,
	\end{gather*}
	so that one has
	\begin{gather*}
		\Pi_A\left( I(u)\right) =\overline{\Pi_{\bar{A}}u} \,.
	\end{gather*}
	\begin{definition}\label{pallacentrata}
		We will denote the ball centered in the origin of $\H^s_e$ of radius $R$ by
		\begin{gather*}
			\mathcal{B}_R^s \coloneqq \left\{u\in\H^s_e\colon \norm{u}_s < R \right\}\,.
		\end{gather*}
	\end{definition}
	\subsection{Hamiltonian structure}
	Given a function $H\in C^\infty(\cO,\mathbb{C})$ with  
	$\cO\subset\H^s_e$ open for some $s$, we define the corresponding Hamiltonian vector field by
	\begin{gather*}
		\dot{u}_+ = i \grad_{u_-} H,\qquad \dot{u}_- = -i \grad_{u_+} H
	\end{gather*}
	where $\grad_{u_{\pm}}$ is the $L^2$-gradient with respect to
	$u_{\pm}$. Namely, it is  defined by the following identity:
	\begin{gather*}
		\text{d}_{u_+} H h_+ = \langle\grad_{u_+} H, h_+\rangle,\qquad \forall h_+\in \H^{\infty}(M)\,,
	\end{gather*}
	and similarly for $\grad_{u_-}$. 
	We will denote the vector field associated to $H$ by
	\begin{gather}\label{defHamVec}
		X_H \coloneqq \left(i\grad_{u_-} H, -i \grad_{u_+} H\right).
	\end{gather}
	The Poisson brackets of two functions are defined as follows.  
	\begin{defn}\label{Poissonbrack}
		Given two functions $f,g \in \mathcal{C}^\infty(\cO,
		\mathbb{C})$, with $\cO$ as above, we  define their Poisson brackets by
		\begin{gather*}
			\left\{f,g\right\}(u) \coloneqq \text{d}f(u) X_g(u )\,. 
		\end{gather*}
	\end{defn}
	\begin{remark}
		In general, the Poisson brackets of two functions can be
		ill-defined. In the framework developed in the next paragraph,  this fact does not
		happen and Poisson brackets will always be well defined, as guaranteed by Lemma \ref{Poisson}.
	\end{remark}

	\subsection{Polynomials with	localized coefficients}\label{locali}
	In order to prove the stability result of Theorem
	\ref{ab.res}, we shall work with nonlinearities $P$ more
	general than those allowed by Hypothesis \ref{Nonline}: they
	belong to a class which is a generalization of the
	\emph{polynomials with localized coefficients} introduced in
	\cite{DS,BamMon}. The generalization rests in the fact that the indexes
	labelling the components of a polynomial are here vectors
	$a\in\Lambda$ instead of scalar coefficients. This complicates
	considerably the proof of the fundamental lemmas.

	\medskip
	\noindent
	First, we recall that given a polynomial function $P$ of degree $r$ there
	exists a unique  $r$-linear symmetric function $\widetilde{P}$ such that
	\begin{equation*}
		P(u)=\widetilde P(u,...,u).
	\end{equation*}
	In the following, we will measure the size of an index by the size of
	the corresponding frequency.  Precisely we give the following definition.

	\begin{defn}
		For any $a\in \Lambda$, we define
		\begin{gather}\label{newOrder}
			\len{a} \coloneqq \omega_a^{\frac{1}{\beta}}\,\,.
		\end{gather} 
		Moreover, for $A\in \Lambda_e$ we define  
		\begin{equation}\label{uovo2}
			\len{A} = \len{(a,\sigma)}\coloneqq \len{a}\,.
		\end{equation}
	\end{defn}
	
	\begin{remark}\label{equi.mod}
		By the asymptotic property of the frequency, there exist  $C_1,C_2>0$
		such that, for any $\Lambda\ni a\not=0$, one has the equivalence
		\begin{gather*}
			C_1 |a|\le \len a \le C_2 |a| \,.
		\end{gather*}
		As a consequence  
		the quantity
		\begin{gather*}
			\norm{u}_s^2 \coloneqq \sum_{a\in\Lambda}(1 + \len{a})^{2s} \norm{\Pi_a u}_0^2\,
		\end{gather*}
		controls and is controlled by the $s$
		Sobolev norm. 
	\end{remark}
	
	\begin{remark}
		In the following, we often use the notation $ a \lesssim b$ 
		meaning that there exists a constant $C$, independent of the relevant parameters, 
		such that $a \le C b$. Moreover, we will write $a\sim b$ if $a\lesssim b$ and $b\lesssim a$. \\ 
		If we need to specify that the constant depends on a parameter, 
		say for example $s$, we will write $a\lesssim_s b$.
	\end{remark}
	
	\begin{definition}\label{mu.S1}
		Consider  $r\geq1$ and 
		a multi-index $\textbf{a}= (a_1,\dots,a_r)\in \Lambda^r$.
		
		\noindent
		$(i)$ We denote by $\tau_{ord}$ 
		the permutation of $(1,...,r)$ with the
		property that
		\begin{equation*}
			\len{a_{\tau_{ord}(j)}}\geq
			\len{a_{\tau_{ord}(j+1)}}\,,\quad\forall j=1,...,r-1\,.
		\end{equation*}

		\noindent
		$(ii)$
		We define
		\begin{equation*}
			\begin{aligned}
				\mu(\textbf{a}) &\coloneqq \len{a_{\tau_{ord}(3)}}\,,
				\\
				S(\textbf{a}) & \coloneqq \mu(\textbf{a}) + |a_{\tau_{ord}(1)} - a_{\tau_{ord}(2)}|\,.
			\end{aligned}
		\end{equation*}
		We also  set
		$\mu(\textbf{A})\coloneqq\mu(\textbf{a})$ and $ S(\textbf{A})\coloneqq S(\textbf{a})$.
	\end{definition}
	We are now in the position to state the localization property fulfilled by the nonlinearity.
	\begin{definition}[Polynomial with localized coefficients]
		\label{d.4.2}
		$(i)$ Let $\nu\in[0,+\infty)$, $N\geq1$. We denote by  $L^{\nu,N}_r$
		the class of the polynomials  $F$ homogeneous of
		degree $r$, such that
		there exists $C_N$ s.t. 
		\begin{eqnarray}
			\label{e.4.2}
			\left|\widetilde F(\Pi_{A_1}u_1,...,\Pi_{A_r}u_r) \right|\leq C_N
			\frac{\mu(\bA)^{\nu+N}}{S(\bA)^{N}}
			\norm{\Pi_{A_1}u_1}_0...\norm{\Pi_{A_r}u_r}_0 \, ,
			\\
			\forall\,
			u_1,\dots,u_r \in\Ph\infty\ , \quad \forall \bA\in\Lambda_e^r\, .\nonumber
		\end{eqnarray}
		The smallest possible constant $C_N$ such that \eqref{e.4.2} holds defines a norm in
		$L^{\nu,N}_r$, precisely
		\begin{gather*}
			\norm{F}^{\nu,N} \coloneqq \sup_{\norm{u_1}_0 = 1 ,\dots, \norm{u_r}_0 = 1} 
			\sup_{A_1,\dots,A_r} |\tilde{F}(\Pi_{A_1}u_1,\dots,\Pi_{A_r}u_r)|
			\frac{S(\textbf{A})^N}{\mu(\textbf{A})^{N+\nu}}\, .
		\end{gather*}
		
		\noindent
		$(ii)$ We say that a polynomial $F$ has \emph{localized coefficients} if there exists
		$\nu\in[0,+\infty)$ and $N_0$, such that, for any
		$N\geq N_0$ one has $F\in L^{\nu,N}_{r}$. In this case we 
		write $F\in L_{r}:=\cup_{N_0}\cup_{\nu\geq0}\cap_{N\geq N_0}L_{r}^{\nu,N}$. 
	\end{definition}
	
	For non homogeneous polynomials, we use the following notation.
	\begin{defn}\label{sommadipoli.0}
		For $r<\bar r$, we define the space 
		\begin{gather*}
			L^{\nu,N}_{r,\br} \coloneqq \bigoplus_{l=r}^{\br}L^{\nu,N}_l \, .
		\end{gather*}
	\end{defn}
	
	\begin{definition}
		For $F\in L^{\nu,N}_r$ we define
		\begin{equation*}
			\left\| F\right\|^{\nu,N}_R:=\left\|F\right\|^{\nu,N}R^r\,.
		\end{equation*}
		For $F\in L^{\nu, N}_{r,\br}$, so that $F=\sum_{l=r}^{\br}F_l$
		with $F_l\in L^{\nu, N}_l$, we define 
		\begin{gather*}
			\norm{F}^{\nu,N}_{R}\coloneqq\sum_{l=r}^{\br}\norm{F_l}^{\nu,N}_{R}\, .
		\end{gather*}
	\end{definition}
	We also need the following definition:
	
	\begin{defn}[Function with localized coefficients]\label{locafunc}
		Let $s_0>0$ and consider a function
		$F\in C^{\infty}(\cO;\C)$, with $\cO\subset\Ph {s_0}$
		an open neighbourhood of the origin.
		
		\noindent
		$F$
		is said to have localized coefficients if both the
		following properties hold:
		\begin{itemize}
			\item all the monomials of the Taylor expansion of $F$ at
			the origin have localized coefficients.
			\item For any $s>0$ large enough there
			exists an open neighbourhood of the origin
			$\cO_s\subset\Ph s$ s.t. $X_F$ belongs to $\mathcal{C}^\infty(\cO_s,\H^s_e)$.
		\end{itemize}
	\end{defn}
	
	\begin{definition}
		\label{real} A function
		$F\in C^{\infty}(\cO;\C)$, with $\cO\subset\Ph {s_0}$
		an open neighbourhood of the origin will be said to be
		real for real $u$ if $F(u)\in\R$ whenever $u=I(u)$ (see Def. \ref{reals.u}).
	\end{definition}
	
	Then (see Theorem \ref{lo.generale}), we will
	show that a non-linearity of the form described in Hypothesis
	\ref{Nonline} is a function with localized coefficients.

	It is also useful to extend the definition to polynomial maps taking value in $\Ph s$. 
	Given a polynomial map $X$ of degree $r$, we still denote with $\widetilde{X}$ 
	the unique  $r$-linear symmetric function such that
	\begin{equation*}
		X(u)=\widetilde X(u,...,u)\,.
	\end{equation*}
	
	\begin{definition}
		\label{d.4.3}
		Let $X:\Ph \infty\to\Ph {-\infty}$ be a polynomial map of
		degree $r$ and let $\widetilde X$ be the associated multilinear
		form. 
		
		\noindent
		$(i)$ Let $\nu\in[0,+\infty)$ and $N\geq1$. We denote by $M_{r}^{\nu,N}$ 
		the space of $r$-homogeneous polynomial 
		maps such that 
		there exists $C_N>0$ such that
		\begin{equation*}
			\begin{aligned}
				\|\Pi_B\widetilde X(\Pi_{A_1}&u_1,...,\Pi_{A_r}u_r) \|_0
				\\&\leq C_N
				\frac{\mu(B,\bA)^{\nu+N}}{S(B,\bA)^{N}}
				\norm{\Pi_{A_1}u_1}_0 \dots\norm{\Pi_{A_r}u_r}_0  \ ,
				\\
				& \forall\, u_1,\dots,u_r\in\Ph \infty\
				,\quad \forall (B,\bA)\in\Lambda_e\times\Lambda_e^r\,.
			\end{aligned}
		\end{equation*} 
		Here we denoted by $(B,\bA)$ the
		multi-index $(B,A_1,...,A_r)$.\\		
		The smallest possible constant $C_N$ defines a seminorm, 
		namely
		\begin{gather*}
			\norm{X}^{\nu,N} = \sup_{\norm{u_1}_0= 1 ,\dots, \norm{u_r}_0 = 1} 
			\sup_{A_1,\dots,A_r,B} \norm{\tilde{X}(\Pi_{A_1}u_1,\dots,\Pi_{A_r}u_r)}_0
			\frac{S(\textbf{A},B)^N}{\mu(\textbf{A},B)^{N+\nu}} \,.
		\end{gather*}
		We also define
		\begin{equation*}
			\norm{X}^{\nu,N}_R:=\norm{X}^{\nu,N}R^{r} \,.
		\end{equation*}
		\noindent
		$(ii)$ We say that a map $X$ has localized coefficients if there is $\nu\geq0$ such that for any $N$ one has
		$F\in M_{r}^{\nu,N}$, and we write $X\in M_{r}:=\cup_{\nu\geq0}\cap_{N\geq1}M_{r}^{\nu,N}$.
	\end{definition}
	
	It is easy to see that if a polynomial function has localized
	coefficients, then its Hamiltonian vector field has localized
	coefficients.
	\begin{lemma}\label{controlField}
		Let $F \in L^{\nu,N}_{r+1}$, then $X_F \in M^{\nu,N}_{r} $. Furthermore, 
		\begin{gather*}
			\norm{X_F}^{\nu,N} \le r
			\norm{F}^{\nu,N} \,
		\end{gather*}
		and therefore
		\begin{equation*}
			\norm{X_F}^{\nu,N}_R \le \frac{r}{R}
			\norm{F}^{\nu,N}_{R}\,.
		\end{equation*}
	\end{lemma}

	\begin{remark}
		\label{r.4.2}
		By the very definition of the property of localization of coefficients
		it is clear that any (finite) linear combination of functions or maps
		with localized coefficients has localized coefficients.
	\end{remark}

	\subsubsection{General properties}
	In this subsection, we state the main properties of polynomials
	with localized coefficients introduce in Def. \ref{d.4.2}. The
	corresponding Lemmas will be proved in Appendix
	\ref{esti.appendix}

	\begin{theorem}[Tame estimate]\label{tameMulti}
		Let $X\in \mathcal{M}^{\nu,N}_r$ and fix
		$s>\frac{3}{2}d+\nu$. If $N>d + s$, for any
		$s_0\in(\frac{3}{2}d+\nu,s)$  one has
		\begin{equation}\label{4.7}
			\begin{aligned}
				\norm{\widetilde{X}(u_1,\dots,u_r)}_s &\lesssim_{s,s_0,N} \norm{X}^{\nu,N} 
				\sum_{j=1}^{r} \norm{u_j}_s\prod_{k\not=j}\norm{u_k}_{s_0}
				\\&
				\forall u_1,\dots,u_r\in \cH^{\infty}_e\ .
			\end{aligned}
		\end{equation}
	\end{theorem}
	
	\begin{corollary}
		\label{c.4.1}
		Let $F$ be a polynomial function with localized coefficients, then
		\eqref{4.7} holds for its Hamiltonian vector field.
	\end{corollary}
	
	\begin{corollary}
		\label{coro.tame}
		Let $X\in \mathcal{M}^{\nu,N}_r$ and let
		$s>\frac{3}{2}d+\nu$. If $N>d + s$, for any
		$s_0\in(\frac{3}{2}d+\nu,s)$,  one has
		\begin{align}\label{homog.esti}
			\sup_{\left\|u\right\|_s\leq R}\norm{{X}(u)}_s \lesssim_{s,s_0,N} \norm{X}^{\nu,N}_R\, .
		\end{align}
	\end{corollary}
	
	\begin{lemma}[Poisson brackets]\label{Poisson}
		Given $F\in L^{\nu_1,N}_{r_1+1}$ and $X\in M^{\nu_2,N}_{r_2}$, we have
		\begin{gather*}
			dP\,X \in  L^{\nu', N'}_{r_1 + r_2}
		\end{gather*}
		with $N' =  N - d - 1 - \max\{\nu_1,\nu_2\} $ and $\nu' = \nu_1 + \nu_2 + d + 1$.
		Moreover, 
		\begin{gather*}
			\norm{dF\,X}^{\nu'\,N'} \lesssim \norm{F}^{\nu_1,N} \norm{X}^{\nu_2,N}
		\end{gather*}
	\end{lemma}
	
	\begin{corollary}\label{esti.poi}
		Let $F\in L^{\nu,N}_{r_1,r_2}$ and $G\in L^{\nu',N'}_{r_1',r_2'}$,
		then one has $\left\{ F; G \right\} \in L^{\nu'', N''}_{r_1 + r'_1 -2, r_2 + r'_2-2}$, with
		\begin{equation*}\label{esti.poi.1}
			\left\|\left\{F;G\right\}\right\|^{\nu'',N''}_{R}\lesssim
			\frac{1}{R^2}\left\| F\right\|_R^{\nu,N}\left\| G\right\|_R^{\nu',N'}\ ,
		\end{equation*}
		$N''=\min\{N, N'\}-d-1-\max\left\{\nu,\nu'\right\}$, and $\nu''=\nu+\nu'+d+1$.
	\end{corollary}
	\subsubsection{High and low modes}
	In the definition of the normal form we will distinguish between low and
	high modes. To this end, we fix some large $K$ with the
	property that $K^\beta$ 
	lays between two intervals $\Sigma_n,$ and $ \Sigma_{n+1}$ defined 
	in Lemma \ref{parti_sigma}. 
	\begin{defn}
		For $K>1$, as above, we define
		\begin{gather*}
			u^\le =\Pi^{\le}u  \coloneqq   \sum_{\{A\colon \len{A}\le K\}}\Pi_A u\, ,
			\\  
			u^\perp = \Pi^{\perp}u \coloneqq   \sum_{\{A\colon \len{A}>K\}}\Pi_A u\, .
		\end{gather*}
	\end{defn}
	\begin{defn}
		A polynomial $F$ of degree $r$ is of order $k\le r$ 
		in the high modes $u^\perp$ if, $\forall \lambda>0$
		one has
		\begin{equation}
			\label{homohigh}
			F(u^{\leq}+\lambda u^{\perp})=\lambda^kF(u^{\leq}+u^{\perp})\ .
		\end{equation}
	\end{defn}
	\noindent
	By Theorem \ref{tameMulti} and noticing that
	\begin{equation*}
		\left\|\Pi^\perp u\right\|_{s_0}\leq \frac{1}{K^{s-s_0}}
		\left\|\Pi^\perp u\right\|_{s}\ ,
	\end{equation*}
	we immediately have the following
	Corollary.
	\begin{corollary}
		\label{cubic}
		Let $F\in L^{\nu,N}_{r+1}$.
		
		\noindent
		i) If $F$ is of order at least three in $u^\perp$ then, for every $s_0\in \left(\frac{3}{2}d + \nu, s\right)$, we have
		\begin{gather*}
			\sup_{\norm{u}_s\le R}\norm{X_F(u)}_s 
			\lesssim \frac{\norm{F}^{\nu,N}_R}{K^{s-s_0}} \frac{1}{R}\,.
		\end{gather*}
		ii) If $F$ is at least of order two in $u^\perp$ then, 
		for every $s_0\in \left(\frac{3}{2}d + \nu, s\right)$, we have
		\begin{gather*}
			\sup_{\norm{u}_s\le R}\norm{\Pi^{\le}X_F(u)}_s 
			\lesssim \frac{\norm{F}^{\nu,N}_{R}}{K^{s-s_0}} \frac{1}{R}\,.
		\end{gather*}
	\end{corollary}
	\noindent
	We also have the following simple, but important corollary.
	\begin{corollary}\label{cutOff2}
		Let $F\in L_{r+1}^{N,\nu}$ be of order $2$ in $u^\perp$; assume that
		\begin{gather*}
			\tilde{F}(\Pi_{A_1}u_1,\dots,\Pi_{A_{r+1}}u_{r+1})\not=0\quad\implies \quad
			|a_{\tau_{ord(1)}} - a_{\tau_{ord(2)}}| > K^\delta\,,
		\end{gather*}
		with $\tau_{ord}$ the ordering permutation defined in
		\ref{mu.S1}.  Then, $\forall N'>N$ one has
		\begin{equation}\label{sti.lunghi}
			\norm{F}^{\nu,N}\leq\frac{\norm{F}^{\nu,N'}}{K^{\delta (N'-N)}}\,,
		\end{equation}
		and therefore, for any $s$ large enough,
		\begin{equation}\label{impro.tame.proof}
			\sup_{\norm{u}_s\le R}\norm{X_F(u)}_s 
			\lesssim \frac{\norm{F}^{\nu,N'}_R}{K^{\delta (N'-N)}}  \frac{1}{R}\,.
		\end{equation}
	\end{corollary}

	\section{The abstract theorem and the normal form lemma}\label{normalformSec}
	In the following two Sections we prove the stability Theorem \ref{ab.res}. We shall deduce the result from
	an abstract normal form theorem on a class of Hamiltonian functions discussed in Section \ref{normalformSec}.
	In Section \ref{almost.sec} we deduce some dynamical consequences
	of the normal form and prove the main abstract 
	result.\\ \\
	
	In this Section, we prove Theorem \ref{ab.res}. We proceed as follows: first, we state an almost global existence result holding for nonlinearities $P$ with localized coefficients, namely Theorem \ref{LC}. Then we prove that all nonlinearities $P$ satisfying Assumption \ref{Nonline} have localized coefficients (this is Theorem \ref{lo.generale}), and as a consequence, we deduce Theorem \ref{ab.res} from Theorems \ref{LC} and \ref{lo.generale}.\\ 
	Consider a Hamiltonian function of the form
	\begin{align}
		\label{startHam}
		H(u) = H_0(u) + P(u) \,,
		\\
		\label{Hzeor}
		H_0(u) = \int_M u_- H_L u_+ dx=\sum_{a\in \Lambda} \omega_a \int_M
		\Pi_{(a,+)}u\,\Pi_{(a,-)}u\, dx\,.
	\end{align}
	We have the following:
	\begin{theorem}
		\label{LC}
		Consider the Hamiltonian \eqref{startHam}, assume that $\omega_a$
		fulfill the Hypotheses \ref{asymptotic}, \ref{hypoNonRes},
		\ref{bourgain.abstract}, and assume that $P$ is a function with
		localized coefficients, which is real for real states and that has a
		zero of order at least three at the origin, then for any integer
		$r\geq3$, there exists $s_r\in\n$ such that, for any $s\ge s_r$, there
		are constants $\epsilon_0>0$, $c>0$ and $C$ for which the following
		holds: if the initial datum $u_0\in \cH^s(M,\C)$ fulfills
		\begin{gather*}
			\epsilon \coloneqq \norm{u_0}_s < \epsilon_0\,,
		\end{gather*}
		then the Cauchy problem has a unique solution 
		$u\in\mathcal C^0\left((-T_\epsilon,T_\epsilon), H^s(M,\C)\right)$ with 
		$T_\epsilon > c \epsilon^{-r}$. Moreover, one has
		\begin{gather}\label{stimafinale}
			\norm{u(t)}_s \le C \epsilon,\quad \forall
			t\in(-T_\epsilon,T_\epsilon)\ .
		\end{gather}
	\end{theorem}
	Theorem \ref{ab.res} is a consequence of Theorem \ref{LC}, in view of the following:
	\begin{theorem}\label{lo.generale}
		A nonlinear functional of the form \eqref{nonlin.1} fulfilling
		Hypothesis \ref{Nonline} is a function with localized
		coefficients.
	\end{theorem}
	We postpone the proof of Theorem \ref{lo.generale} to Appendix \ref{locasec} and we now turn to the proof of Theorem \ref{LC}. All the remaining part of the present section, as long as the following Section \ref{almost.sec}, is devoted to the proof of Theorem \ref{LC}.
	We start with stating and proving the normal form results (see Lemma \ref{iteration} and Proposition \ref{finalNormal}), which are the heart of the proof of Theorem \ref{LC}.
	
	In order to define what we mean by normal form we first give the
	following definition.
	
	\begin{defn}[Support]\label{support}
		Given a homogeneous polynomial 
		\begin{gather*}
			P_l(u) = \sum_{\textbf{A}=(A_1,\dots,A_l)\in \Lambda_e^l} 
			\tilde{P}(\Pi_{A_1} u,\dots,\Pi_{A_l}u)\,,
		\end{gather*}
		we define the support of $P_l$ as
		\begin{gather*}
			supp(P_l) = \left\{\textbf{A} 
			= (A_1,\dots,A_l): \exists\, u_1,\dots,u_l\, s.t.\, \tilde{P_l}(\Pi_{A_1} u_1,\dots,\Pi_{A_l}u_l)\not= 0 \right\}\,.
		\end{gather*}
		If $P=\sum_{l=3}^rP_l$ is a non homogeneous polynomial we define
		$$
		supp(P):=\bigcup_{l}supp(P_l)\ .
		$$
	\end{defn}
	We now define what we mean by normal form.
	
	\begin{defn}[Block Resonant Normal Form]\label{brnf}
		A non homogeneous polynomial $Z_{BR}$ of degree $r$
		will be said to be in \emph{Block Resonant Normal Form} if
		\begin{equation*}
			Z_{BR}=Z_0+Z_B\, ,
		\end{equation*}
		with $Z_0$ and $Z_B$ of order, respectively, $0$ and $2$ in $u^\perp$ and
		\begin{itemize}
			\item[i.] $\bA\in Supp(Z_0)$ implies $\bA\in W$ (see Def. \ref{w.non.res});
			\item[ii.] $\textbf{A}\in Supp(Z_B)$ implies that
			there exists a block $\Omega_\alpha$
			such that
			\[
			a_{\tau_{ord}(1)},a_{\tau_{ord}(2)}\in\Omega_\alpha\ ,\ \text{and}\quad
			\sigma_{\tau_{ord}(1)}\sigma_{\tau_{ord}(2)}=-1\,,
			\]
			where $\tau_{ord}$ is the ordering permutation defined in Def. \ref{mu.S1}.
		\end{itemize}
	\end{defn}
	\begin{defn}[Higher Order Normal Form]\label{brnf3}
		A non homogeneous polynomial $Z_{HO}$ of degree $r$
		will be said to be in \emph{Higher Order Normal Form} if it has the structure
		$$
		Z_{HO}=Z_2+Z_{\geq 3}\ ,
		$$
		with $Z_{\geq
			3}$ of order at least three in $u^\perp$ and $Z_2$ of order $2$ in $u^\perp$  and s.t.:
		\begin{gather*}
			\textbf{A}\in supp(Z_2) \implies
			\left|a_{\tau_{ord}(1)}-a_{\tau_{ord}(2)}
			\right| > C_\delta K^{\delta}.
		\end{gather*} 
	\end{defn}

	The terms in $Z_0$ are resonant in the standard sense
	of perturbation theory, namely they do not enforce the exchange of energy between modes pertaining to different sets $\Sigma_n$; the terms in $Z_B$ do not provoke the exchange of energy between modes pertaining
	to different blocks $\Omega_\alpha$ and thus they conserve the total $L^2$
	norm of the modes of a block $\Omega_\alpha$; finally, according to
	Corollaries \ref{cubic} and \ref{cutOff2}, terms in Higher Order Normal Form will be shown to have a small vector field.
	
	\begin{definition}
		\label{nor.for.gen}
		A polynomial which is the sum of polynomials in normal form according
		to Definitions \ref{brnf} and \ref{brnf3} will be said to be in normal form.
	\end{definition}

	The heart of the proof of Theorem \ref{LC} is the following:
	\begin{proposition}\label{finalNormal}
		For any $r\ge0$, $\exists s_r$  such that, 
		$\forall s\ge s_r$, $\exists R_{s,r}>0$, with the
		property that $\forall R<R_{s,r}$ $\exists K$ and a
		canonical transformation 
		\begin{equation}
			\label{defo.final}
			\cT^{(r)}:\cB^s_{R/2^{2r}}\to\cB_{R}^s\ \quad  \text{with} \quad
			[\cT^{(r)}]^{-1}:\cB^s_{R/4^{2r}}\to\cB^s_{R/2^{2r}}
		\end{equation}
		s.t.
		\begin{gather*}
			H\circ \cT^{(r)} = H_0 + Z_0 + Z_B + \cR^{(\br)}
		\end{gather*}
		with $Z_0$ and $Z_B$ as in K block-resonant normal form, see Def. \ref{brnf}, and 
		\begin{gather*}
			\norm{X_{\cR^{(\br)}}(u)}_s \lesssim R^{r+2},\qquad \forall u\in\mathcal{B}_{R}^s\,.
		\end{gather*}
		Moreover, we have
		\begin{gather}\label{lowmodesZ}
			\norm{\Pi^{\le}X_{ Z_B}(u)}_s \lesssim R^{r+2}\,, \qquad \forall u\in\mathcal{B}_{R}^s\,.
		\end{gather}
	\end{proposition}
	\subsection{Lie Transform}\label{lie}
	The transformation $\cT^{(r)}$ will be constructed by the composition 
	of Lie transforms, so we start by studying the properties of the Lie transform.

	Given $G\in \mathcal{C}^\infty(\H_e^s,\mathbb{C})$, we denote by $\Phi_G^t$ 
	the flow generated by the Hamilton equation $\dot{u} = X_G(u)$. 
	From Lemma \ref{tameMulti} one has the following result.
	{
		\begin{lemma}
			Fir $\br$, let $ 3 \le r \le \br $, $\nu\in[0,+\infty)$, $N\geq1$ and $G\in
			L^{\nu,N}_{r,\br}$.  $\forall s>\frac{3}{2}d+\nu$ there is a
			constant $ C_{\bar{r},N,s}>0$ such that $\forall R>0$ satisfying
			\begin{gather}\label{rs}
				\norm{G}^{\nu,N}_{R} \le \frac{R^2}{C_{\bar{r},N,s}}
			\end{gather}
			the map $\Phi_G^t:\mathcal{B}^s_{R/2}\to \mathcal{B}^s_{R}$ is well
			defined for $|t|\le1$, and moreover
			\begin{gather*}
				\sup_{\norm{u}_s<R/2} \norm{\Phi_G^t(u) - u}_s \lesssim_s |t|  \frac{\norm{G}^{\nu,N}_{R}}{R}\,.
			\end{gather*}
		\end{lemma}
	}
	\begin{defn}
		We call $\Phi_G\coloneqq \Phi_G^t\big\rvert_{t=1}$
		\textit{Lie transform} generated by $G$.
	\end{defn}
	
	\begin{lemma}
		\label{defo}
		Let $G\in L_{r,\br}^{\nu,N}$, $3\leq r\leq \br$ and let $\Phi_G$ be the
		Lie transform it generates. For any $s>\frac{3}{2}d+\nu$ there exists $R_s>0$ such that for any $F\in
		C^{\infty}(B^s_{R_s})$ satisfying 
		\begin{equation*}
			\sup_{\left\|u\right\|_s\leq
				R_s}\left\|X_F(u)\right\|_s<\infty\ ,
		\end{equation*}
		one has 
		\begin{equation*}
			\sup_{\left\|u\right\|_s\leq
				R/2}\left\|X_{F\circ \Phi_G}(u)\right\|_s \le 
			2 \sup_{\left\|u\right\|_s\leq
				R}\left\|X_F(u)\right\|_s\ ,\quad \forall R<R_s\ .
		\end{equation*}
	\end{lemma}

	Defining
	\begin{gather*}
		Ad^0_G (P) \coloneqq P \quad \text{and}\quad Ad^k_G (P)  
		\coloneqq 
		\{Ad^{k-1}_G (P), G\}\, \text{ for } k\ge1
	\end{gather*}
	we have the following standard lemma.
	\begin{lemma}[Lie transform]\label{lie.lemma}
		Let $3 \le r \le \bar{r}$ and $G\in L^{\nu,N}_{r,\br}$. 
		Assume \eqref{rs}, then for any $P\in\mathcal{C}^\infty(\cB^{s}_{2R_s},\mathbb{C})$,  we have, 
		\begin{gather}
			\label{5.10.es}
			P(\Phi_G(u)) = \sum_{k=0}^{n}\frac{1}{k!}\left(Ad^k_G P\right)(u) 
			+ \frac{1}{n!} \int_{0}^{1}(1-\tau)^n \left(Ad^{n+1}_G P\right)\left(\Phi^\tau_G(u)\right)d\tau
		\end{gather}
		for any $ n\in \n$ and $\forall u \in \mathcal{B}_{R/2}^s$.
	\end{lemma}
	From the estimate of Poisson brackets (see Lemma \ref{Poisson}), we deduce the following result.
	\begin{lemma}\label{PoissonBrack}
		Let $G\in L^{\nu,N}_{r_1+2,\br_1+2}$ and $F\in
		L^{\nu,N}_{r_2,\br_2}$, with $r_1 \leq \br_1$ and $r_2
		\leq \br_2$.  $\forall k \geq 0$ we have $Ad^k_G (F) \in
		L^{\nu_k,N_k}_{r_2 + kr_1, \br_2 + k\br_1}$ and
		\begin{gather}
			\label{5.11.e}
			\norm{Ad^k_G (F)}^{\nu_k,N_k}_{R} \lesssim \left(
			\frac{\norm{G}^{\nu,N}_{R}}{R^2}\right)^k \norm{F}^{\nu,N}_{R}\,,
		\end{gather}
		with $N_k = N - k(d + \nu)$ and $\nu_k = k(d+ 2\nu)$.
	\end{lemma}
	\begin{proof}
		We prove it by induction. For $k=1$, the thesis follows from
		\eqref{esti.poi.1}.
		Suppose the thesis is true at step $k$. Exploiting again \eqref{esti.poi.1}, we get
		\begin{equation*}
			\begin{aligned}
				\norm{Ad^{k+1}_G (F)}^{\nu_{k+1},N_{k+1}}_{R} = \norm{\left\{Ad^{k}_G (F), G \right\}}^{\nu_{k+1},N_{k+1}}_{R} 
				\lesssim \\
				\frac{1}{R^2} \left(\frac{\norm{G}^{\nu,N}_{R}}{R^2}\right)^k \norm{F}^{\nu,N}_{R} \norm{G}^{\nu,N}_{R} =
				\left(\frac{\norm{G}^{\nu,N}_{R}}{R^2}\right)^{k+1} \norm{F}^{\nu,N}_{R}\,.
			\end{aligned}
		\end{equation*}
	\end{proof}
	It is useful for the perturbative iteration to summarize the
	last results in the following lemma.
	{
		\begin{lemma}\label{perturbative}
			Fix $\br$, let $P\in L_{r_1,\br_1}^{\nu,N}$ and $G\in
			L^{\nu,N}_{r_2+2,\br_2+2}$ with $r_1<\bar r$, $r_2+2<\bar r$.  Let
			\begin{equation}\label{n.choice}
				n= \frac{\br+3-r_1}{r_2}\,\ ,
			\end{equation}
			if the r.h.s. is integer, otherwise we define $n$ to be
			the r.h.s. of \eqref{n.choice} +1. Let $s>\frac{3}{2}d+2$
			There exists $C_{\bar{r},s,N}>0$  such that if $R$ fulfills \eqref{rs},
			then
			the Lie transform 
			$\Phi_G\colon \mathcal{B}_{R/2}^s\rightarrow\mathcal{B}_{
				R}^s $ is well defined. Moreover, one has
			\begin{gather*}
				P \circ \Phi_G = P + P' + \mathcal{R}_{P,G}
			\end{gather*}
			and $\exists \nu',N'$ s.t., 
			$P' \in L^{\nu',N'}_{r_1+r_2,\br_1+n\br_2}$. Furthermore
			$X_{\mathcal{R}_{P,G}} \in
			\mathcal{C}^\infty(\mathcal{B}_{R}^s,\H^s_e)$ and one has
			\begin{align}
				\label{tesi.lie.2}
				\left\|P'\right\|^{\nu',N'}_R\lesssim \left\|P\right\|^{\nu,N}_R\frac{\left\|G\right\|^{\nu,N}_R}{R^2} 
				\\
				\label{tesi.lie.3}
				\sup_{\norm{u}_s \le
					R/2}\norm{X_{\mathcal{R}_{P,G}}(u)}_s  \lesssim
				\left\|P\right\|^{\nu,N}_R\left(\frac{\left\|G\right\|^{\nu,N}_R}{R^2}\right)^{n+1}
				\frac{1}{R}\ ,
			\end{align}
			so that $X_{\mathcal{R}_{P,G}} $ has a zero of order at least $\br+2$
			at the origin.
		\end{lemma}	
		\proof Define
		$$
		P':=\sum_{k=1}^{n}\frac{1}{k!}Ad^k_G P 
		$$
		and $R_{P,G}$ be the integral term in \eqref{5.10.es}. Then, by Lemma
		\ref{PoissonBrack}, one has
		$$
		\left\|P'\right\|^{\nu',N'}_R\lesssim
		\sum_{k=1}^{n}\frac{1}{k!}\left(\frac{\left\|G\right\|^{\nu,N}}{R^2}\right)^k\left\|P\right\|^{\nu,N}_R\lesssim\frac{
			\left\|G\right\|^{\nu,N}}{R^2}\left\|P\right\|^{\nu,N}_R \ ,
		$$
		provided
		$\left(\frac{\left\|G\right\|^{\nu,N}}{R^2}\right)\leq
		1/2$, which is ensured by \eqref{rs}. This gives \eqref{tesi.lie.2}. To get \eqref{tesi.lie.3}
		just use the estimate \eqref{5.11.e}, Lemma \ref{controlField}
		and Lemma \ref{defo}. \qed
	}

	\subsection{Solution of the homological equation}\label{sec:omo}
	In this subsection, we state and solve the homological equation. 
	\begin{remark}\label{choiceK}
		Recall that, following Lemma \ref{parti_sigma}, the set of frequencies 
		decomposes as the union of bounded and disjoint intervals $\Sigma_n$.
		In the following, we will choose the cut-off $K$ lying between two intervals, 
		namely such that 
		for any  $a,b\in\Lambda $ with $\len{a}< K$ and $\len{b} \ge K$ 
		we have that, for any $n\ge1$,
		\begin{gather}\label{cutoff}
			\omega_a \in \Sigma_n \quad \implies\quad \omega_b \not \in \Sigma_n \,.
		\end{gather}
	\end{remark}
	We start defining the set of nonresonant indexes.
	\begin{defn}[Block $K$ non resonant indexes]\label{nRindex}
		Let $K\gg1$ as in Remark \ref{choiceK}
		We say that a multi-index $\textbf{A} = (A_1,...,A_r)\in\Lambda_e^r$ is
		Block-$K$-non-resonant, and we write $\textbf{A}\in
		\mathcal{I}^K(r)$, if 
		$\len{A_{\tau_{ord}(3)}}< K$ (namely there are at most
		two large indexes) and 
		one of the following holds:
		\begin{itemize}
			\item $\len{A_{\tau_{ord}(1)}}< K$ (there are no indexes larger than K) and $\bA\not \in W$;
			\item  $\len{A_{\tau_{ord}(1)}} \ge K$ and $\len{A_{\tau_{ord}(2)}}< K$ (there is exactly	one large index); 
			\item $\len{A_{\tau_{ord}(2)}} \ge K$ and $\len{A_{\tau_{ord}(3)}}< K$ (there are exactly two
			large indexes) and one of the
			following holds (recall Hyp. \ref{bourgain.abstract}):
			\begin{itemize}
				\item $\exists \alpha$ s.t. $a_{\tau_{ord}(1)},a_{\tau_{ord}(2)} 
				\in \Omega_\alpha$ and $\sigma_{\tau_{ord}(1)}\sigma_{\tau_{ord}(2)} = +1$,
				
				\item $\exists \alpha\not=\beta$ 
				s.t. $a_{\tau_{ord}(1)}\in\Omega_\alpha,$ $a_{\tau_{ord}(2)}
				\in \Omega_\beta$ and
				$|a_{\tau_{ord}(1)}-a_{\tau_{ord}(2)}|\leq K^\delta$.
			\end{itemize}
		\end{itemize}
	\end{defn}
	
	\begin{remark}\label{rmk:giallo}
		By  Definitions \ref{brnf} and \ref{brnf3}, 
		a polynomial supported only on multi-indexes 
		$\textbf{A}\notin\mathcal{I}^K(r)$ is in normal form according to Definition \ref{nor.for.gen}
	\end{remark}

	In the following lemma, we show that there are no resonant
	multi-indexes with just one large index.
	
	\begin{lemma}\label{non.riso.uno}
		Fix $K$ as in\eqref{cutoff}. If $\textbf{A}\in\Lambda_e^r$ 
		is s.t. 
		\begin{gather}\label{hypo.one.large}
			\len{A_{\tau_{ord}(1)}} \ge K\quad\text{and}\quad \len{A_{\tau_{ord}(2)}}< K \,,
		\end{gather}  then there exists a constant $\gamma'_r>0$ such that
		\begin{gather*}
			\Big|\sum_{l=1}^{r}\sigma_l\omega_{a_l}\Big|\ge \gamma'_r K^{-\tau}\,,
		\end{gather*}
		where $\tau>0$ is the constant appearing in Hyp. \ref{hypoNonRes}.
	\end{lemma}
	\begin{proof}
		For simplicity, we can suppose that $A_{\tau_{ord}(1)} = A_1$. 
		We distinguish two cases.\\
		\textbf{Case 1} If $\len{A_1}\ge K_1 =  2 (r-1)^{\frac{1}{\beta}}K$, 
		recalling $\omega_{a_j} < K^\beta$ for any $j=2,\dots,r$, we have that
		\begin{gather}\label{aa}
			\left|\sum_{l=2}^{r}\sigma_l\omega_{a_l}\right| \le  (r-1)  K^\beta\,.
		\end{gather}
		From \eqref{aa} and $\len{A_1}\ge K_1$, we deduce
		\begin{gather*}
			\left|\sum_{l=2}^{r}\sigma_l\omega_{a_l} 
			+ \sigma_1 \omega_{a_1}\right| \ge K_1^\beta - (r-1)  K^\beta \gtrsim 1
		\end{gather*}
		that implies the thesis.\\
		\textbf{Case 2} If $\len{A_1}< K_1 $,  
		we prove that $\textbf{A} \not\in W$. In fact, 
		if by contradiction $\textbf{A} \in W$, it should exist 
		$A' \in \textbf{A}$ with $A_1, A' \in \Sigma_n$ for some $n$. 
		But then, from \eqref{cutoff}, it would follow that $\len{A_{\tau_{ord}(2)}} \ge K$. 
		This is in contradiction with \eqref{hypo.one.large}.\\ 
		Since $\textbf{A} \not\in W$, from Hypothesis \eqref{hypoNonRes} it follows that
		\begin{gather*}
			\left|\sum_{j=1}^{r}\sigma_j\omega_{a_j}\right|
			\ge 
			\frac{\gamma}{\left(\max_{j=1,\dots,r}{|a_j|}\right)^\tau} \gtrsim \frac{\gamma}{K_1^\tau} 
			= \frac{\gamma'_r}{K^\tau}\,,
		\end{gather*}
		with $\gamma'_r = \frac{\gamma}{2(r-1)^{\frac{\tau}{\beta}}}  $.
	\end{proof}
	In the following lemma, we take care of multi-indexes with exactly two indexes with modulus larger than $K$.
	\begin{lemma}\label{non.riso.due}
		If $\textbf{A}\in\mathcal{I}^K(r)$ is s.t. 
		\begin{gather}
			\len{A_{\tau_{ord}(2)}}\ge K \quad\text{and}\quad \len{A_{\tau_{ord}(3)}} < K \,,
		\end{gather}  then there exist constants $\gamma'_r$ and $\tau$ such that
		\begin{gather*}
			\left|\sum_{l=1}^{r}\sigma_l\omega_{a_l}\right|\ge \frac{\gamma'_r}{K^{\tau'}}\,.
		\end{gather*}
	\end{lemma}
	\begin{proof}
		For simplicity, suppose that  
		\[
		A_{\tau_{ord}(1)} = A_1\quad {\rm and} \quad  A_{\tau_{ord}(2)} = A_2\,.
		\] 
		\textbf{Case 1} Consider first the case  $\sigma_1 = \sigma_2$. 
		Reasoning as in the proof of Lemma \ref{non.riso.uno}, we distinguish two cases. \\
		\textbf{Case 1.i} If $\len{A_1}\ge K_1 \coloneqq  2 (r-2)^{\frac{1}{\beta}}K$, 
		the thesis follows trivially since
		\begin{gather*}
			\left|\sum_{l=1}^{r}\sigma_l\omega_{a_l}\right| 
			= \left|\sum_{l=3}^{r}\sigma_l\omega_{a_l} + \sigma_1 \omega_{a_1} + \sigma_2 \omega_{a_2}\right|
		\end{gather*}
		and we can reason as in case 1. in the Proof of Lemma \ref{non.riso.uno}. \\
		\textbf{Case 1.ii} If $\len{A_1} < K_1$, we observe that by definition 
		$\bA \not \in W$, since $\sigma_1 \sigma_2 = +1$ and recalling \eqref{cutoff}. 
		Then, from Hyp. \eqref{hypoNonRes}, it follows
		\begin{gather*}
			\left|\sum_{j=1}^{r}\sigma_j\omega_{a_j}\right|\ge 
			\frac{\gamma}{\left(\max_{j=1,\dots,r}{|a_j|}\right)^\tau} 
			\gtrsim \frac{\gamma}{K_1^\tau} = \frac{\gamma'}{K^\tau}
		\end{gather*}
		for $\gamma' = \frac{\gamma}{2^\tau(r-1)^{\frac{\tau}{\beta}}} $. \\
		\textbf{Case 2} Consider now the case $\sigma_1 \sigma_2 = -1$. 
		It follows by definition of $\mathcal{I}^K(r)$ that $a_{1}\in\Omega_\alpha,$ $a_{2}
		\in \Omega_\beta$ with $\alpha\not=\beta$ and moreover
		$|a_{1}-a_{2}|\leq K^\delta$. From Hyp. \ref{bourgain.abstract}, it follows that 
		\begin{gather*}
			|\omega_{a_1}-\omega_{a_2}| \ge  C_\delta |a_1|^\delta\,.
		\end{gather*}
		\textbf{Case 2.i} If 	
		$\len{a_1} \ge  K_2 \coloneqq 2\left(r-2\right)^{\frac{1}{\delta}} K^{\frac{\beta}{\delta}}$ 
		we observe that $|a_1|\gtrsim K_2$ and then
		\begin{gather*}
			\left|\sum_{l=3}^{r}\sigma_l\omega_{a_l} + \omega_{a_1} - \omega_{a_2}\right| 
			\ge C_\delta|a_1|^\delta - (r-2) K^\beta \gtrsim 1
		\end{gather*}
		and the thesis follows.\\
		\textbf{Case 2.ii} If $\len{ a_1} <  K_2$, one has that $\bA \not \in W$. 
		To see this, remark that
		$|\omega_{a_1}-\omega_{a_2}| \ge C_\delta |a_1|^\delta > C K^\delta > 2$ 
		but each $\Sigma_n$ has a length smaller or equal to 2, so that $a_1, a_2$ 
		belong to different set $\Sigma_n$ and then the definition of $W$ cannot be fulfilled. 
		Then recalling Hyp. \eqref{hypoNonRes}, we conclude that
		\begin{gather*}
			\left|\sum_{j=1}^{r}\sigma_j\omega_{a_j}\right|\ge \frac{\gamma}{\left(\max_{j=1,\dots,r}{|a_j|}\right)^\tau} 
			\gtrsim \frac{\gamma}{K_2^\tau} = \frac{\gamma'}{K^{\tau '}}
		\end{gather*}
		with $\gamma' = \frac{\gamma}{C(r-2)^{\frac{\tau}{\delta}}}$ and $\tau' = \frac{\beta}{\delta}\tau$.
	\end{proof}
	In the next lemma, we solve the homological equation.
	\begin{lemma}[Homological equation]\label{homeq}
		Let $\bar r, l$ be given in such a way that $\bar r\geq
		l\ge3$, then  for any $ F \in L^{\nu,N}_{l,\br}$ there exist
		$G,Z \in L^{\nu,N}_{l,\br}$ which solve the homological equation
		\begin{gather}\label{heq}
			\left\{H_0, G\right\}  + F = Z.
		\end{gather}
		Furthermore there exist $\tau(\br)$ and $\gamma(\br)$ s.t.
		\begin{gather}\label{stimaGen}
			\norm{G}^{\nu,N}_{R}\le\frac{K^\tau}{\gamma} \norm{F}^{\nu,N}_{R}\,,
		\end{gather}
		and $Z \in L^{\nu,N}_{l,\br}$ is in normal form  according to Def. \ref{nor.for.gen} and fulfills the estimate
		\begin{gather*}
			\norm{Z}^{\nu,N}_{R}\le\norm{F}^{\nu,N}_{R}\,.
		\end{gather*}
	\end{lemma}
	\begin{proof}
		Writing
		\begin{gather*}
			F(u) = \sum_{l=3}^{\br} \sum_{\textbf{A}\in \Lambda_e^l} \widetilde{F}_l(\Pi_{A_1}u, \dots, \Pi_{A_l}u)\,,
		\end{gather*}
		we define $G,Z$ through their multilinear map. 
		More precisely, recalling Definition \ref{nRindex},
		we set, for $3\le l\le \br$,
		\begin{equation*}
			\begin{aligned}
				\widetilde{Z_l}(\Pi_{A_1}u_1, \dots, \Pi_{A_l}u_l)&:=
				\left\{
				\begin{aligned}
					&\widetilde{F}_l(\Pi_{A_1}u_1, \dots, \Pi_{A_l}u_l)\quad \text{if} \left(A_1,\dots,A_l\right) \not\in \mathcal{I}^K(l)
					\\
					&0\qquad\qquad\qquad\qquad {\rm otherwise}
				\end{aligned}\right.
				\\
				\widetilde{G}_l(\Pi_{A_1}u_1, \dots, \Pi_{A_l}u_l)&:=
				\left\{
				\begin{aligned}
					&\frac{\widetilde{F}_l(\Pi_{A_1}u_1, \dots, \Pi_{A_l}u_l)}{\sum_{j=1}^{l}\sigma_j\omega_{a_j}}
					\quad \text{if} \left(A_1,\dots,A_l\right) \in \mathcal{I}^K(l)
					\\
					&0 \qquad\qquad\qquad\qquad \text{otherwise}\,.
				\end{aligned}
				\right.
			\end{aligned}
		\end{equation*}		
		These polynomials solve the homological equation \eqref{heq} since
		\begin{gather*}
			\{H_0, G\}(u)= - \sum_{l=3}^{\br} \sum_{\textbf{A}\in\mathcal{I}^K(l)} \tilde{F}(\Pi_{A_1}u, \dots, \Pi_{A_l}u)\,.
		\end{gather*}
		In particular $\widetilde{Z}$ is in normal form (recall Remark \ref{rmk:giallo}).
		The estimate \eqref{stimaGen} follows from Lemmas \ref{non.riso.uno} and \ref{non.riso.due}.
	\end{proof}
	
	\subsection{Proof of the normal form Lemma}\label{sec:itero}
	In this subsection, we complete the proof of Proposition \ref{finalNormal}.

	Fix  $\br\ge2$ and Taylor expand $P$ at order $\br+2$. Recalling \eqref{startHam}
	and Definition \ref{locafunc},
	we have
	\begin{gather*}
		P= P^{(0)} + \mathcal{R}_{T,0}\,. 
	\end{gather*}
	with $P^{(0)} \in L^{N,\nu}_{1,\br+2}$ for some positive
	$N,\nu$. Furthermore
	$\mathcal{R}_{T,0}$ has a zero of order $\br+2$ at $u=0$.
	Moreover, since $ X_{{R}_{T,0}} \in
	\mathcal{C}^\infty(\mathcal{B}^s_{R},\H^s_e) $, one has, for
	$s$ large enough and $R>0$ small enough
	\[
	\norm{X_{\mathcal{R}_T}(u)}_{s}\lesssim  R^{\br+2}\,,\;\;\; \forall u\in\mathcal{B}^s_{R}\,,
	\] 
	In the following lemma, we describe the iterative step that proves Theorem \ref{normalForm}.
	\begin{lemma}[Iteration]\label{iteration}
		Fix $\br \ge 2$; for any $0 \le k \le \br  $, there is a small constant
		$\mu_{k} >0$ s.t., denoting  
		\begin{gather*}
			\mu=\mu(R)\coloneqq
			\frac{\|{P^{(0)}}\|^{N,\nu}_{R}}{R ^2} K^\tau\,,
		\end{gather*}
		if $R_k,K$ are s.t.
		\begin{equation}
			\label{stimu}
			\mu(R_k)\leq \mu_k
		\end{equation}
		then there exist $N_k,\nu_k>0$
		such that, $\forall s_0 > d/2 + \nu_k$ 	the following holds: there
		exists an invertible canonical transformation 
		$T^{(k)}:\mathcal{B}^{s_0}_{R_k/2^k}\rightarrow \mathcal{B}^{s_0}_{R_k}$ such that
		\begin{gather*}
			H^{(k)} = H \circ T^{(k)} = H_0 + Z^{(k)} + P^{(k)} + \mathcal{R}_{T}^{(k)}
		\end{gather*}
		and,
		$\exists R_k>0$ s.t. one has
		\begin{itemize}
			\item $Z^{(k)} \in L^{N_{k-1},
				\nu_{k-1}}_{3,\br_{k-1}+2} $ is in normal form and
			\begin{gather}\label{UNOM}
				\norm{Z^{(k)}}^{N_{k-1},\nu_{k-1}}_{R}
				\lesssim_k \|
				{P^{(0)}}\|^{N,\nu}_{R}\ ,\quad
				\forall R<R_k\ ,
			\end{gather}
			\item $P^{(k)} \in L^{N_k, \nu_k}_{k+3,\br_k+2}$ and
			\begin{gather}\label{DUE}
				\norm{ P^{(k)}}^{N_{k},\nu_{k}}_{R}
				\lesssim_k
				\mu^{k}\|{P}^{(0)}\|^{N,\nu}_{R} \ ,\quad
				\forall R<R_k\ ,
			\end{gather}
			\item $\forall s\geq s_0$ $\exists R_{s,k}$
			s.t. $T^{(k)}\in
			C^{\infty}\left({\cB^s_{\frac{R_{s,k}}{2^k}}};\H^s_e\right)
			$ and
			$[T^{(k)}]^{-1}\in C^{\infty}\left({\cB^s_{\frac{R_{s,k}}{4^k}}};\H^s_e\right)$      \begin{align}
				\label{stime defo}
				T^{(k)}(\cB^s_{R/2^k})\subset \cB^s_{R}\ ,\quad
				[T^{(k)}]^{-1}(\cB^s_{R/4^k})\subset \cB^s_{R/2^k}\ ,\quad
				\forall R<R_{s,k}\ , 
			\end{align}
			\item $X_{\mathcal{R}^{(k)}_{T}}\in
			\mathcal{C}^\infty(\mathcal{B}_{R_{s,k}/2^k}^s,\H^s_e)$
			and,  $\forall u\in\mathcal{B}_{R}^s$, with
			$R<R_{s,k}/2^k$ one has 
			\begin{gather}\label{CINQUE}
				\norm{X_{\mathcal{R}^{(k)}_{T}}(u)}_{s}
				\lesssim_k R^2(K^\tau R)^{\br}
			\end{gather}
		\end{itemize}
	\end{lemma}
	\begin{proof}
		The result is true for $k=0$ with $T^{(0)} = \mathds{1}$ and $N_0 = N$, $\nu_0 = \nu$. 
		We prove the inductive step $k\rightsquigarrow k+1$.	\\
		We determine $G_{k+1}, Z_{k+1}\in
		L^{N_k,\nu_k}_{k+3,\br_{k+1}} $ solving the homological
		equation.  
		\begin{gather}\label{homologicalEquation}
			\{ H_0,G_{k+1} \} +P^{(k)} =  Z_{k+1}\,,
		\end{gather}
		To this end, we remark that the maximal
		degree homogeneity of $P^{(k)}$ appears when $k=\br-1$ and is
		smaller than $\br_{\br}$. So we take $\tau$ and
		$\gamma$ in \eqref{stimaGen} to be $\tau(\br_{\br})$
		and $\gamma(\br_{\br})$.
		Then we write
		\begin{equation*}
			H^{(k+1)} \coloneqq H^{(k)} \circ \Phi_{G_{k+1}}\,.
		\end{equation*}
		\\
		Exploiting \eqref{homologicalEquation}, the Hamiltonian $H^{(k+1)}$ has the form, 
		\begin{align*}
			H^{(k+1)} &= H_0 + Z^{(k)} + Z_{k+1}  + \\
			& + H_0' -\{H_0,G_{k+1}\} +  (P^{(k)})' + \left(Z^{(k)}\right)'  + \\ 
			& + \mathcal{R}_{H_0,G_{k+1}} +
			\mathcal{R}_{P^{(k)},G_{k+1}} +  \mathcal{R}_{Z^{(k)},G_{k+1}} 
			+ \mathcal{R}_{T,k}\circ \Phi_{G_{k+1}}\, ,
		\end{align*}
		with the primed quantities defined as in Lemma \ref{perturbative}.
		Collecting the terms above, we define the following quantities:
		\begin{align*}
			Z^{(k+1)} &\coloneqq Z^{(k)} + Z_{k+1}\,,
			\\
			P^{(k+1)} &\coloneqq H_0' -\{H_0,G_{k+1}\} +  (P^{(k)})' + \left(Z^{(k)}\right)'\,, 
			\\
			\mathcal{R}_{T,k+1}& \coloneqq \mathcal{R}_{H_0,G_{k+1}} +
			\mathcal{R}_{P^{(k)},G_{k+1}} +  \mathcal{R}_{Z^{(k)},G_{k+1}} 
			+ \mathcal{R}_{T,k}\circ \Phi_{G_{k+1}}\,.
		\end{align*}
		First, we check the order of the polynomials, we have $(P^{(k)})'\in
		L_{2k+2,\bar r_{k} +n_1\bar r_{k}}$, with a suitable $n_1$.
		Similarly one has
		$
		H_0' -\{H_0,G_{k+1}\}\in L_{2k+2,\bar r_{k} +n_2\bar r_{k} }\ ,
		$
		and $(Z^{k})'\in L_{3+k+1,\br_{k-1}+n_3\br_{k}}$ with suitable
		$n_2,n_3$. Defining $$\br_{k+1}:=\max\left\{\br_{k-1}+n_3\br_{k},\bar
		r_{k} +n_2\bar r_{k}, \bar r_{k} +n_1\bar r_{k} \right\} $$ one gets
		the result on the order.
		
		We come to the estimates. First we remark that
		\begin{equation}
			\label{murk}
			\mu\lesssim RK^\tau\ .
		\end{equation}
		From Lemma \ref{homeq} and the inductive hypothesis, we get the estimates
		\begin{equation}\label{newHomEq}
			\begin{aligned}
				\norm{G_{k+1}}^{N_k,\nu_k}_{R} &\le \frac{K^\tau}{\gamma} \|P^{(k)}\|^{N_k,\nu_k}_{R} 
				\lesssim 
				\frac{K^\tau}{\gamma}\mu^{k} \|{P^{(0)}}\|^{N,\nu}_{R} \lesssim \mu^{k+1}R^2\,,
				\\ 
				\norm{Z_{k+1}}^{N_k,\nu_k}_{R} &\lesssim \|P^{(k)}\|^{N_k,\nu_k}_{R}
				\lesssim\mu^{k} \|{P}^{(0)}\|^{N,\nu}_{R}\,.
			\end{aligned}
		\end{equation}		
		The estimate \eqref{UNOM} follows from \eqref{newHomEq}.
		\\
		For the estimate \eqref{DUE}, we prove it for 
		$\tilde{H}_0' -\{H_0,G_{k+1}\}$ applying Lemma \ref{PoissonBrack}. 
		Denoting $N' = N_{k+1} = N_k -{n}(\nu_k + d)$ and 
		$\nu' = \nu_{k+1} = {n}(2\nu_k + d)$, and 
		taking profit of \eqref{homologicalEquation}, 
		we have
		\begin{equation*}
			\begin{aligned}
				&\norm{H_0' -\{H_0,G_{k+1}\} }^{N_k,\nu_k}_{R} 
				= \norm{\sum_{l=1}^{{n}}\frac{1}{l!}Ad^l_{G_{k+1}}\left(\{H_0,G_{k+1}\}\right)}^{N_k,\nu_k}_{R}  
				\\&\qquad 
				= \norm{\sum_{l=1}^{{n}}\frac{1}{l!}Ad^l_{G_{k+1}}\left(Z_{k+1}-P^{(k)}\right) }^{N_k,\nu_k}_{R} 
				\\&\qquad 
				\lesssim\sum_{l=1}^{{n}}\frac{1}{l!}
				\left(\frac{\norm{G_{k+1}}^{N_k,\nu_k}_{R}}{R^2}\right)^l\norm{P^{(k)}}^{N_k,\nu_k}_{R} 
				\lesssim \mu^{k}
				\norm{{P}^{(0)}}^{N,\nu}_{k}\sum_{l=1}^{{n}}\frac{1}{l!}\mu^{(k+1)l}
				\\
				&\qquad
				\lesssim \mu^{k}
				\norm{P^{(0)}}^{N,\nu}_{R} \mu^{k+1}\sum_{l=0}^{{\infty}}\frac{1}{l!}\mu^{(k+1)l} 
				\lesssim \mu^{2k+1} \norm{\hat{P}}^{N,\nu}_{R}\,,
			\end{aligned}
		\end{equation*}
		that is \eqref{DUE} with $k$ replaced by $k+1$ (also
		in the case $k=0$). 
		
		Concerning $\left(Z^{(k)}\right)'$ we just use
		\eqref{tesi.lie.2} which gives 
		$$
		\left\|\left(Z^{(k)}\right)'\right\|^{N_k,\nu_k}_R\lesssim
		\left\|P^{(0)}\right\|_R^{N,\nu}\mu^{k+1}\ . 
		$$
		Similarly one gets the estimate of
		$\left(P^{(k)}\right)'$.
		
		We come to the remainders. Consider first
		$\cR_{Z^{(k)},G_{k+1}}$. Here one has $Z^{(k)}\in L_{3,\br_{k-1}}$, so
		that we apply \eqref{tesi.lie.3} with $r_1=k+1$, $r_2=3$, which gives
		$$n+1\geq \frac{\br}{k+1}\ ,$$ so that the remainder is estimated by
		$$
		\left\|X_{\cR_{Z^{(k)},G_{k+1}}   }\right\|\lesssim
		\frac{\left\|Z^{(k)}\right\|^{N_k,\nu_k}_{R}}{R}(\mu^{k+1})^{n+1}\lesssim
		\frac{\left\|P^{(0)}\right\|_{R}^{N_k,\nu_k}}{R}\mu^{\br}\lesssim
		R^2(RK^\tau)^{\br}\ .
		$$
		Proceeding in the same way for the other terms one gets the thesis.
	\end{proof}
	
	We thus have the following:
	
	\begin{corollary}[Normal form]\label{normalForm}
		Assume Hypotheses \ref{asymptotic}, \ref{hypoNonRes} and
		\ref{bourgain.abstract} for the Hamiltonian $H = H_0 + P$ of the form
		\eqref{startHam} and that $P$ is a function with localized
		coefficients according to Definition \ref{locafunc} and with a zero of
		order at least $3$ at the origin.  For any $\br>3$ there exist $\tau$
		and $s_{\bar{r}}\geq s_0>d/2$ such that for any $s\geq s_{\bar{r}}$
		the following holds. 
		There exist constants $R_{\br,s}$ and $C_{\br,s}$ such that, if $K$
		and $R$ fulfill			\begin{gather*}
			RK^\tau \le R_{\br,s}\,,
		\end{gather*} 
		then there exists an invertible canonical transformation 
		\begin{align*}
			T^{(\br)}:\mathcal{B}^s_{R/2^{\br}}\rightarrow
			\mathcal{B}^s_{R}\ ,\quad \left[T^{(\br)}\right]^{-1}:\mathcal{B}^s_{R/4^{\br}}\rightarrow
			\mathcal{B}^s_{R/2^{\br}}
		\end{align*}
		such that
		\begin{equation}\label{hamHrrNormal}
			H^{(\br)} = H \circ T^{(\br)} = H_0 + Z^{(\br)} +  \mathcal{R}_T
		\end{equation}
		where $Z^{(\br)}\in L^{\nu_{\br},N_{\br}}_{3,\br_{\br}}$ is in normal form and  
		$X_{\mathcal{R}_T} \in \mathcal{C}^\infty(\mathcal{B}^s_{R/2^{\br}},\H^s_e)$. 
		Moreover, for any $u\in\mathcal{B}^s_{R}$, we have		
		\begin{gather*}
			\norm{X_{\mathcal{R}_T}(u)}_{s}\lesssim
			R^2(K^\tau R)^{\br}\ .
		\end{gather*}
	\end{corollary}
	\begin{proof}[Proof of Proposition \ref{finalNormal}]
		We apply Corollary \ref{normalForm}.  Choosing $K$ in such a way that
		\begin{gather}\label{sim}
			R K^\tau \sim R^{\frac{1}{2}}
		\end{gather}
		we have
		\begin{gather*}
			(RK^{\tau})^{\br}R^2\simeq R^{\frac{\br+4}{2}}.
		\end{gather*}
		Choosing $\br := 2r$ we get
		\begin{gather*}
			\norm{X_{\mathcal{R}_T}(u)}_{s}\lesssim R^{r+2}\,.
		\end{gather*}
		From Theorem \ref{normalForm}, Definitions \ref{brnf} and \ref{brnf3}, we can write 
		\begin{gather*}
			Z^{(\br)} = Z_0 + Z_B + Z_2 + Z_{\ge3}\,.
		\end{gather*}
		We have to show that $Z_2$ and $Z_{\ge3}$ can be considered remainder 
		terms of order $R^{r+1}$. 
		From Corollary \ref{cubic}, we have that, for any  $u\in\mathcal{B}_{R/2^{\br}}^s$,
		\begin{gather*}
			\norm{X_{Z_{\ge3}}(u)}_s \lesssim 
			\frac{\norm{Z_{\ge3}}^{\nu,N}_R}{K^{s-s_0}}\frac{1}{R} \lesssim \frac{R^2}{K^{s-s_0}}\,.
		\end{gather*}
		Since \eqref{sim} implies $K\sim R^{-\frac{1}{2\tau}}$, we have 
		\begin{gather*}
			\frac{R^2}{K^{s-s_0}} \sim R^ {2 + \frac{s-s_0}{2\tau}} \lesssim R^{r+1}
		\end{gather*}
		if $s>s_r = s_0 + 2\tau(r - 1)$ and so
		\begin{gather*}
			\norm{X_{Z_{\ge3}}(u)}_s \lesssim R^{r+1}\,.
		\end{gather*}
		It remains to consider $X_{Z_2}$. From Lemma \ref{cutOff2}, 
		it follows that, for any  $u\in\mathcal{B}_{R}^s$,
		\begin{gather*}
			\norm{X_{Z_2}(u)}_s \lesssim  
			\frac{\norm{Z_2}^{\nu,N'}_R}{K^{\delta (N'-N)}} \frac{1}{R} 
			\lesssim \frac{R^2}{K^{\delta(N'-N)}}\,,
		\end{gather*}
		for any $N' > N$. 
		Denoting $N' = N + M_1$, we have
		\begin{gather*}
			\frac{R^2}{K^{\delta M_1}}  \sim R^{2 + \frac{\delta M_1}{2\tau}} \sim R^{r+1}
		\end{gather*}
		for $M_1 = \frac{2\tau}{\delta}(r-2)$. We get the thesis denoting 
		\begin{gather*}
			\cR^{(\br)} \coloneqq Z_2 + Z_{\ge 3} + \mathcal{R}_T.
		\end{gather*}
		The estimate \eqref{lowmodesZ} follows from Lemma \ref{cubic}.
	\end{proof}
	
	\section{From normal form to almost global existence}\label{almost.sec}
	In this section, we conclude the proof of Theorem \ref{LC}. In particular, we use Proposition \ref{finalNormal} to study the dynamics
	corresponding to a smooth, small, real initial
	datum for the Hamilton equations of
	\eqref{startHam}. Precisely, take some large $s$ and assume that 
	\begin{equation}
		\label{ini.sec.al}
		\left\| u_0\right\|_s=:\epsilon< \frac{R}{2 \cdot 4^{2r}}\ ,
	\end{equation}
	with a small $R<R_{s,r}$ and $R_{r,s}$ from Proposition \ref{finalNormal}. Denote
	\begin{gather*}
		z_0 \coloneqq \left(T^{(\br)}\right)^{-1}(u_0)\,,
	\end{gather*}
	and we consider the evolution in the $z$ variables.
	Let $z = (z_+,z_-)\in \H^s\cross\H^s \equiv \H^s_e$ and, for 
	$K$ defined in Proposition \ref{finalNormal}, recall that
	\begin{align*}
		\Pi^{\le}z  &= z^\le =   \sum_{\left\{\len{A}\le K\right\}}\Pi_A z\,,
		\\  
		\Pi^{\perp}z &=  z^\perp =   \sum_{\left\{\len{A}>K\right\}}\Pi_A z\,.
	\end{align*}
	Write the Hamilton equations of $H^{(\br)}$ as a system:
	\begin{gather}
		\begin{cases}\label{stimaAltaBassa}
			\dot{z}^\le = X_{H_0}(z^\le) + X_{Z_0}(z^\le) + \Pi^\le X_{Z_B}(z) + \Pi^\le X_{R^{(\br)}}(z)\,,\\
			\dot{z}^\perp = X_{H_0}(z^\perp) +  \Pi^\perp X_{Z_B}(z) + \Pi^\perp X_{R^{(\br)}}(z)\,.
		\end{cases}
	\end{gather}
	
	We start considering the dynamics of $z^\le$. 
	We remark first that $\Pi^\le X_{Z_B}$  and $\Pi^\le X_{R^{(\br)}} $ 
	are remainder terms of order $R^{r+2}$ (see Proposition \ref{finalNormal}). 
	Then, it remains to analyze the role of $Z_0$.\\ 
	To this end, we define the set of indexes correspondent 
	to each "band" of the spectrum $\Sigma_n$, 
	the correspondent projector and the correspondent "superaction", namely
	\begin{align*}
		&E_n \coloneqq \{a\in\Lambda: \omega_a \in \Sigma_n\}\,,\qquad
		\Pi_n z \coloneqq \sum_{\left\{A=(a,\sigma):a \in E_n\right\}}  \Pi_A z \,,\\
		&J_n(z) \coloneqq  \sum_{a\in E_n}\int \Pi_{(a,+)}z\,\Pi_{(a,-)}z\, dx\,.
	\end{align*}
	In particular, if $z$ is real (as we assumed)
	\begin{gather*}
		J_n(z) = \sum_{a\in E_n}\norm{\Pi_{(a,+)}z}_0^2 
		= \frac{1}{2}  \sum_{a\in E_n} \norm{\Pi_{(a,+)}z}_0^2 + \norm{\Pi_{(a,-)}z}_0^2\,.
	\end{gather*}
	In the next lemma, we prove that $J_n$ is preserved under 
	the dynamics associated to $Z_0$. 
	\begin{lemma}\label{enBloLem}
		Let $Z$ a polynomial supported on $W$, then
		\begin{gather*}
			\{Z,J_n\} = 0\,.
		\end{gather*}
	\end{lemma}
	\begin{proof}
		From Def. \ref{support} follows that $\widetilde{Z}$ is the sum of terms of the form 
		\begin{gather*}
			\int_M \Pi_{A_1}u_1 \dots \Pi_{A_l}u_l
		\end{gather*} 
		with $\textbf{A} = (A_1,\dots,A_l)\in W$. First, note that we have
		\begin{gather*}
			X_{J_n}(u) = i \sum_{a\in E_n}\left(\Pi_{(a,+)} u, - \Pi_{(a,-)}u
			\right)=\im
			\sum_{(a,\sigma)}\delta_{a\in\Sigma_n}(\delta_{\sigma,+}-\delta_{\sigma,-})\Pi_{(a,\sigma)}u
			\,.
		\end{gather*}
		Moreover, for any homogeneous polynomial $F$, one has
		\begin{gather*}
			\text{d}F(u) X = \tilde{F}(X,u,\dots,u) +
			\tilde{F}(u,X,\dots,u) + \dots +
			\tilde{F}(u,\dots,u,X)\ .
		\end{gather*}
		Decompose $Z=\sum_{l}Z_l$ in homogeneous polynomials,  we have
		\begin{gather*}
			\{Z_l,J_n\}(u) = \text{d}Z(u) X_{J_n}(u) =
			\\ =
			\im  \sum_{A_1,...,A_l}\tilde
			Z(\Pi_{A_1}u,...,\Pi_{A_l}u)\sum_{j=1}^n\delta_{a_j\in\Sigma_n}(\delta_{\sigma_j,+}-\delta_{\sigma_j,-})\ .
		\end{gather*}
		We recall that $Z$ is in normal form, which means that the sum can be
		restricted to multi-indexes belonging to $W$. So, fix one of the
		multi-indexs $A\in W$. The definition of $W$ implies that there exists
		a permutation $\tau$ of $1,...,l$, and indexes $n_1,...,n_l$ s.t,
		$a_{\tau(j)},a_{\tau(j+l/2)}\in\Sigma_{n_j}$ and
		$\sigma_{\tau(j)}\sigma_{\tau(j+l/2)}=-1$. Thus consider the sum
		$$
		\sum_{j=1}^n\delta_{a_j\in\Sigma_n}(\delta_{\sigma_j,+}-\delta_{\sigma_j,-})=
		\sum_{j=1}^n\delta_{a_{\tau(j)}\in\Sigma_n}(\delta_{\sigma_{\tau(j)},+}-\delta_{\sigma_{\tau(j)},-})
		\ .
		$$
		If $a_{\tau(j)}\in\Sigma_n$ with a sign, it means that also
		$a_{\tau(j+l/2)}\in\Sigma_n$ with the opposite sign. Thus the sum
		vanishes for any index in $W$.
	\end{proof}
	
	\begin{corollary}
		\label{bassi}
		There exists a positive constant $C_1$ with the following property:
		assume that \eqref{ini.sec.al} holds and that there exists $T_e>0$ s.t.
		\begin{equation}
			\label{te}
			\left\| z(t)\right\|_s\leq \frac{R}{2\cdot 2^{2r}}\ ,\quad \forall
			\left|t\right|\leq T_e
		\end{equation}
		and some $R<R_{rs}$, then one has
		\begin{equation}
			\label{aimUno}
			\norm{z^\le(t)}^2_s \leq C_1( 	\norm{z^\le(0)}^2_s + |t|R^{r+3})\,.\end{equation}
	\end{corollary}
	\begin{proof}
		Define
		\begin{gather*}
			a_{n,s} \coloneqq \inf_{a\in E_n} (1 + \len a)^{2s}
		\end{gather*}
		then there exist two constants $C_3,C_4$ such that, for any $n$, one has
		\begin{gather*}
			C_3 \norm{\Pi_n z}_s^2 \le a_{n,s} J_n \le C_4  \norm{\Pi_n z}_s^2\,,
		\end{gather*}
		
		\begin{equation}\label{equiv.Jn}
			\|z^{\leq}\|^2_{s} \simeq \sum_{n \text{ s.t. } \atop \max E_n \leq K} a_{n,s} J_n\,.
		\end{equation}
		Then, by Proposition \ref{finalNormal} and Lemma \ref{enBloLem},
		one has
		\begin{align}\label{st.en.low}
			\frac{d}{dt}\sum_{n \text{ s.t. } \atop \max E_n \leq K} a_{n,s} J_n  \lesssim \sum_n  a_{n,s} 
			\left\{J_n, Z_2 + R^{(\br)} \right\} \lesssim R^{r+3}\,,
		\end{align}
		which is valid for $|t|\leq T_e$. From \eqref{equiv.Jn} and \eqref{st.en.low} the thesis immediately
		follows.
	\end{proof}
	
	Consider the dynamics of the high modes $z^\perp$ given by
	\eqref{stimaAltaBassa}.  From Proposition \ref{finalNormal}),
	$\Pi^\perp X_{R^{(r)}} $ is a remainder term of order
	$R^{r+2}$ and $Z_B$ is in Block Resonant Normal Form and of
	order 2 in $z^\perp$.
	
	Recalling the dyadic partition
	$\Lambda = \bigcup_{\alpha\in\mathfrak{A}} \Omega_\alpha$, we
	define the correspondent projectors and the correspondent
	superactions.  Since we are interested in the dynamics of the
	high modes, we consider only superactions defined on modes
	$|a|\ge K$.  This amount to consider a cutoff of the
	Bourgain's blocks, that clearly do not break the dyadicity of
	the partition.
	\begin{align*}
		\Pi_\alpha z  &\coloneqq
		\sum_{\left\{A=(a,\sigma)\colon a\in \Omega_\alpha,\,
			\len a\ge K \right\}} \Pi_A z
		\\	
		J_\alpha(z) &\coloneqq
		\sum_{\left\{a\in\Omega_\alpha,\, \len a\ge K \right\}}
		\int \Pi_{(a,+)}z\,\Pi_{(a,-)}z\,dx = \norm{\Pi_\alpha z}_0^2\,.
	\end{align*}
	By definition of Block Resonant normal form (see
	Def. \ref{brnf}) $\Pi^\perp X_{Z_B}$ is linear in
	$z^\perp$. Then, we show that the $L^2$ norm on each block is
	conserved along the dynamics associated to the normal
	form since $Z_B$ is a real polynomial. Namely, the dynamics of the normal form $Z_B$ enforces
	the exchange of energy only within high modes in the same block
	$\Omega_\alpha$. This is the content of the following lemma.
	\begin{lemma}
		For any real $z \in \H_e^s$, we have that, for any $\alpha \in \mathfrak{A}$, 
		\begin{gather*}
			\left\{J_\alpha, Z_B \right\} (z) = 0\,.
		\end{gather*}
	\end{lemma} 
	\begin{proof}
		In this proof, we denote again $u_A \coloneqq \Pi_A u$.
		Since $Z_B$ is real and recalling Def. \ref{brnf}, it is the sum of terms of the form 
		\begin{equation}\label{ex.proof}
			\begin{aligned}
				\widetilde {Z}_\beta(z)&:= \int z_{(a,+)}  z_{(b,-)} z_{A_3}\dots z_{A_l} 
				+ \int z_{(a,-)}  z_{(b,+)}z_{\overline A_3}\dots z_{\overline{A}_r} 
				\\&
				= 2 \text{Re}\left(\int z_{(a,+)}  z_{(b,-)}z_{A_3}\dots z_{A_r} \right),\quad \forall z \text{ real}
			\end{aligned}
		\end{equation}
		with $\textbf{A} =
		\left((a,+),(b,-),A_3,\dots,A_r\right) \in
		\Lambda_e^r$ such that $a,b \in \Omega_\beta$, $\len
		a>K$ and $\len b>K$ for some $\beta$.  Here, for the
		sake of simplicity, we are considering multi-indexes
		for which the two largest indexes are the first and
		the second.
		
		If $\beta \not = \alpha$, then
		$\left\{J_\alpha,Z_\beta\right\} = 0$. Otherwise, we have
		\begin{align*}
			\left\{J_\alpha, \widetilde{Z}_\alpha\right\}(z)&=   
			i \int z_{(a,-)} z_{(b,+)}\dots z_{\bar{A_r}} - i\int z_{(a,+)} z_{(b,-)}\dots z_{A_r} 
			\\&
			+ i \int z_{(b,-)}z_{(a,+)}  \dots z_{A_r}  - i\int z_{(b,+)}z_{(a,-)}  \dots z_{\bar{A_r}} = 0\,,
		\end{align*}
		with $a,b\in\Omega_\alpha$.
	\end{proof}
	The following Corollary is proved exactly in the same way as Corollary
	\ref{bassi} 
	\begin{corollary}
		\label{alti}
		There exists a positive constant $C_1$ with the following property:
		assume there exists $T_e>0$ s.t. \eqref{te} holds for some $R<R_{rs}$,
		then one has
		\begin{equation}
			\label{aimDue}
			\norm{z^\perp(t)}^2_s \leq C_1( 	\norm{z^\perp(0)}^2_s + |t|R^{r+3})\,.\end{equation}
	\end{corollary}
	
	\noindent{\it Proof of Theorem \ref{LC}.} We prove by a
	bootstrap argument
	that, if
	$\epsilon$ is small enough,  the escape time $T_e$ fulfills $
	1/R^{r+1}\lesssim T_e$. First we make the canonical transformation $\cT^{(r)}$ and
	apply the estimate \eqref{defo.final} with $R = 2^{-1} \epsilon 4^{2r}$ where $\epsilon := \|u_0\|$, getting
	that
	$$
	\left\|z_0\right\|_s\leq \frac{R}{2 \cdot 2^{2r}}\ .
	$$
	Then we can apply Lemmas \ref{bassi} and Lemma \ref{alti}, since the assumptions on
	the initial datum $z(0)$ are fulfilled. Assume now, by contradiction,
	that there exists $t_*< T_e := R^{-r-1} (2 \cdot 2^{2r})^{-2}$ s.t. $z(t)\in
	\cB^s_{R'_r}$ for all $|t|<t_*$ and $z(t_*)\in
	\partial\cB^s_{R'_r}$, with $R'_r:=\frac{\sqrt{2 C_1}R}{2\cdot2^{2r}}$ and $C_1$ as in Lemmas \ref{alti}, \ref{bassi}. For $|t|\leq t_*$, Lemmas
	\ref{bassi} and \ref{alti} give that
	$$
	\left(\frac{\sqrt{2 C_1}R}{2\cdot2^{2r}}\right)^2=\left\|z(t_*)\right\|_s^2\leq
	C_1(\left\|z_0\right\|^2_s+R^{r+3}|t|)\leq
	C_1\left(\frac{R^2}{2^2\cdot 2^{4r}}+\frac{R^2}{2^2\cdot 2^{4r}}\right)\ ,  
	$$ which is absurd. Going back to the variables $u$, changing $r+1$ to
	$r$ and adjusting the name of the constants one gets the result. \qed
	%

	\section{Proof of the results on the applications}\label{pro.appli}
	In this Section, we prove the results stated in
	Sect. \ref{applications}, concerning applications of
	Thm. \ref{ab.res}. We recall that we are considering manifolds
	in which the Laplacian is a globally integrable quantum
	system. In
	particular, we denote by $I_1,\dots,I_d$ the quantum actions of the Laplacian  and by $h_{L0}$ the function such that
	\begin{equation*}
		-\Delta = h_{L0}(I_1,\dots,I_d)	\,.
	\end{equation*}
	Moreover, we denote by $\left\{\lambda_a\right\}_{a\in\Lambda}$ 
	the eigenvalues of $-\Delta$, namely
	\begin{gather*}
		\lambda_a = h_{L0}(a),\qquad \forall a\in \Lambda\,.
	\end{gather*}
	As emphasized in Sect.\ref{applications} they 
	fulfill Hypotheses \ref{integro}, \ref{asymptotic} and
	\ref{bourgain.abstract}. 
	
	\medskip
	\noindent
	In the applications, we apply Thm. \ref{ab.res} to an
	Hamiltonian with linear part $H_L$ that is a relatively
	compact perturbation of $-\Delta$; for each of them we will
	check that they still fulfill Hypotheses \ref{integro},
	\ref{asymptotic} and \ref{bourgain.abstract}. We remark that
	considering small perturbations of the Laplacian is
	fundamental in order to verify the nonresonance assumption
	\ref{hypoNonRes}.

	Finally, we remark that all the nonlinear
	perturbations that we will meet are of the form \eqref{nonlin.1}, so they 
	fulfill Hypothesis \ref{Nonline}. In other words, following Theorem \ref{lo.generale}, 
	they are functions with localized coefficients.

	\subsection{Schr\"odinger equation with convolution potentials
	}
	The nonlinear Schr\"odinger equation 
	\eqref{convoluzione}
	\begin{equation}\label{conv2}
		\ii \partial_t \psi = - \Delta \psi + V \sharp \psi + f(x,|\psi|^2)\psi , \quad x \in M,
	\end{equation}
	is Hamiltonian with Hamiltonian function
	\begin{equation*}
		H=\int_{M}\left(\bar \psi (-\Delta\psi)+\bar \psi(V\sharp\psi)+F(x,|\psi|^2)\right)dx
	\end{equation*}
	where $F$ is such that $f(x,u)=\partial_u F(x,u)$. 
	
	\noindent
	The Hamiltonian of Equation \eqref{conv2} fits the abstract settings 
	\eqref{H0.s} with $H_L$ the globally integrable quantum system
	\begin{align*}
		H_L = -\Delta + V\sharp\,.
	\end{align*}
	The actions are $I_1,\dots,I_d$, namely they are the actions of the Laplacian; 
	the associated function $h_L$ is 
	\begin{equation*}
		\r^d\ni \xi \mapsto h_{L0}(\xi) + v(\xi)\,,
	\end{equation*}
	where $v(\xi)$ is any $C^\infty$ interpolation of $V\sharp$ on the lattice $\Lambda$, 
	namely it is a function such that $v(\xi) = V_\xi, \forall \xi \in \Lambda$. 
	Moreover, we remark that the frequencies are given by
	\begin{equation*}
		\omega_{a}:=\lambda_a+V_a\ .
	\end{equation*}
	In order to apply Theorem \ref{ab.res} and prove Theorem \ref{nsl.esti}, it remains to verify
	that Hypotheses \ref{asymptotic},
	\ref{bourgain.abstract} and \ref{hypoNonRes} hold.
	
	\medskip
	\noindent    
	Hypotheses \ref{asymptotic} and \ref{bourgain.abstract} clearly hold, 
	since they hold for $\left\{\lambda_a\right\}_{a\in\Lambda}$ and 
	the coefficients $V_a$ have strong decay (see \eqref{spaz.V}). 
	Indeed, for any $a,b \in \Lambda$,
	\begin{gather*}
		|a-b| + |\lambda_a-\lambda_b| \ge C_\delta (|a|^\delta + |b|^\delta)
	\end{gather*}
	implies
	
	\begin{equation}\label{newBourg}
		\begin{aligned}
			|a-b| + |\omega_a-\omega_b| &\ge |a-b| + |\lambda_a-\lambda_b| - |V_a - V_b|
			\\&
			\ge  C_\delta (|a|^\delta + |b|^\delta) - 1/2 \ge  C'_\delta (|a|^\delta + |b|^\delta)\,.
		\end{aligned}
	\end{equation}
	We come to prove the non-resonance condition \ref{hypoNonRes}. 
	Actually, we will prove a stronger condition, namely:
	
	\begin{lemma}\label{tutto.res}
		For any $r$ there exists $\tau$ and a set $\cV^{(res)}\subset\cV$ of zero
		measure, s.t., if $V\in\cV\setminus \cV^{(res)}$ there exists
		$\gamma>0$ s.t. for all $K\geq 1$ one has
		\begin{equation}\label{ikappa}
			\begin{aligned}
				\left|\omega\cdot
				k\right|\geq\frac{\gamma}{K^\tau}\ ,\
				\\
				\forall
				k=(k_{a_1},...,k_{a_r})\ s.t.\ |a_j|\leq
				K\ \forall j=1,...,r\ , |k|_{\ell^1}\leq r  \,.
			\end{aligned}
		\end{equation}
	\end{lemma}
	First consider, for $k$ fulfilling the second of \eqref{ikappa},
	\[
	\cV(k,\gamma):=\left\{V\in\cV_n\ :\ \left|\omega\cdot
	k\right|<\gamma\right\}\, .
	\]
	We have the following lemma.
	\begin{lemma}
		\label{meas.nls}
		One has 
		\begin{equation}
			\label{meas.nls.1}
			\left| \cV(k,\gamma) \right|\leq 2\gamma K^n\ .
		\end{equation}
	\end{lemma}
	\begin{proof}
		Actually, we prove that the Lemma is true for any arbitrary sequence
		$\lambda_a$, namely that the asymptotic behaviour is not important. 
		
		\noindent
		First, if $\cV(k,\gamma) $ is empty there is nothing to
		prove. Assume that $\tilde V\in \cV(k,\gamma)$. 
		Since $k\not=0$, there exists $ \bar a$ such that $k_{\bar a}\not=0$ 
		and thus $\left|k_{\bar a}\right|\geq 1$; so we have 
		\[
		\left|\frac{\partial \omega\cdot k}{\partial \hat V_{\bar a}}\right|\geq 1\, .
		\]
		It means that if $\cV(k,\gamma) $ is not empty it is contained in
		the layer 
		\[
		\left|\widehat{\tilde V_{\bar a}}\null'-\hat V'_{\bar a}\right|\leq \gamma\, ,
		\]
		whose measure is
		$\gamma\langle \bar a\rangle^n\leq 2\gamma N^n$.
		This implies \eqref{meas.nls.1}.
	\end{proof}
	
	\begin{proof}[Proof of Lemma \ref{tutto.res}] 
		From Lemma \ref{meas.nls} it follows that the measure of the set
		\[
		\cV^{(res)}(\gamma):=\bigcup_{K\geq 1}
		\bigcup_{k}
		\cV\left(k,\frac{\gamma}{ K^{dr+2}}\right)
		\]
		is estimated by a constant times $\gamma$. It follows that the set
		\[
		\cV^{(res)}:=\cap_{\gamma>0}\cV^{(res)}(\gamma)
		\]
		has zero measure and with this definition the lemma is proved. 
	\end{proof}

	\subsection{Stability of the ground state of NLS}
	The equation \eqref{nls} is Hamiltonian with Hamiltonian
	\begin{equation}
		\label{Ham.gro}
		H(\psi,\bar \psi) = \int_M \left(\overline{\psi}
		(-\Delta \psi) + F(|\psi|^2)\right) dx \ ,
	\end{equation}
	where $F$ is such that $F'=-f$.
	Following \cite{plane}, we introduce variables in which the
	ground state becomes an equilibrium point of a reduced
	system. To this end, we consider the space
	\begin{equation*}
		L^2_0(M;\C):=\left\{\varphi\in L^2(M;\C)\ :\ \int_M \varphi(x)dx=0\right\}\,,
	\end{equation*}
	of the functions with vanishing average and we denote 
	\begin{gather*}
		N(\varphi) = \int_M \left|\varphi\right|^2\,dx \ .
	\end{gather*}
	Then we consider the map
	\begin{equation}\label{gro.2}
		\begin{aligned}
			L^2_0(M,\C)\times\R\times\T&\to L^2(M;\C)
			\\
			(\varphi,p,\theta)&\mapsto
			\psi(\varphi,p,\theta): = e^{- i\theta}\left(\sqrt{p-\norm{\varphi}^2}+ \varphi(x) \right) \ .
		\end{aligned}
	\end{equation}

	\begin{lemma}[Faou,Lubich,Glouckler \cite{plane}]
		\label{symplettiche}
		The map \eqref{gro.2} defines a local coordinate system close
		to $\varphi=0$. Furthermore, such coordinates are symplectic, namely the Hamilton
		equations of a Hamiltonian function $H$ have the form
		$$
		\dot \theta=\frac{\partial H}{\partial p}\ ,\quad\dot
		p=-\frac{\partial H}{\partial \theta}\ ,\quad \dot z=-\ii \nabla_{\bar
			z}H\ .
		$$
	\end{lemma}
	Then the Hamiltonian is just \eqref{Ham.gro} with $\psi$ given by
	\eqref{gro.2}. We now fix a value $p_0$ of $p$ (which is an integral
	of motion) and expand in power series in $\varphi$ getting (neglecting
	irrelevant terms independent of $\varphi$) a Hamiltonian of the form
	\begin{equation*}
		H=H_0+\hat P
	\end{equation*}
	with
	\begin{equation*}
		H_0(\varphi)=\int_M\left[\bar \varphi (-\Delta \varphi )+2f(p_0)\bar \varphi \varphi +\frac{f(p_0)}{2}\left(\varphi ^2+\bar
		\varphi ^2\right)\right] dx
	\end{equation*}
	and $\hat P=\hat P(\varphi ,\bar \varphi )$ of the form \eqref{nonlin.1}. 
	Thus we have just to
	verify the assumptions on the linear part. 
	
	\medskip
	\noindent 
	Introducing the spectral decomposition relative to the quantum actions of the Laplacian, we get
	\begin{gather*}
		H_0 = \sum_{a\in \Lambda}
		\int_M\left(\left(\lambda_a + 2f(p_0)\right)\Pi_{a}
		\bar \varphi \Pi_a \varphi  + \frac{f(p_0)}{2} \left(\Pi_a
		\varphi \right) ^2 + \frac{f(p_0)}{2} \left(\Pi_a \bar
		\varphi \right) ^2 \right)dx\ ,
	\end{gather*}
	which can be diagonalized through a symplectic change of coordinates
	of the form 
	$$
	\left(
	\begin{matrix}
		w_a\\ \bar w_a
	\end{matrix}
	\right)=S_a\left(\begin{matrix}
		\varphi _a\\ \bar \varphi _a
	\end{matrix}
	\right)
	$$
	where $S_a$ are matrices uniformly bounded with respect to $a$. Such a
	change of coordinates of course does not change the nature of $\hat P$
	of having localized coefficients. In the new coordinates, one gets
	\begin{equation*}
		H_0(w,\overline w)=\sum_{a}\omega_a\int_M\Pi_a\bar w\Pi_a w dx
	\end{equation*}
	with
	\begin{equation}\label{gro.43}
		\omega_a = \sqrt{\lambda_a^2 + 2f(p_0) \lambda_a}\, .
	\end{equation}
	Hypotheses \ref{asymptotic} and \ref{bourgain.abstract} 
	hold for $\left\{\omega_a\right\}_{a\in\Lambda}$, 
	since they hold for $\left\{\lambda_a\right\}_{a\in\Lambda}$
	and one has 
	\begin{gather*}
		\omega_a = \sqrt{\lambda_a^2 + 2 f(p_0) \lambda_a} 
		= 
		\lambda_a\left(1 + \frac{f(p_0)}{\lambda_a} +
		o\left(\frac{f(p_0)}{\lambda_a}\right)\right)
	\end{gather*}
	so that 
	\begin{gather*}
		|a-b| + |\lambda_a-\lambda_b| \ge C_\delta (|a|^\delta + |b|^\delta)
	\end{gather*}
	implies
	\begin{gather*}
		|a-b| + |\omega_a-\omega_b| \ge  C'_\delta (|a|^\delta + |b|^\delta)\,.
	\end{gather*}
	with computations analogous to the ones in \eqref{newBourg}
	
	\noindent

	We prove now the non-resonance condition \ref{hypoNonRes}. 
	
	The proof is essentially identically to that given by Delort and Szeftel in \cite{DS0} 
	(see also \cite{DS}) for the case of the wave equation on tori
	and Zoll Manifolds. 
	In particular, it is obtained by applying the following theorem
	\begin{theorem}\label{analytic}
		[Thm. 5.1 of
		\cite{DS0}]	Let $X$ be a closed ball $\mathcal{B}_{R_0}$ in $\r^p$ for some $R_0>0$ and by $Y$ a
		compact interval in $\r$.\\
		Let $f \colon X \cross Y \mapsto \r$ be a continuous subanalytic 
		function, $\rho \colon X \mapsto \r$ a real analytic function, $ \rho \not \equiv0$. Assume
		\begin{itemize}
			\item $f$ is real analytic on $\{x \in X; \rho(x) \not= 0\} \cross Y $;
			\item for all $x_0 \in X$, the equation $f(x_0, y) = 0$ has only finitely many
			solutions $y \in Y$.
		\end{itemize}
		Then there are $N_0 \in \n, \alpha_0 > 0, \delta > 0, C > 0$, 
		such that for any $\alpha \in (0, \alpha_0)$, any $N\ge N_0$, any $x \in X$ with $\rho(x) \not= 0$,
		\begin{gather*}
			meas\left\{y\in Y \colon |f(x,y)|\le \alpha |\rho(x)|^N     \right\}\le C\alpha^\delta |\rho(x)|^{N \delta}\,.
		\end{gather*}
	\end{theorem}
	The strategy to deduce the existence of large sets of values $p_0$ corresponding to which assumption \ref{hypoNonRes}
	holds is nontrivial, but in our case, one can reproduce line by line
	the computations of Sect. 5.2 of \cite{DS0}, just
	substituting the function $f$ of that paper with the function
	$$
	\tf(x,y):=\sum_{j=1}^l\sqrt{x_j^2y+x_j}-\sum_{j=l+1}^p\sqrt{x_j^2y+x_j}\ ,
	$$
	which has the property that
	$$
	m^2\tf\left(\lambda_{a_j},\frac{1}{m^2}\right)
	$$
	with $m^2=2f(p_0)$ are exactly the small divisors one has to control.
	
	The only nontrivial point is to prove that for any fixed value of
	$x\in[0,1]^{p}$ the equation $\tf(x,y)=0$ has only isolated solutions
	in $y$. We just give a detailed proof of this fact. 
	
	First, we have to give a selection property for the sequences
	$(x_1,...,x_p)$: 
	we will say that a sequence $(x_1,...,x_p)$, given an integer
	$l\leq p$ satisfy condition $Z$, if one of the following holds:
	\begin{itemize}
		\item $p$ is odd
		\item $p$ is even and $l\not=p/2$
		\item $p$ is even, $l=p/2$ and for any permutation $\tau$ of
		$1,...,p/2$ exists $j$ s.t. $x_{\tau(j)}\not=x_{p/2+j}$.
	\end{itemize}

	\begin{lemma}\label{discrete}
		For any $(x_1,...,x_p)\in[0,1]^p$ fulfilling condition $Z$, the equation $\tf(x,y) = 0$ has a discrete set of solution.
	\end{lemma}
	\begin{proof} Following \cite{DS} we remark that since $\tf$ is an
		analytic function of $y$, its roots can have accumulation points
		only if the function identically vanishes. We compute its Taylor
		expansion at $y=0$ and show that it can be identically zero only if
		condition $Z$ is violated. Denote
		$$
		\nu_j=\sqrt{x_j^2y+x_j}\ ,
		$$
		by direct computation, we get
		\begin{gather}\label{deriv}
			\frac{ d^k \nu_j }{dy^k} = c_k  \left(\frac{x_j
			}{\nu_j}\right)^{2k}\nu_j\ ,
		\end{gather}
		with suitable constants $c_k$. Thus we have
		\begin{equation}
			\label{taylor}
			\frac{\partial^k\tf}{\partial y^k}=c_k\left[ \sum_{j=1}^j \left(\frac{x_j
			}{\nu_j^2}\right)^{2k}\nu_j-\sum_{j=l+1}^p  \left(\frac{x_j
			}{\nu_j^2}\right)^{2k}\nu_j\right]
			\ .
		\end{equation}
		
		Let consider the equivalence relation $$x_i\sim x_j \iff \left(x_i =x_j
		\text{ and } (i,j\leq l\ or\ i,j>l) \right)$$ and denote by $n_j$ the
		cardinality of the correspondent equivalence classes.  So,
		we can write the condition $\eqref{taylor}=0$ (renaming the indexes
		$j$), as 
		\begin{equation}
			\label{taylor2}
			0= \sum_{j=1}^{l_1}n_j  \left(\frac{x_j
			}{\nu_j}\right)^{2k}\nu_j-\sum_{j=l_1+1}^{p_1}n_j  \left(\frac{x_j
			}{\nu_j}\right)^{2k}\nu_j
			\ ,
		\end{equation}
		and now we have that $\forall i\not=j$ $x_i\not=x_j$. Remark
		that in our computation we have implicitly erased the terms
		with a plus sign which have a corresponding term with a minus sign, and
		that, 
		since condition
		$Z$ is fulfilled there must be at least a couple of indexes $i, j$
		with $i\le l$ and $j>l$ such that $x_i\not =x_j$, so that not all the
		$n_j'$s vanish.
		
		Now, \eqref{taylor2} is a linear equation in the ``unknown'' $n_j$:
		to have nontrivial solutions its determinant must vanish. However the
		determinant is essentially a Vandermonde determinant that can be
		explicitly computed giving
		\begin{gather*}
			\nu_1\dots\nu_{p_1}\prod_{1\le k < l\le p_1}\left(\frac{x_l^2}{\nu_l^2}
			-\frac{x_k^2}{\nu_k^2}\right) \not = 0 \ .
		\end{gather*}
	\end{proof}
	As we anticipated the rest of the proof goes exactly as in \cite{DS}
	and is omitted.

	\subsection{Beam equation}
	Consider the beam equation introduced in \eqref{beam}
	\begin{equation}\label{beam2}
		\psi_{tt}+\Delta^2\psi+m\psi=-\partial_\psi P(x,\psi)\,.
	\end{equation}
	Introducing the variable $\varphi=\dot{\psi}\equiv\psi_{t}$,
	it is well known that \eqref{beam2} can be seen as an 
	Hamiltonian system in the variables $(\psi,\varphi)$ with Hamiltonian function
	\begin{equation*}
		H(\psi,\varphi):=\int_{M}\left(\frac{\varphi^2}{2}
		+\frac{\psi(\Delta_{g}^2+m)\psi}{2} 
		+P(x,\psi)
		\right)dx \, . 
	\end{equation*} 
	To put the system in the form \eqref{H0.s} 
	introduce the variables 
	\begin{align*}
		u(x)&:=\frac{1}{\sqrt 2}\left(\left(\Delta_{g}^2+m\right)^{1/4}\psi+i
		\left(\Delta_{g}^2+m\right)^{-1/4}\varphi \right)\ ,\quad  
		\\
		\overline u (x)&:=\frac{1}{\sqrt 2}\left(\left(\Delta_{g}^2+m\right)^{1/4}\psi
		-i \left(\Delta_{g}^2+m\right)^{-1/4}\varphi\right)\ ,
	\end{align*}
	so that the Hamiltonian takes the form
	\begin{equation*}
		H(u,\bar u):=\int_{M}\bar u\left(H_L u\right) + 
		P(x,u,\bar u)\,dx \,,
	\end{equation*} 
	with
	\begin{equation*}
		H_L = \sqrt{-\Delta_{g}^2+m}\,.
	\end{equation*}
	Then the verification of the hypotheses goes as in the previous case
	(but in a simpler way) and thus it is omitted.

	\appendix
	\section{Estimates on polynomial with localized coefficients}\label{esti.appendix}
	In this Appendix, we will prove the results stated in Section
	\ref{locali}.
	
	We denote $\la \cdot,\cdot \ra= \la \cdot,\cdot \ra_{\H^0}$ and 
	$\la \cdot,\cdot \ra_e = \la \cdot,\cdot \ra_{\H^0_e}$
	and $\left\|.\right\|:=\left\|.\right\|_0$.
	
	\subsection{Tame estimates}
	\begin{lemma}
		Let $P \in L^{\nu,N}_{r+1}$, then $X_P \in M^{\nu,N}_{r} $. Furthermore, 
		\begin{gather}\label{vecControl.proof}
			\norm{X_P}^{\nu,N}\le r
			\norm{P}^{\nu,N}\,.
		\end{gather}
		
	\end{lemma}
	\begin{proof}
		For $B\in \Lambda_e$ and $\textbf A = (A_1,\dots,A_r) \in \Lambda_e^r$, we want to bound
		\begin{gather*}
			\norm{\Pi_B \widetilde{X}_P\left(\Pi_{A_1}u,\dots,\Pi_{A_r}u\right)}_0\,.
		\end{gather*}
		Suppose for simplicity that $B = (b,+)$, the case $B = (b,-)$ being totally analogous. 
		
		\noindent
		We compute, exploiting self-adjointness of $\Pi_B$ and 
		the definition of the $L^2$-gradient \eqref{defHamVec},
		\begin{equation*}
			\begin{aligned}
				&\norm{\Pi_B \widetilde{X}_P\left(\Pi_{A_1}u_1,\dots,\Pi_{A_r}u_r\right)}_0 
				\\&\qquad
				= \sup_{h\in\H^\infty, \norm{h} = 1}
				\left| \la \Pi_B \widetilde{X}_P(\Pi_{A_1}u_1,\dots,\Pi_{A_r}u_r), h\ra_e \right| 
				\\& \qquad
				= \sup_{h\in\H^\infty, \norm{h} = 1}
				\left| \la  \widetilde{X}_p(\Pi_{A_1}u_1,\dots,\Pi_{A_r}u_r), \Pi_B h\ra_e \right|  
				\\&\qquad
				\le r \sup_{h\in\H^\infty, \norm{h} = 1} 
				\left| \widetilde{P}(\Pi_{A_1}u_1,\dots,\Pi_{A_r}u_r, \Pi_Bh)\right| 
				\\& \qquad
				= r \norm{P}^{\nu,N}_{r+1} \frac{\mu(\textbf{A},B)^{N+\nu}}{S(\textbf{A},B)^N} \norm{\Pi_{A_1}u_1}\dots\norm{\Pi_{A_r}u_r} \norm{\Pi_B h}.
			\end{aligned}
		\end{equation*}
		From that, we deduce \eqref{vecControl.proof}.
	\end{proof}
	
	\begin{lemma}\label{tecni}
		Fix $\nu>0$. For $s>\nu +\frac{d}{2}$, we have
		\begin{gather}
			\sum_{A\in\Lambda_e}(1+\len a)^\nu\norm{\Pi_A u}_0 \lesssim_s \norm{u}_s.
		\end{gather}
	\end{lemma}
	\begin{proof}
		We compute
		\begin{gather*}
			\sum_{A\in\Lambda_e}(1+\len a)^\nu \norm{\Pi_A u}_0 \le 
			\sum_{A\in\Lambda_e} \frac{(1+\len a)^s}{(1+\len a)^{s-\nu}}\norm{\Pi_A u}_0  \le 
			\\
			\le \sqrt{\sum_{A\in\Lambda_e}\frac{1}{(1+\len a)^{2(s-\nu)}}} \sqrt{\sum_{A\in\Lambda_e}(1+\len a)^{2 s} \norm{\Pi_A u}_0} \lesssim C_s \norm{u}_s
		\end{gather*}
		since the first sum converges for $2(s-\nu)>d$, that is $s>\nu + \frac{d}{2}$.
	\end{proof}
	\begin{remark}\label{newDiago.rem}
		In the following computations, we will repeatedly compare $|a|$ and $\len a$, 
		taking profit of Remark \ref{equi.mod}.
		In particular, we notice that for any constant $0<K_2<1$ small enough there exists  
		$K_1>0$ large enough such that
		\begin{gather}\label{newDiago}
			\len a \ge K_1 \len b \quad\implies\quad|a-b|\ge K_2 \len a\,.
		\end{gather}
		In fact we have, defining $C_1,C_2$ as in Remark \ref{equi.mod}. 
		\begin{gather*}
			|a-b| \ge |a| - |b| \ge \frac{1}{C_2}\len{a} - \frac{1}{C_1} \len{b} 
			\ge 
			\left(\frac{1}{C_2} - \frac{1}{C_1 K_1}\right)\len{a}\,.
		\end{gather*} 
		Moreover $\exists C$ such that
		\begin{gather}\label{newDiago2}
			\len a \ge \len b \quad\implies\quad |a-b|\le C \len a\,.
		\end{gather}
	\end{remark}
	\begin{theorem}[Tame estimate]\label{tameMultiproof}
		Let $X\in \mathcal{M}^{\nu,N}_r$ and fix
		$s>\frac{3}{2}d+\nu$. If $N>d + s$, for any
		$s_0\in(\frac{3}{2}d+\nu,s)$ one has
		\begin{align*}
			\norm{\tilde{X}(u_1,\dots,u_r)}_s \lesssim  
			\norm{X}^{\nu,N} \sum_{j=1}^{r} \norm{u_j}_s\prod_{k\not=j}\norm{u_k}_{s_0}\,,
			\\
			\forall u_1,\dots,u_r\in \cH^{\infty}_e\ .
		\end{align*}
	\end{theorem}
	\begin{proof}
		We have
		\begin{gather*}
			\norm{\tilde{X}(u_1,\dots,u_r)}_s^2 \leq \sum_{B\in\varLambda}(1+\len B)^{2s} 
			\norm{\sum_{A_1,\dots A_r} \Pi_B\tilde{X}(\Pi_{A_1}u_1,\dots,\Pi_{A_r}u_r)}_0^2\,.
		\end{gather*}
		Exploiting the definition of localized coefficients \ref{d.4.3}, the argument of the sum 
		in $B$ of the equality above 
		is controlled by the square of
		\begin{gather*}
			\norm{X}^{\nu,N}\sum_{A_1,\dots A_r} (1+ \len b)^{s} 
			\frac{\mu(\textbf{A},b)^{N+\nu}}{S(\textbf{A},b)^N} 
			\norm{\Pi_{A_1}u_1}_0\dots\norm{\Pi_{A_r}u_r}_0\,.
		\end{gather*}
		By symmetry, we can consider the case 
		$\len{a_1}\ge\len{a_2}\ge\dots\ge \norm{a_r}$. 
		By definition of $S$ and $\mu$, we have
		\begin{gather*}
			\begin{cases}
				\len b\ge\len{a_2} \implies S(\textbf{A},b) = |a_1 - b| + \len{a_2}\ge |a_1 - b|\\
				\len b\le \len{a_2}\implies S(\textbf{A},b) \ge |a_1-a_2| + \len b\,.
			\end{cases}
		\end{gather*}
		We have, $\forall \kappa > d$,
		\begin{gather*}
			\sum_{B\in \Lambda_e}S^{-\kappa} \le \sum_{\len b\ge \len{a_2}}S^{-\kappa} + \sum_{\len b < \len{a_2}}S^{-\kappa} \lesssim 
			\\ \lesssim \sum_{B\in\Lambda_e}\frac{1}{(|a_1 - b| + 1)^\kappa} + \sum_{B\in\Lambda_e}\frac{1}{(|a_1 - a_2| + \len b)^\kappa} \lesssim \\
			\lesssim \sum_{B'\in\Lambda_e}\frac{1}{(|b'| + 1)^\kappa} + \sum_{B\in\Lambda_e}\frac{1}{(|a_1 - a_2| +  |b|)^\kappa}  < \infty
		\end{gather*}
		where $B' = (b',\sigma_b)$, with $b' = b-a_1$. With a similar calculation, we obtain also
		\begin{gather*}
			\sum_{A_1\in \Lambda_e}S^{-\kappa} < \infty\, .
		\end{gather*}
		By Cauchy-Schwarz inequality, we estimate
		\begin{gather*}
			\sum_{A_1,\dots A_r} (1+\len b)^{s}
			\frac{\mu(\textbf{A},b)^{N+\nu}}{S(\textbf{A},b)^N}
			\norm{\Pi_{A_1}u_1}_0\dots\norm{\Pi_{A_r}u_r}_0
			\le \\ \label{s.B.1}
			\le \left( \sum_{A_1,\dots A_r} (1+\len b)^{2s}
			\frac{\mu(\textbf{A},b)^{2N+\nu-\kappa}}{S(\textbf{A},b)^{2N-\kappa}}
			\norm{\Pi_{A_1}u_1}_0^2\norm{\Pi_{A_2}u_2}_0\dots\norm{\Pi_{A_r}u_r}_0\right)^{\frac{1}{2}}\cdot
			\\
			\label{s.B.2}
			\cdot \left( \sum_{A_1,\dots A_r} \frac{\mu(\textbf{A},b)^{\nu+\kappa}}{S(\textbf{A},b)^{\kappa}} \norm{\Pi_{A_2}u_2}_0\dots\norm{\Pi_{A_r}u_r}_0\right)^{\frac{1}{2}}\, .
		\end{gather*}
		Exploiting $\mu(\bA,b)\leq \len{a_2}$, 
		the second term is estimated by 
		\begin{gather*}
			\left( \sum_{A_1,\dots A_r} \frac{\mu(\textbf{A},b)^{\nu+\kappa}}{S(\textbf{A},b)^{\kappa}}\norm{\Pi_{A_2}u_2}_0\dots\norm{\Pi_{A_r}u_r}_0\right)^{\frac{1}{2}}\\ \le 
			\left( \sum_{A_1}S(\textbf{A},b)^{-\kappa}\sum_{A_2,\dots,A_r}\len{a_2 }^{\nu+\kappa}\norm{\Pi_{A_2}u_2}_0\dots\norm{\Pi_{A_r}u_r}_0\right)^{\frac{1}{2}} \lesssim \prod_{l=2}^{r}\norm{u_l}^{\frac{1}{2}}_{s_0}
		\end{gather*}
		for each $s_0>\nu + \frac{d}{2} + \kappa$.
		
		\noindent
		Consider now the first term \eqref{s.B.1}. We claim that 
		\begin{gather}\label{2a1}
			\frac{\mu(\textbf{A},b)}{S(\textbf{A},b)}
			\left(1 + \len b\right) \lesssim 1+\len{a_1}\ .
		\end{gather}
		Indeed, \eqref{2a1} is trivial for $1 + \len b\lesssim \left(1 +
		\len{a_1}\right)$ , since $\frac{\mu}{S}<1$ by definition. On the other end, if $1  + \len b \gtrsim \left(1 + \len{a_1}\right)$, we have, from \eqref{newDiago},
		\begin{gather*}
			S(A,b)\ge |b - a_1|\succeq \len{b}
		\end{gather*}
		and then
		\begin{gather*}
			\frac{\mu(\textbf{A},b)}{S(\textbf{A},b)}(1 + \len{b}) \lesssim \mu(\textbf{A},b) \lesssim   1 + \len{a_1}
		\end{gather*}
		that is \eqref{2a1}.
		Then we can control the first term, provided that $N>s+\kappa$, with
		\begin{equation*}
			\begin{aligned}
				&\sum_{B}\left(\sum_{A_1,\dots A_r} (1+\len b)^{2s} 
				\frac{\mu(\textbf{A},b)^{2N+\nu-\kappa}}{S(\textbf{A},b)^{2N-\kappa}} 
				\norm{\Pi_{A_1}u_1}_0^2\norm{\Pi_{A_2}u_2}_0\dots\norm{\Pi_{A_r}u_r}_0\right)^{\frac{1}{2}}
				\\& 
				\lesssim \sum_{B}\left(
				\sum_{A_1,\dots A_r} (1 + \len{a_1})^{2s} \frac{\mu(\textbf{A},b)^{\kappa + \nu}}{S(\textbf{A},b)^{\kappa  }}
				\norm{\Pi_{A_1}u_1}_0^2\norm{\Pi_{A_2}u_2}_0\dots\norm{\Pi_{A_r}u_r}_0 \right)^{\frac{1}{2}}
				\\& 
				\lesssim  \sum_{B}  
				\left(\sum_{A_1}\frac{(1 + \len{a_1})^{2s}}{S(\textbf{A},b)^{\kappa}} 
				\norm{\Pi_{A_1}u_1}_0^2 \sum_{A_2,\dots,A_r} \len{a_2}^{\nu+\kappa}
				\norm{\Pi_{A_2}u_2}_0\dots\norm{\Pi_{A_r}u_r}_0 \right)^{\frac{1}{2}}
				\\& \lesssim
				\norm{u_1}_s \prod_{l=2}^{r}\norm{u_l}^{\frac{1}{2}}_{s_0} \,,
			\end{aligned}
		\end{equation*}
		with $s_0\ge\kappa +\nu + \frac{d}{2}$. 
		Summing over all the possible choices of the biggest index, we obtain the sum in the thesis.
	\end{proof}

	\subsection{High and low modes}
		%
		%
	%
	\begin{proof}[Proof of Corollary \ref{cubic}]
		Firstly recall the usual high modes estimate
		\begin{equation}\label{high}
			\begin{aligned}
				\norm{u^\perp}_{s_0}^2 = \sum_{\len
					A>K}(1+\len a)^{2  s_0} \norm{\Pi_A
					u}_0 ^ 2  
				\lesssim \frac{1}{K^{2(s-s_0)}} \norm{u}_s^2\,.
			\end{aligned}
		\end{equation}
		
		\noindent
		i) If $P$ is at least of order 3 in $u^\perp$, $X_P$ is at least of order two in $u^\perp$. 
		Then
		\begin{equation*}
			\begin{aligned}
				X_P(u)& = X(u^{\le} + u^\perp) = \tilde{X}_P(u^{\le} + u^\perp,\dots,u^{\le} + u^\perp) 
				\\&
				= \sum_{l=2}^{r} \binom{r}{l} 
				\tilde{X}_P(\underset{r-l-times}{u^{\le},\dots,u^{\le}},\underset{l- times}{u^\perp,\dots,u^\perp})\,.
			\end{aligned}
		\end{equation*}
		Applying the tame estimate in Lemma \ref{tameMultiproof}, we get
		\begin{equation*}
			\begin{aligned}
				&\norm{X_P(u)}_s 
				\\&\quad 
				\le  \norm{P}^{\nu,N}\sum_{l=2}^{r}
				\left(\norm{u^\le}_s\norm{u^\le}^{r-l-1}_{s_0}\norm{u^\perp}_{s_0}^l 
				+ \norm{u^\le}^{r-l}_{s_0}\norm{u^\perp}_{s_0}^{l-1} \norm{u^\perp}_s\right)  
				\\&\quad
				\stackrel{\eqref{high}}{\lesssim}   
				\frac{\norm{P}^{\nu,N}}{K^{s-s_0}} \norm{u}_s^2\norm{u}_{s_0}^{r-2}\,.
			\end{aligned}
		\end{equation*}
		ii) We reason as in the previous case, since again $\Pi^\le X_P$ is of order two in $u^\perp$.
	\end{proof}
	\begin{proof}[Proof of Corollary \ref{cutOff2}]
		Writing $N' = N - M_1$, \eqref{sti.lunghi} amounts to show that 
		\begin{equation*}
			\begin{aligned}
				\frac{\mu(A_1,\dots,A_r)^{N+\nu}}{S(A_1,\dots,A_r)^N} 
				&= \frac{\mu(A_1,\dots,A_r)^{N'+\nu + M_1}}{S(A_1,\dots,A_r)^{N' + M_1}} 
				\\&
				\le 
				\frac{\mu(A_1,\dots,A_r)^{N'+\nu + M_1}}{S(A_1,\dots,A_r)^{N'}} 
				\frac{1}{|a_{\tau_{ord(1)}} - a_{\tau_{ord(2)}}|^{M_1}}
				\\ &
				\le
				\frac{\mu(A_1,\dots,A_r)^{N'+\nu + M_1}}{S(A_1,\dots,A_r)^{N'}} \frac{1}{K^{\delta M_1}}\,.
			\end{aligned}
		\end{equation*}
		Then we apply directly Lemmas \ref{vecControl.proof} and \ref{tameMultiproof} and we get \eqref{impro.tame.proof}.
	\end{proof}	
	\subsection{Poisson brackets}
	We conclude this section by showing that our class is closed with respect to Poisson brackets. 
	Before proving the result, we state the following useful lemma.
	\begin{lemma}\label{polyPoisson}
		Given a polynomial $P$ and a polynomial map $X$, we have
		\begin{equation*}
			dP\,X (u) = \eta\left[\tilde{P}(u,\dots,u, X(u))\right]\,,
		\end{equation*}
		where
		\begin{equation*}
			\eta [\tilde{P}(u,\dots,u,h)] \coloneqq \tilde{P}(h,u,\dots,u) + \tilde{P}(u,h,\dots,u) + \dots + \tilde{P}(u,\dots,u,h)\,,
		\end{equation*}
		Moreover if $P$ has degree $p+1$ and $X$ has degree $q$, the multilinear polynomial associated to $dP\,X $ is given by
		\begin{equation}\label{poiPoly2}
			\widetilde{dP\,X }(u_1,\dots,u_{p+q}) 
			= \eta\left[\tilde{P}\left(u_1,\dots,u_p,\tilde{X}(u_{p+1},\dots,u_{p+q})\right)\right]\,.
		\end{equation}
	\end{lemma}
	
	\begin{lemma}[Poisson brackets]\label{Poissonproof}
		Given $P\in L^{\nu_1,N}_{r_1+1}$ and $X\in M^{\nu_2,N}_{r_2}$, we have
		\begin{equation*}
			dP\,X \in  L^{\nu', N'}\,,
		\end{equation*}
		with $N' =  N - d - 1 - \max\{\nu_1,\nu_2\} $ and $\nu' = \nu_1 + \nu_2 + d + 1$.
		Moreover, 
		\begin{equation*}
			\norm{dP\,X}^{\nu'\,N'} \lesssim \norm{P}^{\nu_1,N} \norm{X}^{\nu_2,N}\,.
		\end{equation*}
	\end{lemma}
	\proof
	Let $u_1,\dots, u_p,u_{p+1},\dots, u_{p+q} \in \H^\infty_e$ 
	and $\textbf{A} = (A_1,\dots, A_{p+q}) \in \Lambda_e^{p+q}$. 
	By Lemma \ref{polyPoisson}, we have
	\begin{equation*}
		\begin{aligned}
			&\left|\widetilde{dP\,X}\left(\Pi_{A_1}u_1,\dots,\Pi_{A_{p+q}}u_{p+q}\right) \right| 
			\\&\quad
			= \left|\eta\left[\widetilde{P}\left(\Pi_{A_1}u_1,\dots,\Pi_{A_{p}}u_p,\widetilde{X}(\Pi_{A_{p+1}}u_{p+1},\dots,\Pi_{A_{p+q}}u_{p+q})\right)\right]\right| 
			\\&\quad
			\le (p+1)\sum_{B\in\Lambda_e}
			\left|\widetilde{P}\left(\Pi_{A_1}u_1,\dots,\Pi_{A_{p}}u_p, 
			\Pi_B\widetilde{X}(\Pi_{A_{p+1}}u_{p+1},\dots,\Pi_{A_{p+q}}u_{p+q})\right)\right| 
			\\&\quad
			\le (p+1)\norm{P}^{\nu_1,N} \times
			\\&\qquad\qquad
			\times\sum_{B\in\Lambda_e} \frac{\musb{1}{p}^{N+\nu_1}}{\S{1}{p}^N} 
			\norm{\Pi_{B}\widetilde{X}(\Pi_{A_{p+1}}u_{p+1},\dots,\Pi_{A_{p+q}}u_{p+q})}_0  
			\\& \quad
			\le(p+1)\norm{P}^{\nu_1,N}\norm{X}^{\nu_2,N} \times
			\\&\qquad\qquad
			\times\sum_{B\in\Lambda_e}
			\frac{\musb{1}{p}^{N+\nu_1}}{\S{1}{p}^N} \frac{\musb{p+1}{p+q}^{N+\nu_2}}{\S{1}{p}^N}\,,
		\end{aligned}
	\end{equation*}	
	since $P,X$ have localized coefficients. 
	That is, we need to prove the following estimate
	\begin{equation}\label{tesi.wave}
		\begin{aligned}
			\sum_{B\in\Lambda_e}  \frac{\mu(A_1,\dots,A_p,b)^{N+\nu_1}}{S(A_1,\dots,A_p,b)^N} &\frac{\mu(A_{p+1},\dots,A_{p+q},b)^{N+\nu_2}}{S(A_{p+1},\dots,A_{p+q},b)^N} 
			\\&\qquad\quad\lesssim \frac{\mu(A_1,\dots,A_{p+q})^{N'+\nu'}}{S(A_1,\dots,A_{p+q})^{N'}}\, .
		\end{aligned}
	\end{equation}
	By symmetry, we can assume the following relations:
	\begin{equation*}
		\len{a_1}\ge\dots\ge\len{a_p},\quad \len{a_{p+1}}\ge\dots\ge\len{a_{p+q}}\,, \qquad \len{a_{p+1}}\le\len{a_1}\, .
	\end{equation*}
	\textbf{Case 1} Assume $\len{a_1}\ge\len{a_{p+1}}\ge \len{a_2}$.\\
	In this case, we have
	\begin{equation*}
		\begin{aligned}
			\mus{1}{p+q} &= \max(\len{a_2},\len{a_{p+2}})\,,
			\\
			\S{1}{p+q} &= |a_ 1-a_{p+1}| + \mus{1}{p+q}\,,
		\end{aligned}
	\end{equation*}
	and
	\begin{equation}\label{stimaUno}
		\begin{aligned}
			\musb{1}{p} &\le \mus{1}{p+q}\,, 
			\\
			\musb{p+1}{p+q} &\le \mus{1}{p+q}\,.
		\end{aligned}
	\end{equation}	
	\noindent
	\textbf{Case 1.i} For $\len b >
	\max(\len{a_2},\len{a_{p+2}})$, we have 
	\begin{align}
		\nonumber
		\Sb{1}{p} = |a_1 - b| + \len{a_{2}}\,,
		\\
		\nonumber
		\Sb{p+1}{p+q} = |a_{p+1}- b| + \len{a_{p+2}}\,.
		\\
		\label{stimaDue}
		|a_1 - a_{p+1}| = |a_1 - b - a_{p+1} +b | \le  |a_1 - b| +  |a_{p+1} - b|\,.
	\end{align}
	From \eqref{stimaUno} and \eqref{stimaDue}, we deduce
	\begin{equation}\label{stimaMax}
		\begin{aligned}
			&\frac{\S{1}{p+q}}{\mus{1}{p+q}} = 1 + \frac{|a_1 - a_{p+1}|}{\mus{1}{p+q}} 
			\\& \qquad\qquad
			\le 1 + \frac{|a_1 -b|}{\musb{1}{p}} + \frac{|a_{p+1} -b|}{\musb{p+1}{p+q}} 
			\\& \qquad\qquad
			\le \frac{\Sb{1}{p}}{\musb{1}{p}} +  \frac{\Sb{p+1}{p+q}}{\musb{p+1}{p+q}}
			\\&\qquad\qquad
			\leq
			2 \max\left\{\frac{\Sb{1}{p}}{\musb{1}{p}}\,, \frac{\Sb{p+1}{p+q}}{\musb{p+1}{p+q}}\right\}\,.
		\end{aligned}
	\end{equation}
	Let us define
	\[
	L_1 = \left\{ b\in\Lambda\ :\ \len{b}>  \max(\len{a_{2}},\len{a_{p+2}}):\frac{\Sb{1}{p}}{\musb{1}{p}}
	\ge \frac{\Sb{p+1}{p+q}}{\musb{p+1}{p+q}}\right\}\,.
	\]
	Depending on the value of $\bA$ one could have $L_1=\emptyset$, but
	this is irrelevant for the following. 
	If $b\in L_1$, estimate \eqref{stimaMax} implies 
	\begin{equation}\label{stimaQuattro}
		\frac{\musb{1}{p}}{\Sb{1}{p}}\le 2 \frac{\mus{1}{p+q}}{\S{1}{p+q}}\,.
	\end{equation}
	We observe moreover that
	\begin{equation}\label{stimaOvvia}
		\frac{\mu(A_{p+1},\dots,A_{p+q},b)^{N+\nu_2}}{S(A_{p+1},\dots,A_{p+q},b)^N} 
		\le \musb{p+1}{p+q}^{\nu_2}\,.
	\end{equation}
	Then, using \ref{stimaUno}, \ref{stimaQuattro} and \ref{stimaOvvia}, we have
	\begin{align*}
		\sum_{B\in L_1}  &\frac{\mu(A_1,\dots,A_p,b)^{N+\nu_1}}{S(A_1,\dots,A_p,b)^N} 
		\frac{\mu(A_{p+1},\dots,A_{p+q},b)^{N+\nu_2}}{S(A_{p+1},\dots,A_{p+q},b)^N}
		\\
		&
		\stackrel{\eqref{stimaUno}}{\lesssim} 
		\mus{1}{p+q}^{\nu_1 + \nu_2} \times
		\\
		&\qquad\qquad \times\sum_{B\in L_1}  \frac{\mu(A_1,\dots,A_p,b)^{N - d
				-1}}{S(A_1,\dots,A_p,b)^{N-d-1}}  
		\frac{\mu(A_1,\dots,A_p,b)^{d+1}}{S(A_1,\dots,A_p,b)^{d+1}}
		\\
		&
		\lesssim
		\mus{1}{p+q}\null^{\nu_1 + \nu_2}
		\sum_{b\in L_1}  \frac{\mus{1}{p+q}\null^{N-d-1}}{\S
			1{p+q} \null^{N-d-1}}\frac{\mus 1{p+q}\null^{d+1}}{S(A_1,...,A_p,b)^{d+1}}\,,
		\\
		&
		\stackrel{\eqref{stimaQuattro},\eqref{stimaUno}}{\lesssim}
		\frac{\mus{1}{p+q}^{N +\nu_1 + \nu_2}}{\S{1}{p+q}^{N-d-1}}\sum_{B\in L_1}  \frac{1}{(|a_1 - b|+1)^{d+1}}\,,
	\end{align*}
	that is \eqref{tesi.wave}, with $N'= N - d - 1$ and $\nu' = \nu_1 + \nu_2 + d + 1 $. 
	The case $b\in L_1^c$ is analogous.
	
	\smallskip
	\noindent
	\textbf{Case 1.ii} For $\len b \le  \max(\len{a_{2}},\len{a_{p+2}})$, 
	we remark that the index of the sum in \eqref{tesi.wave} runs over a set with cardinality 
	controlled by $\mus{1}{p+q}$. That is, it suffices to prove that
	\begin{equation}\label{aim}
		\frac{\mu(A_1,\dots,A_p,b)^{N+\nu_1}}{S(A_1,\dots,A_p,b)^N} \frac{\mu(A_{p+1},\dots,A_{p+q},b)^{N+\nu_2}}{S(A_{p+1},\dots,A_{p+q},b)^N}\frac{S(A_1,\dots,A_{p+q})^{N'}}{\mu(A_1,\dots,A_{p+q})^{N'+\nu'}}
	\end{equation}
	is controlled by a constant  independent of $b$.
	
	Take $0<K_2<1$ and let  $K_1$ and $C$ be as in \eqref{newDiago}
	If $\len{a_1}\ge K_1 \max(\len{a_2},\len{a_{p+2}})$, we get, 
	\begin{equation}\label{Aim.uno}
		\begin{aligned} 
			\eqref{aim} &\le \frac{\max(\len b,\len{a_{3}})^{N+\nu_1}}{(K_2\len{a_{1}})^N} \max(\len b,\len{a_{p+3}})^{\nu_2} 
			\times
			\\&\qquad\qquad\qquad
			\times\frac{\left(C\len{a_{1}} + \max(\len{a_{2}},\len{a_{p+2}})\right)^{N'}}{\max(\len{a_{2}},\len{a_{p+2}})^{N'+\nu'}} 
			\\& 
			\lesssim \frac{\max(\len{a_{2}},\len{a_{p+2}})^{N+\nu_1+\nu_2}}{\len{a_1}^N}
			\frac{(\len{a_1})^{N'}}{\max(\len{a_2},\len{a_{p+2}})^{N'+\nu'}}\lesssim 1\,,
		\end{aligned}
	\end{equation}
	choosing $N' = N$ and $\nu' = \nu_1 + \nu_2$.\\
	If $\len{a_1} \le K_1 \max(\len{a_2},\len{a_{p+2}})$, we get instead
	\begin{equation}\label{Aim.due}
		\begin{aligned} 
			\eqref{aim}&\le \max(\len b,\len{a_{3}})^{\nu_1} \max(\len b,\len{a_{p+3}})^{\nu_2} 
			\times
			\\&\qquad\qquad\qquad
			\times\frac{(C\len{a_{1}}+ \max(\len{a_{2}},\len{a_{p+2}})^{N'}}{\max(\len{a_2},\len{a_{p+2}})^{N'+\nu'}} 
			\\& 
			\lesssim 
			\max(\len{a_2},\len{a_{p+2}})^{\nu_1 + \nu_2} 
			\frac{( \max(\len{a_2},\len{a_{p+2}})^{N'}}{\max(\len{a_2},\len{a_{p+2}})^{N'+\nu'}} \lesssim 1\,,
		\end{aligned}
	\end{equation}
	for $\nu' = \nu_1 + \nu_2$. 
	This proves the claim in the case $1$.ii.
	
	\smallskip
	\noindent
	\textbf{Case 2} Assume $\len{a_1}\ge\len{a_2}\ge\len{a_{p+1}}$. In this case, we have
	\begin{equation*}
		\begin{aligned}
			\mus{1}{p+q} &=\max(\len{a_3},\len{a_{p+1}})\,,
			\\
			\S{1}{p+q} &= |a_1 - a_2| + \max(\len{a_3},\len{a_{p+1}})\,. 
		\end{aligned}
	\end{equation*}
	
	\textbf{Case 2.i} Take $K_2$ as in \eqref{newDiago}
	and determine the corresponding $K_1$. For $\len b > K_1 \max\left(\len{a_3},\len{a_{p+1}}\right)$, 
	we have
	\begin{equation*}
		\begin{aligned}
			\Sb{1}{p} &= |a_1- \text{argmax}(\len{a_2},\len{b})|  + \min(\len{a_2},\len{b})\,,
			\\
			\Sb{p+1}{p+q} &= |b-a_{p+1}| + \len{a_{p+2}}\,,
		\end{aligned}
	\end{equation*}
	and moreover,
	\begin{equation*}
		|a_{p+1}-b| \ge K_2 \len b \,.
	\end{equation*}
	Let us define 
	\[
	G_1 \coloneqq\{b\colon\len b > K_1 \max\left(\len{a_3},\len{a_{p+1}}\right), \len b <\len{a_2} \}\,.
	\] 
	For $b\in G_1$ , we estimate
	\begin{equation}\label{stimaTre}
		\begin{aligned}
			\sum_{b\in G_1} & \frac{\mu(A_1,\dots,A_p,b)^{N+\nu_1}}{S(A_1,\dots,A_p,b)^N} 
			\frac{\mu(A_{p+1},\dots,A_{p+q},b)^{N+\nu_2}}{S(A_{p+1},\dots,A_{p+q},b)^N} 
			\\&
			= 
			\sum_{b\in G_1} \frac{\len b^{N+\nu_1}}{\left(|a_1- a_2|  + \len b\right)^N} 
			\frac{\len{a_{p+2}}^{N+\nu_2}}{\left(|a_{p+1}- b| + \len{a_{p+2}}\right)^N}
			\\& \lesssim
			\sum_{b \in G_1}  \frac{\len b^{N+\nu_1}}{\left(|a_1- a_2|  + \len b\right)^N} 
			\frac{\len{a_{p+2}}^{N+\nu_2}}{\len b^N} 
			\\&
			\lesssim \sum_{\len b > K_1 \max\left(\len{a_3},\len{a_{p+1}}\right)}
			\frac{\len b^{\nu_1}}{\left(|a_1- a_2|  + \len b\right)^N}\len{a_{p+2}}^{N+\nu_2}
			\\&
			\lesssim\mus{1}{p+q}^{N+\nu_2} \frac{1}{\left(|a_1-a_2| + \max\left(\len{a_3},\len{a_{p+1}}\right)\right)^{N-\nu_1-d-1}} 
			\\&
			\lesssim \frac{\mus{1}{p+q}^{N+\nu_2}}{\S{1}{p+q}^{N-\nu_1-d-1}}\,,
		\end{aligned}
	\end{equation} 
	where we are using the inequality
	\begin{equation*}
		\sum_{|k|>A} \frac{|k|^l}{(|k|+B)^N}\leq \sum_{|k|>A} \frac{(|k|+B)^l}{(|k|+B)^N} \le \frac{1}{(A+B)^{N-l-d-1}}\,.
	\end{equation*}
	This is the thesis with $N' = N-d-\nu_1-1$ and $\nu' =  \nu_1 + \nu_2 +d$.
	
	\noindent
	Let $G_2 \coloneqq\{b\colon\len b > K_1 \max\left(\len{a_3},\len{a_{p+1}}\right), \len b \ge\len{a_2} \}$, 
	then we have to estimate
	\begin{gather}\label{stimaAim}
		\sum_{b\in G_2 } \frac{\len{a_2}^{N + \nu_1}}{(|a_1 - b|+ \len{a_2})^N} 
		\frac{\len{a_{p+2}}^{N + \nu_2}}{(|a_{p+1} - b|+ \len{a_{p+2}})^N} \,.
	\end{gather} 
	We observe that there exist two constants $K_3, K_4$  such that
	\begin{equation}\label{K3K4}
		\len{b}\leq K_4 \len{a_1} \quad \Rightarrow |a_1 - b| \geq K_3|a_1-a_2|\,.
	\end{equation}
	Then we estimate \eqref{stimaAim} by
	\begin{equation}\label{double.sum}
		\begin{aligned}
			\sum_{b \in G_2 }& \frac{\len{a_2}^{N + \nu_1}}{(|a_1 - b|+ \len{a_2})^N} 
			\frac{\len{a_{p+2}}^{N + \nu_2}}{|a_{p+1} - b|^N}   
			\\& 
			\stackrel{\eqref{K3K4}}{\lesssim}
			\sum_{b\in G_2, \len{b}\leq K_4\len{a_1} } 
			\frac{\len{a_2}^{N + \nu_1}}{(|a_1 - b|+ \len{a_2})^N} 
			\frac{\len{a_{p+2}}^{N + \nu_2}}{\len{b}^N} 
			\\& + 
			\sum_{b\in G_2, \len b \ge K_4\len{a_1} } 
			\frac{\len{a_2}^{N + \nu_1}}{(|a_1 - b|+ \len{a_2})^N} \frac{\len{a_{p+2}}^{N + \nu_2}}{\len{b}^N}\,.
		\end{aligned}
	\end{equation}
	The first term in \eqref{double.sum} is controlled by
	\begin{gather*}
		\frac{\max(\len{a_{p+1}}, \len{a_3})^{N+\nu_2}}{(|a_1-a_2| + \max(\len{a_{p+1}}, \len{a_3}))^{N-\nu_1 -d-1}} \sum_{b\in G_2} \frac{\len{a_2}^{N+\nu_1} \len{b}^{-N}}{(|a_1-a_2| + \len{a_2})^{\nu_1 + d+ 1}}  
		\lesssim 
		\\ 
		\lesssim \frac{\max(\len{a_{p+1}}, \len{a_3})^{N+\nu_2}}{(|a_1-a_2| + \max(\len{a_{p+1}}, \len{a_3}))^{N-\nu_1 -d-1}} \sum_{b\in G_2}\frac{1}{\len{b}^{d+1}}  \lesssim 
		\\  
		\lesssim \frac{\mu(A_1,\dots,A_{p+q})^{N'+\nu'}}{S(A_1,\dots,A_{p+q})^{N'}}\,.
	\end{gather*}
	For the second term in \eqref{double.sum} we claim that $\len{b}\ge K_4 \len{a_1}$ implies
	\begin{equation*}
		\frac{\len{a_2}^{N + \nu_1}}{(|a_1 - b|+ \len{a_{2}})^N} 
		\frac{\len{a_{p+2}}^{N + \nu_2}}{\len b^N} \lesssim 
		\frac{\max(\len{a_{p+1}}, \len{a_3})^{N'+\nu'}}{\left(|a_1 - a_2| + \max(\len{a_{p+1}}, \len{a_3})\right)^{N'}}
		\frac{1}{\len b ^{d+1}}\,,
	\end{equation*}
	with $N'=N-\nu_1-d-1$ and $\nu'=\nu_1 + \nu_2 + d + 1$.
	Indeed, we have
	\begin{equation*}
		\begin{aligned}
			&\frac{\len{a_2}^{N + \nu_1}}{(|a_1 - b|+ \len{a_2})^N} 
			\frac{\len{a_{p+2}}^{N + \nu_2}}{\len b ^N} \frac{\left(|a_1 - a_2| + \max(\len{a_{p+1}}, \len{a_3})\right)^{N'}}{\max(\len{a_{p+1}}, \len{a_3})^{N'+\nu'}}\len b ^{d+1} 
			\\&\quad
			\lesssim \len{a_2}^{\nu_1} \max(\len{a_{p+1}}, \len{a_3})^{N+\nu_2} \frac{1}{\len b ^{N-d-1}}\frac{\len{a_1}^{N'}}{\max(\len{a_{p+1}}, \len{a_3})^{N'+\nu'}} 
			\\&\quad 
			\stackrel{\len b\ge K_4\len{a_2} }{\lesssim} \frac{\len{a_1}^{N'}}{\len{b}^{N-d-1-\nu_1}}
			\lesssim \frac{\len{a_1}^{N'}}{\len{a_1}^{N-d-1-\nu_1}} \lesssim 1\,,
		\end{aligned}
	\end{equation*}
	for any $N' \le N - d - 1 -\nu_1$. 
	Then, we estimate 
	\begin{equation*}
		\begin{aligned}
			\sum_{b\in G_2, \len b\ge K_4\len{a_1} }& \frac{\len{a_2}^{N + \nu_1}}{(|a_1 - b|+ \len{a_2})^N} 
			\frac{\len{a_{p+2}}^{N + \nu_2}}{\len b ^N} 
			\\&  
			\lesssim \frac{\max(\len{a_{p+1}}, \len{a_3})^{N'+\nu'}}{\left(|a_1 - a_2| + \max(\len{a_{p+1}}, \len{a_3})\right)^{N'}}
			\sum_{b\in G_2} \frac{1}{\len{b}^{d+1}}
			\\& 
			\lesssim  \frac{\max(\len{a_{p+1}}, \len{a_3})^{N'+\nu'}}{\left(|a_1 - a_2| + \max(\len{a_{p+1}}, \len{a_3})\right)^{N'}}\,.
		\end{aligned}
	\end{equation*}
	
	\noindent
	\textbf{Case 2.ii} For $\len{b} \le K_1 \max(\len{a_{p+1}}, \len{a_3})$, we argue as in Case 1.ii, and it is sufficient to bound from above
	\begin{gather}\label{aimB}
		\frac{\mu(A_1,\dots,A_p,b)^{N+\nu_1}}{S(A_1,\dots,A_p,b)^N} \frac{\mu(A_{p+1},\dots,A_{p+q},b)^{N+\nu_2}}{S(A_{p+1},\dots,A_{p+q},b)^N}\frac{S(A_1,\dots,A_{p+q})^{N'}}{\mu(A_1,\dots,A_{p+q})^{N'+\nu'}}
	\end{gather} 
	with a constant.
	
	\noindent
	If $|a_1-a_2|\le 2\max(\len{a_{p+1}}, \len{a_3})$, we get
	\begin{equation*}
		\begin{aligned}
			\eqref{aimB} &\le \max(\len{b}, \len{a_3})^{\nu_1} \musb{p+1}{p+q}^{\nu_2} \times
			\\&\qquad\qquad\qquad \times
			\frac{\left(|a_1-a_2| +\max(\len{a_{p+1}}, \len{a_3})\right)^{N'}}{\max(\len{a_{p+1}}, \len{a_3})^{N' + \nu'}} 
			\\&
			\lesssim \max(\len{a_{p+1}}, \len{a_3})^{\nu_1 + \nu_2} \frac{\left(\max(\len{a_{p+1}}, \len{a_3})\right)^{N'}}{\max(\len{a_{p+1}}, \len{a_3})^{N'+\nu'}} \lesssim 1\,,
		\end{aligned}
	\end{equation*}
	for $\nu' \geq \nu_1 + \nu_2$.
	
	\noindent
	If $|a_1-a_2|\ge 2 \max(\len{a_{p+1}}, \len{a_3})$ , we consider separately 
	the case $\len b \le \len{a_2}$ and the case $\len b>\len{a_2}$. 
	
	\noindent
	If  $\len b \le \len{a_2}$ , we have
	\begin{equation*}
		\begin{aligned}
			&\eqref{aimB} \le \frac{\max(\len{b}, \len{a_3})^{N+\nu_1}}{|a_1 - a_2|^N} 
			\musb{p+1}{p+q}^{\nu_2} \times
			\\&\qquad\qquad\times
			\frac{\left(|a_1-a_2| +\max(\len{a_{p+1}}, \len{a_3})\right)^{N'}}{\max(\len{a_{p+1}}, \len{a_3})^{N' + \nu'}} 
			\\&\quad
			\le 2^{N+\nu_1+\nu_2}
			\frac{\max(\len{a_{p+1}}, \len{a_3})^{N + \nu_1 + \nu_2}}{|a_1-a_2|^N} 
			\frac{\left(|a_1-a_2| + \frac{1}{2}|a_1-a_2|\right)^{N'}}{\max(\len{a_{p+1}}, \len{a_3})^{N' + \nu'}}  
			\le 3^{N'}\,,
		\end{aligned}
	\end{equation*}
	for $\nu' \geq \nu_1 + \nu_2$.
	
	\noindent
	If $\len{b}>\len{a_2}$ and recalling that $\max(\len{a_{p+1}}, \len{a_3})\le \frac{1}{2}|a_1-a_2| $, 
	we need to bound
	\begin{equation}\label{last}
		\begin{aligned}
			\eqref{aimB} &\le 
			\frac{\len{a_2}^{N+\nu_1}}{(|a_1 - b| + \len{a_2})^N} 
			\frac{\len{a_{p+2}}^{N + \nu_2}}{(|a_{p+1}-b|+\len{a_{p+2}})^{N}} 
			\frac{\left(\frac{3}{2}|a_1-a_2|\right)^{N'}}{\max(\len{a_{p+1}}, \len{a_3})^{N' + \nu'}} 
			\\& \lesssim
			\frac{\len{a_2}^{N+\nu_1}}{(|a_1 - b| +\len{a_2})^N}  
			\max(\len{a_{p+1}}, \len{a_3})^{\nu_2} 
			\frac{\left( \frac{1}{C}|a_1 - b| + \len{b} \right)^{N'}}{\max(\len{a_{p+1}}, \len{a_3})^{N' + \nu'}} \,,
		\end{aligned}
	\end{equation}
	where we are using the following triangular inequality,
	\begin{equation*}
		|a_1 -a_2| \le |a_1 - b| + |b-a_2|\le|a_1 - b| + C \len{b},
	\end{equation*}
	and $\len{a_2}<\len b$.
	
	\noindent
	If $|a_1-b|\le \len{b}$ we conclude, 
	recalling that $\len b\le K_1\max(\len{a_{p+1}}, \len{a_3})$, that
	\begin{equation*}
		\begin{aligned}
			\eqref{last}&\lesssim  
			\max(\len{a_{p+1}}, \len{a_3})^{\nu_1 + \nu_2} \frac{\len{b} ^{N'}}{\max(\len{a_{p+1}}, \len{a_3})^{N' + \nu'}} 
			\\&
			\lesssim  \frac{\max(\len{a_{p+1}}, \len{a_3})^{\nu_1 + \nu_2 + N}}{\max(\len{a_{p+1}}, \len{a_3})^{N'+\nu'}}  
			\lesssim 1 \,,
		\end{aligned}
	\end{equation*}
	for $N = N'$ and $\nu' = \nu_1+\nu_2$. If instead $|a_1-b|\ge \len b$, we conclude that 
	\begin{gather}
		\eqref{last}\lesssim  
		\frac{\max(\len{a_{p+1}}, \len{a_3})^{N+\nu_1+\nu_2}}{|a_1 -b |^N} 
		\frac{|a_1 - b| ^{N'}}{\max(\len{a_{p+1}}, \len{a_3})^{N' + \nu'}} \lesssim 1\,,
	\end{gather}
	for $N \leq N'$ and $\nu' \geq \nu_1+\nu_2$.
	 \qed
	
	\section{Proof of Theorem \ref{lo.generale}} \label{locasec}
	
	In this section, we prove theorem \ref{lo.generale}.
	
	The main step to prove it is the following.
	\begin{theorem}\label{eigenDue}
		Let $g\in C^\infty(M)$, then there exists $\nu>0$, depending on $k$ and the dimension $d$, and $\forall N\in\n $ there exists a constant $C_N$ such that, $\forall u_1,\dots,u_k\in \H^\infty$, 
		\begin{equation*}
			\left|\int_M g\Pi_{a_1}u_1 \dots\Pi_{a_k} u_k \right| 
			\le 
			C_N\frac{\mu(a_1,\dots,a_k)^{N+\nu}}{S(a_1,\dots,a_k)^N} \prod_{l=1}^{k} \norm{\Pi_{a_l}u_l}_{0}\,,
		\end{equation*}
		for any $ (a_1,\dots, a_k) \in \Lambda^k$, where $\Pi_{a_i}$ are the projectors defined in \eqref{proje}.
	\end{theorem}
	The strategy of the proof is very similar to that of the
	corresponding results in \cite{DS,BDGS,DI17}; the difference being that the indexes run over the set $\Lambda$
	related to the quantum actions. We just recall the main steps and we write the proof of the new lemmas. 
	
	\smallskip
	\noindent    
	First, given two linear operators $P,B$ define,  for any $N \in \n$,
	\begin{equation*}
		Ad^N_P(B) = [Ad^{N-1}_P(B), P]\,, \qquad Ad_P^0(B) = B\,.
	\end{equation*}
	then we recall the following Lemma from \cite{DS}, (see the proof of Prop. 1.2.1).
	\begin{lemma}\label{estiAdjoint}
		There exists $\nu$ s.t. for any $P\in\Psi^1(M)$ and any $N\in\N$ there exist
		constants $C_N=C_N(P)$ with the property that for any $f\in
		C^\infty(M)$, denote by $f$ also the operator of
		multiplication by $f$, then one has
		\begin{equation*}
			\left\| Ad^N_Pf\right\|_{\cB(L^2(M))}\leq C_N\norm{f}_{N+\nu}\, .
		\end{equation*}
	\end{lemma}
	
	Given $a=(a^1,\dots,a^d)\in \Lambda$, we denote with $l^{*}(a)$ 
	the index for which $|a^l|$ is maximum. Namely,
	\begin{equation}\label{lmax}
		l^*(a) \coloneqq \underset{l=1,\dots,d}{\operatorname{argmax}}  \left|a^l \right|\,.
	\end{equation}
	
	\begin{lemma}\label{comtrickDue}
		Let $B\in \B\left(L^2(M)\right)$ and $a,b \in \Lambda$. For any $N\ge0$ and $u_1,u_2 \in \H^\infty$ we have 
		\begin{equation*}
			|\la B\Pi_a u_1,\Pi_b u_2\ra| 
			\le 
			C_N \frac{\norm{ad_{I_{l^*}}^N(B)}_{\cB(L^2(M))}}{|a - b |^N}\norm{\Pi_a u_1}_0\norm{\Pi_b u_2}_0\,.
		\end{equation*}
		with $l^* = l^*(a-b)$.
	\end{lemma}
	\begin{proof}
		First we claim that, for any $l=1,\dots,d$,
		\begin{equation}\label{claim}
			|\la B\Pi_a u_1,\Pi_b u_2\ra| 
			\le C_N \frac{\norm{ad_{I_{l}}^N(B) \Pi_a u_1}_0}{|a^l- b^l |^N}\norm{\Pi_b u_2}_0\,.
		\end{equation}
		For N = 1, recalling that the actions $I_j$ are selfadjoint, we have
		\begin{equation*}
			\begin{aligned}
				\la Ad_{I_l}(B)\Pi_a u_1,\Pi_b u_2 \ra 
				&= \left|\la BI_l\Pi_a u_1,\Pi_b u_2 \ra - \la I_l B\Pi_a u_1,\Pi_b u_2\ra\right| 
				\\& 
				=  \left|\la B I_l\Pi_a u_1,\Pi_b u_2 \ra - \la  B\Pi_a u_1,I_l\Pi_b u_2\ra\right|  
				\\&=  
				\left|a^l\la B\Pi_a u_1,\Pi_b u_2 \ra - b^l\la B\Pi_a u_1,\Pi_b u_2\ra\right| 
				\\&= 
				\left|a^l-b^l\right| \left|\la B\Pi_a u_1,\Pi_b u_2 \ra\right|.
			\end{aligned}
		\end{equation*}
		Then, by induction on N, and arguing as in the case $N=1$ with $B$ replaced 
		by $ad^{N-1}_{I_l}(B)$, we get
		\begin{equation*}
			\begin{aligned}
				&\la Ad^N_{I_l}(B)\Pi_a u_1,\Pi_b u_2 \ra = \la Ad_{I_l}\left(Ad^{N-1}_{I_l}(B)\right)\Pi_a u_1,\Pi_b u_2 \ra 
				\\&
				\qquad = |a^l-b^l|\la Ad^{N-1}_{I_l}(B)\Pi_a u_1,\Pi_b u_2 \ra = |a^l-b^l|^N \la B \Pi_a u_1,\Pi_b u_2 \ra\,.
			\end{aligned}
		\end{equation*}
		This means that
		\begin{equation*}
			|\la B\Pi_a u_1,\Pi_b u_2\ra| = \frac{|\la ad_{I_l}^N(B)\Pi_a u_1, \Pi_b u_2\ra|}{|a^l - b^l|^N}\,.
		\end{equation*}
		Then claim \eqref{claim} holds by Cauchy-Schwartz inequality.
		Choosing $l=l^*(a-b)$ defined in \eqref{lmax} and exploiting the trivial inequality
		\begin{equation*}
			|a|  \le \sqrt{d} \left|a^{l^*}\right|\,,
		\end{equation*}
		we get the thesis.
	\end{proof}
	
	\begin{proof}[Proof of Theorem \ref{eigenDue}]
		We first prove the theorem for $k=3$.
		Without loss of generality, we can assume $\len a\ge\len b\ge \len c$. We distinguish two cases.
		
		\smallskip
		\noindent
		\textbf{Case 1} If  $\len c\le 2 |a-b|$, 
		we apply Lemma \ref{estiAdjoint} with $f = g\Pi_c u_3 $ and $P = I_{l^*(a-b)}$ and we get
		\begin{equation*}
			\norm{Ad_{I_{l^*}}^N \, \left(g\Pi_c
				u_3\right)}_{\B(L^2(M))} \lesssim  \norm{g\Pi_c
				u_3}_{N+\nu}\lesssim \left\|g\right\|_{N+\nu}  \norm{\Pi_c
				u_3}_{N+\nu} \lesssim  \len{c}^{N+\nu}
			\norm{\Pi_c u_3}_0\, ,
		\end{equation*}
		where the nonwritten constant depends on $g$.
		In order to get the thesis we apply Lemma \ref{comtrickDue}. We have
		\begin{equation*}
			\begin{aligned}
				\left|\int g \Pi_a u_1\Pi_b u_2\Pi_c u_3 \right| &\le 
				\frac{\norm{Ad_{I_{l^{*}}}^N\, \left(g\Pi_c u_3\right)}_{{\B(L^2(M))}}}{|a-b|^N} 
				\norm{\Pi_a u_1}_0 \norm{\Pi_b u_2}_0 
				\\&
				\lesssim \frac{\len{c}^{N+\nu}}{|a-b|^{N}} \norm{\Pi_a u_1}_0\norm{\Pi_b u_2}_0\norm{\Pi_c u_3}_0\,,
			\end{aligned}
		\end{equation*}
		where the non written constant depends on $g$, on $I_{l^{*}}$ and on
		$N$. 
		Finally, we observe that
		\begin{equation*}
			\frac{\len{c}^{N+\nu}}{|a-b|^{N}} \le 3^N \frac{\len{c}^{N+\nu}}{(\len{c}+|a-b|)^{N}}\,,
		\end{equation*}
		since $\len{c} \le 2 |a-b|$. 
		
		\smallskip
		\noindent
		\textbf{Case 2} If $ \len c > 2 |a-b|$, then
		\begin{equation}\label{unopiccolo}
			\frac{2}{3}\le \frac{\mu(a,b,c)}{S(a,b,c)} \le 1\,. 
		\end{equation}
		In this case, we take profit of the Sobolev embedding
		$H^{d/2+}\hookrightarrow  L^\infty$ and we get
		\begin{equation*}
			\begin{aligned}
				\left|\int g \Pi_a u_1 \Pi_b u_2 \Pi_c u_3 \right| &\le \norm{ g\Pi_c u_3}_{{L}^\infty} \norm{\Pi_a u_1}_0
				\norm{\Pi_b u_2}_0 
				\\&\le 
				C_{s_0} \left\|g\right\|_{L^\infty}  \norm{\Pi_c u_3}_{s_0} \norm{\Pi_a u_1}_0
				\norm{\Pi_b u_2}_0 
				\\
				&
				\lesssim 
				\len{c}^{s_0}\norm{\Pi_c u_3}_{0} \norm{\Pi_a u_1}_0
				\norm{\Pi_b u_2}_0 
				\\&\stackrel{\eqref{unopiccolo}}{\lesssim} 
				\frac{\mu(a,b,c)^{N+\nu}}{S(a,b,c)^N}  
				\norm{\Pi_a u_1}_{0} \norm{\Pi_b u_2}_0
				\norm{\Pi_c u_3}_0 \,.
			\end{aligned}
		\end{equation*}
		with $\nu \geq s_0$.
		The general case follows from Lemma \ref{Poissonproof}. 
	\end{proof}
	Finally, in order to prove Theorem \ref{lo.generale}, we need also the
	following lemma.
	\begin{lemma}\label{productNorm3}
		Let $P$ be a polynomial with localized coefficients, then also 
		\begin{equation}
			\label{con la norma}
			Q(u):=P(u) \int_M u\bar u dx
		\end{equation}
		has localized coefficients.
	\end{lemma}	
	\begin{proof}
		Let $r$ be the degree of $P$, then one has
		\begin{equation}
			\label{Q.da}
			|\widetilde Q(\Pi_{a_1}u_1,\Pi_{a_2}u_2,\Pi_{b_1}u_3,....,\Pi_{b_r}u_{r+2})|=\delta_{a_1,a_2}
			\left|\widetilde P(\Pi_{b_1}u_3,....,\Pi_{b_r}u_{r+2}) \right|
			\left\|\Pi_{a_1}u_1\right\| \left\|\Pi_{a_2}u_2\right\|\ . 
		\end{equation}
		Therefore, in order to get the thesis it is enough to show that 
		\begin{gather}\label{tesi}
			\delta_{a_1,a_2}
			\frac{\mu(\bb)^{N+\nu}}{S(\bb)^{N}} \lesssim
			\frac{\mu(a_1,a_2,\bb)^{N+\nu}}{S(a_1,a_2,\bb)^{N}}\ .
		\end{gather}
		For simplicity, we will also denote $a\coloneqq
		a_1 = a_2$ and we consider, in full generality that
		$\len{b_1}\ge\len{b_2}\ge \len{b_3}$ are the three largest indexes
		among $b_1,\dots,b_{r}$.

		\noindent\textbf{Case 1.}  If $\len{a}\ge
		\len{b_1}$ we have the trivial estimate
		\begin{equation*}
			\begin{aligned}
				&\frac{\mu(\bb)^{N+\nu}}{S(\bb)^{N}} =
				\frac{\len{b_3}^{\nu+N}}{(|b_1-b_2| + \len{ b_3})^N} 
				\le \len{b_3}^\nu  
				\\
				&\qquad \qquad \le \len{b_1}^\nu = \frac{\len{b_1}^{\nu+N}}{(|a_1 - a_2| + \len{b_1})^N} 
				=
				\frac{\mu(a_1,a_2,\bb)^{N+\nu}}{S(a_1,a_2,\bb)^{N}}\,,  
			\end{aligned}
		\end{equation*}
		since $|a_1 - a_2| = 0$.\\
		\textbf{Case 2.} If $\len{b_1} >  \len{a} > \len{b_2}$, we need to distinguish two cases. \\
		\textbf{Case 2.i.} Consider first the case $\len{b_1}
		> K_1 \len{b_2}$, with $K_1>0$ so large that
		\begin{gather}
			|b_1-b_2| \ge K_2 \len{b_1}
		\end{gather}
		for a constant $0<K_2<1 $; the existence of such constants is established in \ref{newDiago.rem} .
		Then using $|b_1-a| \lesssim \len{b_1} + \len{a} \lesssim \len{b_1}$, we estimate
		$$
		\frac{\len{b_3}^{\nu+N}}{(|b_1-b_2| + \len{b_3})^N} \le 
		\frac{\len{b_3}^{\nu+N}}{(K_2\len{b_1} + \len{b_3})^N} 
		\lesssim
		\frac{\len{b_3}^{\nu+N}}{(|b_1-a| + \len{b_3})^N}\,.
		$$
		Since the function $f(x) = \frac{x^{N+\nu}}{(K + x)^N}$ is
		increasing for any $N,\nu>0$, $K\geq0$ and $x\ge0$, the above quantity
		is bounded, up to a constant, by
		$$
		\begin{aligned}
			\frac{\len{a}^{\nu+N}}{(|b_1-a| + \len{a})^N} = \frac{\mu(a_1,a_2,\bb)^{N+\nu}}{S(a_1,a_2,\bb)^{N}}\,.
		\end{aligned}
		$$
		\textbf{Case 2.ii.} If $\len{b_1} \le K_1 \len{b_2}$ we observe that
		\begin{equation*}
			K_1 C \len{b_2} \ge C \len{b_1}\ge |b_1-a|
		\end{equation*}
		and we estimate
		\begin{equation}\label{abab}
			\begin{aligned}
				\frac{\len{b_3}^{\nu+N}}{(|b_1-b_2| + \len{b_3})^N} &\le \len{b_2}^\nu 
				=  
				(K_1 C+1)^N \frac{\len{b_2}^{\nu+N}}{(K_1C\len{b_2} + \len{b_2})^N} 
				\\&\le 
				(K_1 C+1)^N \frac{\len{b_2}^{\nu+N}}{(|b_1-a| + \len{b_2})^N}\,.
			\end{aligned}
		\end{equation}
		Using again the monotonicity of the function $f(x) = \frac{x^{N+\nu}}{(K+x)^{N}}$, we get
		\begin{equation*}
			\begin{aligned}
				\eqref{abab}&\leq (K_1 C+1)^N \frac{\len{a}^{\nu+N}}{(|b_1-a| + \len{a})^N}  
				\\ & =  (K_1C+1)^N \frac{\mu(a_1,a_2,\bb)^{N+\nu}}{S(a_1,a_2,\bb)^{N}}\,.
			\end{aligned}
		\end{equation*}

		\noindent		\textbf{Case 3.} If $\len{b_2}\ge\len{a} > \len{b_3}$ we get 
		\begin{equation*}
			\begin{aligned}
				\frac{\mu(\bb)^{N+\nu}}{S(\bb)^{N}} 
				&= \frac{\len{b_3}^{\nu+N}}{(|b_1-b_2| + \len{b_3})^N} 
				\\&
				\le \frac{\len{a}^{\nu+N}}{(|b_1 - b_2| + \len{a})^N} 
				=
				\frac{\mu(a_1,a_2,\bb)^{N+\nu}}{S(a_1,a_2,\bb)^{N}}\,. 
			\end{aligned}
		\end{equation*}
		\textbf{Case 4.} If $\len{a}\le \len{b_3}$ the estimates \eqref{tesi.wave} is obvious since it does not involve $a$.

		This concludes the proof.
	\end{proof}
	\begin{proof}[Proof of Theorem \ref{lo.generale}]
		The proof follows directly from Lemma \ref{productNorm3} and the definition of function with localized coefficients \ref{locafunc}, since the Taylor expansion of a functional fulfilling Hypothesis \ref{Nonline} is the sum of terms $P_m$ of the form \eqref{con la norma} 
	\end{proof}
	
	\section{The clusterization of the lattice. 
		Proof of Theorem \ref{QNtoNoi}.}\label{QNtoNoiSection}
	The following result follows from the construction in \cite{QN}.
	\begin{proposition}
		There exists a partition of the lattice $\Lambda$
		$$
		\Lambda= \bigcup_{\alpha \in \N} \Omega_\alpha
		$$
		satisfying Hypothesis 3 (\ref{bourgain.abstract}), namely: 
		\begin{itemize}
			\item[(i)] There exists $C>0$ such that $\forall \alpha \in \N$
			\begin{equation*}
				\sup_{a \in \Omega_\alpha} | a | \leq C \inf_{a \in \Omega_\alpha} | a| \,,
			\end{equation*}
			\item[(ii)] There exist $\mu>0$ and $C_{\mu}>0$ such that, if $a \in \Omega_\alpha$ and $b \in \Omega_\beta$ with $\alpha \neq \beta$, then
			\begin{equation}\label{separati}
				|a - b| + |\omega_a - \omega_b| \geq C_{\mu} \left(|a|^\mu + |b|^{\mu}\right)\,.
			\end{equation}
		\end{itemize}
	\end{proposition}
	\begin{proof}
		We consider the partition
		\[
		\Lambda = \bigcup_{s = 1}^d \bigcup_{\substack{M \textrm{ s.t.} \\ \textrm{dim}(M) = s}}
		\bigcup_{j \in \mathcal{J}_M} E^{(s)}_{M, j}\,,
		\]
		as in Definition 8.26 of \cite{QN}. 
		We take $\mathtt{R}>0$ large enough such that Theorem $8.28$  
		of \cite{QN} is satisfied and we modify the partition as follows:
		\begin{gather*}
			\widetilde{E}_\mathtt{R} := \bigcup\{ E^{(s)}_{M,j}\ |\ \exists a \in E^{(s)}_{M,j} \quad \textrm{s.t.} |a| \leq \mathtt{R}\}\\
			\widetilde{E}^{(s)}_{M,j} = E^{(s)}_{M,j} \quad \text{if}\quad  |a| \geq \mathtt{R} \quad \forall a \in E^{(s)}_{M,j}\,.
		\end{gather*}
		We are going to show that such a new partition satisfies properties (i) and (ii).\\
		{\sc Proof of Property $(i)$.} By Theorem 8.28, the blocks $E^{(s)}_{M,j}$ are dyadic, thus (i) holds for the blocks $\widetilde{E}^{(s)}_{M,j}$ with $C = 2$. Let us consider $\widetilde{E}_{\mathtt{R}}$: for any $a \in \widetilde{E}_{\mathtt{R}}$ one has
		$$
		|a| \leq 2 \min_{b \in E^{(s)}_{M,j}} |b| \leq 2 \mathtt{R}\,,
		$$
		where $E^{(s)}_{M,j}$ is the block in the original partition $\{E^{(s)}_{M,j}\}_{s, M, j}$ to which the lattice point $a$ belongs. 
		This immediately implies
		$$
		\sup_{a \in \widetilde{E}_{\mathtt{R}}} | a | \leq 2 \mathtt{R} \leq 2 \gamma^{-1} \mathtt{R} \min_{a \in \widetilde{E}_{\mathtt{R}}} | a |\,,
		$$
		if $\gamma := \min \{|b| \ |\ b \in \Lambda \setminus\{0\}\}$. Then (i) holds with $C := \max\{2, 2 \gamma^{-1} \mathtt{R}\}$ for all the elements of the new partition.\\
		{\sc Proof of Property $(ii)$.} 
		We show that there exist $\mu$ and $C_\mu$ such that, if $a, b \in \Lambda$ are such that \eqref{separati} is violated, namely
		\begin{equation}\label{non.separati}
			|a - b| + |\omega_a - \omega_b| < C_\mu \left(|a| + |b|\right)^\mu\,,
		\end{equation}
		then $a$ and $b$ must belong to the same block. 
		If both $|a| \leq \mathtt{R}$ and $|b| \leq \mathtt{R}$, then $a$ and $b$ belong to the same block since they are both contained into $\widetilde{E}_{\mathtt{R}}$. Thus let us suppose that $|a| \geq \mathtt{R}$.
			First of all, we observe that there exists $c_{\mu}>0$ such that, if $C_\mu \leq c_\mu$, then \eqref{non.separati} implies
			\begin{equation}\label{ma.va}
				|a-b| \leq |a|^\mu\,.
			\end{equation}
			Indeed, one has $|b|^\mu \leq 2^\mu (|a|^\mu + |b-a|^\mu)$, which implies
			$$
			|a-b| \leq C_\mu (1 + 2^\mu)|a|^\mu + C_\mu 2^\mu |a-b|^\mu\,.
			$$
			Now, if $|a-b|\leq (2^{\mu + 1} C_\mu)^{\frac{1}{\mu-1}}$, one has
			\begin{equation}\label{se.piccolo}
				|a-b| \leq (2^{\mu + 1} C_\mu)^{\frac{1}{\mu-1}} \leq (2^{\mu + 1} C_\mu)^{\frac{1}{\mu-1}} |a|^\mu\,,
			\end{equation}
			while if $|a-b| \geq (2^{\mu + 1} C_\mu)^{\frac{1}{\mu-1}}$ one has
			\begin{equation}\label{se.grande}
				\frac{|a-b|}{2} \leq |a-b| - C_\mu 2^\mu |a-b|^\mu \leq C_\mu (1 + 2^\mu)|a|^\mu\,.
			\end{equation}
			Combining \eqref{se.piccolo} and \eqref{se.grande} 
			and reducing the value of $C_\mu$, one gets \eqref{ma.va}.\\
			We proceed to show that, if \eqref{non.separati} holds with 
			$C_\mu$ is small enough, one also has
			\begin{equation}\label{eh.si}
				|\upsilon (a)\cdot (b-a) | \leq |a|^{\delta} |a-b|\,,
			\end{equation}
			where $\upsilon(a) := \partial_a h_0(a)$ and $h_0$ is the function whose existence is assumed in \ref{glob.int.quant}.
			Indeed, let us suppose by contradiction that \eqref{eh.si} 
			does not hold true: then, since $|a-b| \leq |a|^{\mu},$ $\forall t \in [0, 1]$ one has
			$$
			\begin{aligned}
				|\upsilon (a + t b)\cdot (b-a) | &\geq |\upsilon(a) \cdot (b-a)| - |\upsilon(a+ t b) - \upsilon(a)| |b-a|
				\\
				&\geq |\upsilon(a) \cdot (b-a)| - \sup_{\substack{a' \textrm{ s.t.} \\ |a-a'| \leq |b-a|}} |\partial_a^2 h_0(a')| |b-a|^{2}\\
				&\geq |a|^\delta |a-b|- C_{h_0} \sup_{\substack{a' \textrm{ s.t.} \\ |a-a'| \leq |b-a|}} |a'|^{M-2} |b-a|^2\,,
			\end{aligned}
			$$
			where in the last passage we have used the fact that $h_0$ 
			is a homogeneous function of degree $M$. Then $\forall t$ one has
			\begin{equation}\label{fine.giornata}
				\begin{aligned}
					|\upsilon (a + t b)\cdot (b-a) | &\geq |a|^\delta |b-a| - C_{h_0} (|a| + |b-a|)^{M-2} |b-a|^2\\
					& \geq |a|^\delta - 2 C_{h_0} |a|^{M-2 + 2\mu} \geq \frac{|a|^{\delta}}{2}\,,
				\end{aligned}
			\end{equation}
			provided $\mathtt{R}$ is large enough. But then on the one hand \eqref{fine.giornata} gives
			\begin{equation}\label{freq.below}
				|\omega_a - \omega_b|= \left|\int_{0}^1 \upsilon(a + t (b-a))\cdot (b-a)\, d t  \right| \geq \frac{|a|^{\delta}}{2}\,,
			\end{equation}
			while on the other hand \eqref{non.separati} gives
			\begin{equation}\label{freq.above}
				\begin{aligned}
					|\omega_a - \omega_b| \leq C_\mu |a|^\mu + C_\mu |b|^\mu 
					&\leq C^\mu (2^\mu + 1) |a|^\mu + 2^\mu C_\mu |b-a|^\mu 
					\\&\leq 2 C^\mu (2^\mu + 1) |a|^\mu\,.
				\end{aligned}
			\end{equation}
			Since estimates \eqref{freq.above} and \eqref{freq.below} are clearly not compatible for $\mathtt{R}$ large enough, due to the fact that $\delta > \mu$, one gets a contradiction. Thus we conclude that \eqref{eh.si} holds true and, combining estimates \eqref{ma.va} and \eqref{eh.si}, that
			\begin{equation}\label{risonanti}
				|\upsilon (a)\cdot (b-a) | \leq |a|^{\delta} |a-b|\,, \quad |a-b| \leq |a|^\mu\,, \quad |a| \geq \mathtt{R}\,,
			\end{equation}
			namely $a$ is resonant with $b-a$ according to Definition 6.3 of \cite{QN}. Then the proof follows exactly with the same passages used to prove Lemma 8.39 of \cite{QN}. 
		\end{proof}
		
		\bibliographystyle{amsalpha}
		\bibliography{refs2}

	\end{document}